\documentclass{amsart}
\usepackage{algtop}
\usepackage[all, cmtip]{xy}
\usepackage{caption}
\usepackage{graphics}
\usepackage{mathtools}
\usepackage{enumitem}
\usepackage[T1]{fontenc}
\usepackage[utf8]{inputenc}
\usepackage{url}

\usepackage{tikz}
\usetikzlibrary{arrows.meta}
\usetikzlibrary{calc}

\makeatletter
	\let\c@equation\c@theorem
\makeatother
\numberwithin{equation}{section}

\newcommand{\makebibliography}{
    \newpage
    \bibliographystyle{plain}
    \bibliography{main.bib}
}

\newcommand{\AF}{\mathrm{AF}}
\newcommand{\ESS}[1]{\prescript{#1\!}{}{E}}
\newcommand{\ZESS}[1]{\prescript{#1\!}{}{Z}}
\newcommand{\BESS}[1]{\prescript{#1\!}{}{B}}
\newcommand{\SynSS}{\prescript{\mathrm{syn}\!}{}{E}}
\newcommand{\HF}{\mathrm{H}\bF_2}

\definecolor{lightgray}{RGB}{200,200,200}
\definecolor{llgray}{RGB}{235,235,235}

\newcommand{\tikzgrid}[4]{\foreach \x in {#1,...,#2} {
        \draw[black!20] (\x-0.5, #3-0.5) -- (\x-0.5, #4-0.5);
    }
    \foreach \y in {#3,...,#4} {
        \draw[black!20] (#1-0.5, \y-0.5) -- (#2-0.5, \y-0.5);
    }
}

\newcommand{\drawarrow}{\draw[-{Stealth[length=1.5mm]}, shorten >=(0.08cm)]}

\title{On the Last Kervaire Invariant Problem}
\author{Weinan Lin, Guozhen Wang, and Zhouli Xu}
\address{Shanghai Center for Mathematical Sciences, Fudan University, Shanghai, China, 200433}
\email{linweinan@fudan.edu.cn}
\address{Shanghai Center for Mathematical Sciences, Fudan University, Shanghai, China, 200433}
\email{wangguozhen@fudan.edu.cn}
\address{UCLA Department of Mathematics, Los Angeles, CA 90095-1555, USA}
\email{xuzhouli@ucla.edu}
\dedicatory{Dedicated to Mark Mahowald}
\date{}

\begin{document}
\begin{abstract}
We prove that the element $h_6^2$ is a permanent cycle in the Adams spectral sequence. As a result, we establish the existence of smooth framed manifolds with Kervaire invariant one in dimension 126, thereby resolving the final case of the Kervaire invariant problem.

Combining this result with the theorems of Browder, Mahowald--Tangora, Barratt--Jones--Mahowald, and Hill--Hopkins--Ravenel, we conclude that smooth framed manifolds with Kervaire invariant one exist in and only in dimensions  $2, 6, 14, 30, 62$, and $126$.
% Keywords: \Groebner{} basis, Ext groups, Steenrod algebra
\end{abstract}

\maketitle
% \tableofcontents
% \blfootnote{The author is supported by China Postdoctoral Science Foundation 2021TQ0015 and the Fundamental Research Funds for the Central Universities, Peking University}

\section{Introduction}

%[Recollection to $HF_2$-synthetic spectra.]

%[Relationship between $En$-pages and $E_\infty(\nu X/\lambda^{n-1})$. We talk about maps between $E_n$-pages and extensions.]

%[Rigidity, total differential]

%[Examples and pictures]

%[Choices of letters]

%[https://arxiv.org/pdf/2202.11305]

%\huge

For a smooth framed manifold in dimension $4k+2$, the Kervaire invariant, taking values in $\mathbb{F}_2$, determines whether the manifold could be converted into a homotopy sphere via surgery -- it takes value 0 if the manifold can be converted to a homotopy sphere and 1 otherwise. Formally speaking, the Kervaire invariant is defined as the Arf invariant of the quadratic refinement of the intersection pairing in the cohomology of the manifold with $\mathbb{F}_2$ coefficients. Using this invariant, Kervaire \cite{Kervaire} discovered a PL-manifold in dimension 10 that does not admit any smooth structure. 

The Kervaire invariant problem seeks to identify the dimensions in which there exist  framed smooth manifolds with Kervaire invariant one. In these dimensions, exactly half of the cobordism classes of framed manifolds have Kervaire invariant one, and the other half have Kervaire invariant 0. The Kervaire invariant problem is closely related to many problems in differential topology, especially Kervaire--Milnor's classification theorem \cite{KervaireMilnor} on exotic smooth structures on spheres. 

%Browder \cite{Browder} proves that there exists a framed manifold in dimension $n$ if and only if $n=2^{j+1}-2$ and that the element $h_j^2$ in the Adams spectral sequence survives to the $E_\infty$-page, converging to a nonzero homotopy class in the stable homotopy groups of spheres in dimension $n=2^{j+1}-2$.

In this paper we prove the following Theorem~\ref{thm:main}.

\begin{theorem} \label{thm:main}
    There exist framed manifolds of Kervaire invariant one in dimension 126.
\end{theorem}

Together with previous results by Browder \cite{Browder}, Mahowald--Tangora \cite{MahowaldTangora}, Barratt--Jones--Mahowald \cite{BJMtheta5}, and Hill--Hopkins--Ravenel \cite{HHR}, this is the last case of the Kervaire invariant problem.

\begin{corollary}
  The dimensions that there exist framed manifolds of Kervaire invariant one are 2, 6, 14, 30, 62, and 126.  
\end{corollary}

In dimensions 2, 6, and 14, the product of spheres $S^1 \times S^1, \ S^3 \times S^3, S^7 \times S^7$ can be framed to have Kervarie invariant one.
In dimension 30, an explicit framed manifold of Kervarie invariant one was constructed by J.Jones in \cite{Jones30}. For dimensions 62 and 126, we would like to comment that no explicit manifold of Kervarie invariant one was known, although $50\%$ of all framed manifolds in these two dimensions have Kervarie invariant one.

Our Theorem \ref{thm:main} is a consequence of Browder's theorem (Theorem~\ref{thm:browder}) and the following main theorem of this paper, Theorem~\ref{thm:h62}. 

\begin{theorem}[Browder \cite{Browder}] \label{thm:browder}
   There exist framed manifolds in dimension $n$ of Kervaire invariant one if and only if
   \begin{enumerate}
       \item $n=2^{j+1}-2$, and
       \item the element $h_j^2$ in the Adams spectral sequence survives to the $E_\infty$-page.
   \end{enumerate}
\end{theorem}

The Adams spectral sequence \cite{Adams58} converges to the stable homotopy groups of spheres, which, by Pontryagin's theorem, correspond to framed manifolds up to framed cobordism. When the element $h_j^2$ survives (see below for an introduction to the Adams spectral sequence and the element $h_j^2$), the homotopy classes it detects are denoted by $\theta_j \in \pi_{2^{j+1}-2}$. Browder showed that these classes correspond to framed manifolds with Kervaire invariant one. 

\begin{theorem}[Theorem~\ref{thm:126survives}] \label{thm:h62}
    The element $h_6^2$ survives to the $E_\infty$-page in the Adams spectral sequence.
\end{theorem}

In addition to Theorem \ref{thm:main}, further implications of Theorem~\ref{thm:h62}, particularly in manifold topology and unstable homotopy theory, can be found in \cite{MillerHHR, MahowaldRemark, HHR, BJMinduction}, among others.

Recall that the 2-primary Adams spectral sequence has the following form.
$$E_2^{s,t} = \Ext^{s, t}_A(\mathbb{F}_2, \mathbb{F}_2) \Longrightarrow \pi_{t-s} {S^0},$$
$$d_r: E_r^{s,t} \longrightarrow E_r^{s+r, t+r-1}.$$
Here the $E_2$-page is the cohomology of the mod 2 Steenrod algebra A, $S^0$ is the 2-completed sphere spectrum, and the spectral sequence converges to the 2-completed stable homotopy groups of spheres. In general, for a spectrum $X$, its $\HF$-Adams spectral sequence is denoted by $E_r^{*,*}(X)$, converging to the 2-completed homotopy groups of $X$.

For the sphere, Adams \cite{Adams58} computed the Adams $1$-line: 
$$\Ext^{1,t}_A(\mathbb{F}_2, \mathbb{F}_2) = \begin{cases}
 	\mathbb{F}_2, \ \text{if} \ t = 2^j \ \text{for some} \ j \geq 0, \\
 	0, \ \text{otherwise}.
\end{cases}
$$
Denote by $h_j$ the generator of $\mathbb{F}_2$ in the bidegree $(s, t)=(1, 2^j)$. Adams \cite{Adams60} proved that the element $h_j$ survives in the Adams spectral sequence if and only if $j \geq 3$, resolving the famous Hopf invariant problem. In fact, Adams proved that
$$d_2(h_j) = h_0h_{j-1}^2 \neq 0, \ \text{for} \ j \geq 4.$$

The Kervaire invariant element $h_j^2$ lives on the Adams 2-line, which is generated by elements of the form $\{h_ih_j\}$ that are subject to the relations $h_ih_{i+1} = 0$. Notably, $h_6^2$ was the last unknown element on the Adams 2-line whose survival in the Adams spectral sequence remained uncertain prior to our work. As a corollary of Theorem~\ref{thm:h62}, we have:

\begin{corollary} \label{cor:2line}
    On the Adams 2-line, the only non-trivial elements that survive in the Adams spectral sequence are:
    $$h_0h_2, \ h_0h_3, \ h_2h_4, \ h_1h_j \ (j \geq 3), \ h_j^2 \ (j \leq 6).$$
\end{corollary}

\begin{proof}
    From Adams's Hopf invariant one differentials, Lin's computations of the Adams 4-line \cite{LinExt4}, and the Leibniz rule, we know that the only elements that survive to the Adams $E_3$-page are
    $$h_0h_2, \ h_0h_3, \ h_2h_4, \ h_2h_5, \ h_3h_5, \ h_3h_6, \ h_1h_j \ (j \geq 3), \ h_j^2 \ (j \leq 6).$$
May \cite{Maythesis} proved that the first three elements and $h_j^2$ for $j \leq 3$ survive. Mahowald--Tangora \cite{MahowaldTangora} proved that $h_2h_5$ supports a nonzero $d_3$-differential, $h_3h_5$ supports a nonzero $d_4$-differential, and $h_4^2$ survives. Isaksen--Wang--Xu \cite{IWX} proved that $h_3h_6$ supports a nonzero $d_3$-differential. Mahowald \cite{Mahowaldetaj} proved that $h_1h_j$ for $j \geq 3$ survive. Barratt--Jones--Mahowald \cite{BJMtheta5} proved that $h_5^2$ survives (see \cite{Xu, IWX} for alternative proofs). Hill--Hopkins--Ravenel \cite{HHR} proved that $h_j^2$ supports nonzero differentials for $j \geq 7$ (the targets of these differentials are still unknown). Finally, by Theorem~\ref{thm:h62}, $h_6^2$ survives.   
\end{proof}

Before Hill--Hopkins--Ravenel's proof on the non-existence of $\theta_j$ for $j \geq 7$, Barratt--Jones--Mahowald \cite{BJMinduction} had an inductive argument trying to establish the existence of all $\theta_j$: They proved that if there exists a $\theta_j$ satisfying $2\cdot \theta_j = 0, \ \theta_j^2 = 0$, then there exists a $\theta_{j+1}$ satisfying $2\cdot \theta_{j+1} = 0$. In particular, the problem of whether a $\theta_{j+1}$ of order 2 exists is known as the strong Kervaire invariant problem. It was known that $2\cdot \theta_j = 0$ for $j \leq 5$ (see \cite{MahowaldTangora} for $j=4$ and \cite{Xu, IWX} for $j=5$). Although we prove that $\theta_6$ exists in this paper, the following questions are still open:

\begin{question} \label{que:2theta6}
    Does there exist a $\theta_6$ that has order 2?
\end{question}

\begin{question} \label{que:theta5sq}
    Does there exist a $\theta_5$ such that $\theta_5^2 = 0$?
\end{question}

By Barratt--Jones--Mahowald's theorem \cite{BJMinduction}, if the answer to Question~\ref{que:theta5sq} is positive, then it would imply our Theorem~\ref{thm:h62} and a positive answer to Question~\ref{que:2theta6}. We plan to study these questions in a future project.\\

As suggested by the proof of Corollary~\ref{cor:2line}, computation of differentials in the Adams spectral sequence has a long history. See Section 2 of \cite{WX61} for a brief summary at all primes and computations at the prime 2 up to around 60-stem. For more recent computations up to around 90-stem, see \cite{IWX, IWXsurvey, WXems, WXicm}.

Many methods were introduced to compute differentials in the Adams spectral sequence. We highlight some of the major methods as follows.
\begin{enumerate}
    \item Multiplicative structure of the Adams $E_2$-page and the Leibniz rule. \\
    
     \noindent By computing the Adams $E_2$-page with its multiplicative structure via the May spectral sequence, and comparing with Toda's unstable computations, May \cite{Maythesis} computed all differentials up to the 28-stem.\\
    
    \item Higher product structure -- interactions between Massey products and Toda brackets through Moss's theorem \cite{Moss}.\\

    \item The Mahowald trick \cite{MahowaldTangora}, which translates between differentials and extension problems.\\

    \noindent By using both Moss's theorem and the Mahowald trick, Barratt--Mahowald--Tangora \cite{MahowaldTangora, BarrattMahowaldTangora} computed all but one differentials up to the 45-stem; The remaining one was computed by Bruner \cite{Brunerdiff} using power operations.\\

    \item Comparison with the motivic Adams spectral sequence via the Betti realization functor \cite{DuggerIsaksen}.\\

    \noindent By comparison via the Betti realization functor, Isaksen \cite{Isaksen} gave rigorous arguments for all but one differentials in both the motivic and \emph{classical} Adams spectral sequence up to the 59-stem; The remaining one was computed by Xu \cite{IsaksenXu} using the higher Leibniz rule for motivic Massey products. \\

    \item Wang--Xu's $\mathbb{R}P^\infty$-method \cite{WX61}.\\

    \noindent Using Lin's algebraic Kahn--Priddy theorem \cite{LinalgKP}, Wang--Xu introduced a technique to prove Adams differentials inductively using differentials in subquotients of $\mathbb{R}P^\infty$, and completed computations of Adams differentials up to the 61-stem. (See also \cite{WangXu51ext} for application of this method on extension problems.) \\

    \noindent There is a nice geometric application of these differentials in stems 60 and 61. The sphere $S^{61}$ has a unique smooth structure and it is the last odd dimensional case: The only ones are $S^1, \ S^3, \ S^5$ and $S^{61}$.\\

    \item Gheorghe--Isaksen--Wang--Xu's motivic cofiber of tau method \cite{GWX, IWX}. \\

    \noindent By identifying the algebraicity of the special fiber of a motivic deformation, Gheorghe--Wang--Xu \cite{GWX} proved that the motivic Adams spectral sequence of the cofiber of tau is isomorphic to the algebraic Novikov spectral sequence for BP$_*$, which can be completely computed in a large range by Wang's program \cite{Wang}. Then differentials in the classical Adams spectral sequence follow from naturality.\\ 
    
    \noindent Using this method, together with some others, Isaksen--Wang--Xu \cite{IWX} computed Adams differentials up to the 90-stem, with only a few exceptions. \\

    \item $\HF$-synthetic/filtered spectra method \cite{Pst, Burklund, BurklundXu, BurklundIsaksenXu}.\\

    \noindent The homotopy ring of the $\HF$-synthetic sphere \cite{Pst, BHS} can be viewed as a tool to encode homotopical information of the classical Adams spectral sequence, in analog to the relation between the homotopy ring of the $\mathbb{C}$-motivic sphere \cite{Isaksen} and the Adams--Novikov spectral sequence. \\

    \noindent By studying the $\HF$-synthetic sphere, Burklund--Isaksen--Xu \cite{BurklundIsaksenXu} proved a few Adams differentials up to the 90-stem that were left out by Isaksen--Wang--Xu \cite{IWX}.\\

\end{enumerate}

Our strategy for proving our main Theorem \ref{thm:h62} that $h_6^2$ survives is three-fold. 

\begin{itemize}
    \item Lin's Program.

    \noindent By establishing the theory of noncommutative Gr\"obner bases for Steenrod algebras \cite{LinGrobner}, Lin made computer programs that compute Ext groups in a very effective way. We use Lin's program to compute the Adams $E_2$-pages for a vast collection of finite spectra and the maps between them. For a detailed account of the range of Adams $E_2$ data computed, see \cite{LWXMachine}.\\
    
    \noindent Building on the theory of secondary Steenrod algebras \cite{BJ, Nassau, Chua}, we also leverage Lin’s program to compute Adams $d_2$-for certain finite spectra. Details on the $d_2$-differentials we have computed can also be found in \cite{LWXMachine}.\\
    
    \noindent Additionally, Lin’s program facilitates the propagation of differentials and extensions by employing the Leibniz rule and the naturality of Adams spectral sequences and extension spectral sequences (Definition \ref{def:ess}). \\

    \item The Generalized Leibniz Rule and the Generalized Mahowald Trick. \\

    \noindent The use of $\HF$-synthetic/filtered spectra allows us to discuss elements on the Adams $E_n$-page as classes in the homotopy groups of certain synthetic spectra \cite{Pst, BurklundXu}. In particular, naturality on synthetic homotopy groups allows us to discuss jump-of-filtration phenomena on any Adams $E_n$-page in a rigorous way. \\
    
    \noindent Using $\HF$-synthetic spectra, we develop the Generalized Leibniz Rule (Theorem~\ref{thm:e73f481e}) and the Generalized Mahowald Trick (Theorem~\ref{thm:158d451a}), which can be viewed as generalizations of the classical Mahowald's trick and geometric boundary theorem studies by Ravenel, Behrens and others \cite{RavenelGreenbook, Behrens, Ma, BurklundCookware}. We then use them to further propagate differentials from the known ones.\\

    \noindent With the Generalized Leibniz Rule and the Generalized Mahowald Trick, Lin’s program becomes even more powerful, enabling the derivation of additional differentials. The program rigorously verifies the conditions of the relevant theorems and generates human-readable proofs \cite{LWXZenodo} for all differentials computed by the machine. In the Appendix, readers will find tables of Adams differentials that are used in the proof of the main Theorem \ref{thm:h62}.\\

    \item The inductive approach and adhoc arguments near stem 126.\\

    \noindent By studying Barratt--Jones--Mahowald's inductive approach \cite{BJMinduction} in the $\HF$-synthetic context, Burklund--Xu (Proposition 7.19 of \cite{BurklundXu}) proved that $h_6^2$ is a permanent cycle in the classical Adams spectral sequence if and only if $\lambda \eta \theta_5^2 = 0$ in the homotopy group of the synthetic sphere. Here $\lambda$ is our notation for the synthetic deformation parameter, and $\theta_5$ is any synthetic homotopy class that is detected by $h_5^2$ on the Adams $E_2$-page.   \\

    \noindent There are 105 additive generators in stem 125 of the classical Adams $E_2$-page -- these are the potential targets that could be hit by a nonzero differential from $h_6^2$. Using Lin's program and the inductive approach, we can rule out 101 out of the 105 elements.\\
    
    \noindent By further applying the Generalized Leibniz Rule and the Generalized Mahowald Trick, we can reduce the possibility of a nonzero differential supported by $h_6^2$ to only one case: A nonzero $d_{12}$ hits a certain element in stem 125, filtration 14 (see Proposition~\ref{prop:possibleh62}(2)). A necessary condition for this nonzero $d_{12}$ is an $\eta$-extension from stem 124, filtration 10 to stem 125, filtration 14 (see Proposition~\ref{prop:possibleh62}(5)). \\
    
    \noindent By a careful inspection of the homotopy groups of the synthetic sphere (and on certain Adams $E_n$-pages) near stem 126, we use adhoc arguments to prove that the $\eta$-extension in the previous paragraph cannot happen (see Proposition~\ref{prop:state5false}) and conclude that $h_6^2$ survives in the Adams spectral sequence. \\
\end{itemize} 

\noindent\textbf{Organization.} \ In Section 2, we setup and discuss properties of an extension spectral sequence for a map $f: X\to Y$ between two spectra, whose $E_0$-page is isomorphic to the direct sum of $E_\infty$-pages of the Adams spectral sequences of $X$ and $Y$, and differentials correspond to $f_*: \pi_*X \to \pi_*Y$, filtered by the Adams filtration. This setup allows us to discuss the jump-of-filtration phenomena rigorously. In Section 3, we recall and discuss certain properties of $\HF$-synthetic spectra. In Section 4, we discuss the extension spectral sequence from Section 2 in the context of $\HF$-synthetic spectra. In Section 5, we discuss extensions on a classical Adams $E_n$-page, defined in terms of $\HF$-synthetic spectra. In Section 6, using the language of Sections 4 and 5, we develop the Generalized Leibniz Rule and the Generalized Mahowald Trick. In Section 7, we use the inductive approach, and ad hoc arguments prove that $h_6^2$ survives in the Adams spectral sequence. Necessary information of the classical Adams spectral sequence for the final ad hoc arguments is included in the Appendix.\\

\noindent\textbf{Acknowledgement.} \ The second author is partially supported by grants NSFC-12325102, NSFC-12226002, the New Cornerstone Science Foundation, and Shanghai Pilot Program for Basic Research–Fudan University 21TQ1400100 (21TQ002). The third author is partially supported by NSF Grant DMS 2105462 and the AMS Centennial Research Fellowship. 

The authors express their deepest gratitude to Mark Behrens and Peter May for their invaluable guidance, insightful advice, and continuing support throughout this journey.

We also extend our heartfelt thanks to Soren Galatius, Lars Hesselholt, Mike Hill, Mike Hopkins, Dan Isaksen, Ciprian Manolescu, Haynes Miller, Yi Ni, Doug Ravenel, Gang Tian, Chenyang Xu, and Weiping Zhang for their encouragement and support of this project. The authors also thank Yuchen Wu and Shangjie Zhang for helpful comments on the early drafts of this paper. 

This work is dedicated to the memory of Mark Mahowald, whose groundbreaking contributions to algebraic topology continue to inspire generations of mathematicians. His deep insights and unwavering passion for the field remain a guiding light in our endeavors.

\section{A Spectral Sequence for Extensions}\label{sec:e6be737c}
In this paper, we assume that all symbols $X,Y,Z,W,\dots$ for spectra represent 2-completed connective spectra of finite type. In particular, their $\HF$-Adams spectral sequences converge strongly.

Consider a map $f: X\to Y$ between two spectra and we can construct the following chain complex
$$0\to \pi_*X\fto{\pi_*f} \pi_*Y\to 0$$
whose homology is 
$$\Ker(\pi_*f)\oplus \Coker(\pi_*f).$$
We encounter $f$-extension problems when analyzing the homotopy groups through the $E_\infty$-pages of the $\HF$-based Adams spectral sequences. To address these extensions, we introduce the following spectral sequence.

\begin{definition}\label{def:ess}
  For a map $f: X\to Y$, we define an $f$-extension spectral sequence (denoted $f$-ESS) as the spectral sequence derived from filtering the above chain complex according to the Adams filtration on $\pi_*$. The $E_0$-page of this spectral sequence is isomorphic to the direct sum of $E_\infty$-pages of the Adams spectral sequences of $X$ and $Y$: 
  $$\ESS{f}_0^{s,t} \cong E_\infty^{s,t}(X) \oplus E_\infty^{s,t}(Y) \Longrightarrow \Ker(\pi_*f)\oplus \Coker(\pi_*f)$$
  where the $d_0$-differential is induced by $f$ on the $E_\infty$-pages.
  
  Differentials in this spectral sequence has the form
  $$d_n^{f}: \ESS{f}_n^{s,t}\to \ESS{f}_n^{s+n, t+n}.$$
\end{definition}

To prevent confusion between differentials $d_n$ in the Adams spectral sequence of a spectrum $X$ and those in the $f$-extension spectral sequence, we denote the latter with a superscript, writing them as $d_n^{f}$. 

In the remainder of the paper, we denote by $\AF(f)$ the Adams filtration of a map $f$. Recall that $\AF(f) \ge k$ if and only if there exist $k$ composable maps $f_1, \ldots, f_k$, each inducing a trivial map in $\HF$-homology, such that $f$ is homotopic to the composition $f_k \circ \cdots \circ f_1$ (see \cite[Theorem~2.2.14]{RavenelGreenbook} for example).

\begin{notation}\label{nota:6fb333a2}
    Let $\ZESS{f}^{s,t}_n(X)\subset E^{s,t}_\infty(X)$ be the subgroup consisting of elements for which the $f$-extension differentials $d^f_0,\dots,d^f_n$ vanish.
    Let $\BESS{f}^{s,t}_n(Y)\subset E^{s,t}_\infty(Y)$ be the subgroup generated by the sum of images of the $f$-extension differentials $d^f_0,\dots,d^f_n$. We define $\BESS{f}^{s,t}_{-1}(Y)=0$ and $\ZESS{f}^{s,t}_{-1}(X)=E^{s,t}_\infty(X)$.\\
    
    By definition we have
    $$\ESS{f}^{s,t}_n\iso \ZESS{f}^{s,t}_{n-1}(X)\oplus \big(E^{s,t}_\infty(Y)/\BESS{f}^{s,t}_{n-1}(Y)\big).$$
    
    A nonzero $d^f_n$ differential takes the form
    $$d^f_n: \ZESS{f}^{s,t}_{n-1}(X)\to E^{s+n,t+n}_\infty(Y)/\BESS{f}^{s+n,t+n}_{n-1}(Y).$$
     When we write 
    $$d^f_n(x)=y$$
    for $x\in E_\infty^{s,t}(X)$ and $y\in E_\infty^{s+n,t+n}(Y)$, it is understood that $x$ belongs to $\ZESS{f}^{s,t}_{n-1}(X)$, and 
    $$d^f_n(x)=y+\BESS{f}^{s+n,t+n}_{n-1}(Y).$$
    Furthermore, we may sometimes write
    $$d^f_n(x)\equiv y\mod M$$
    for some subgroup $M\subset E_\infty^{s+n,t+n}(Y)$. This means that 
    $$d^f_n(x)=y+m+\BESS{f}^{s+n,t+n}_{n-1}(Y)$$
    for some $m\in M$.
\end{notation}

\begin{definition}\label{def:768fba8a}
    We say there is an $f$-extension from $x\in E_\infty^{s,t}(X)$ to $y\in E_\infty^{s+n,t+n}(Y)$ if $d_n^{f}(x)=y$ in the $f$-ESS. We say this $f$-extension is $\textbf{essential}$ if $y$ is nontrivial in the $E_n$-page of the $f$-ESS, or equivalently $$y\notin \BESS{f}^{s+n,t+n}_{n-1}(Y).$$
    Otherwise we refer to it as inessential.
\end{definition}

\begin{notation}
    For $x\in E_\infty^{s,t}(X)$, we denote by $\{x\}$  the set of all classes in $\pi_{t-s}X$ that are detected by $x$. We use $[x]$ to refer to a specific class or a general class in $\{x\}$, with the choice understood from the context.
\end{notation}

\begin{proposition}\label{prop:8154e6f1}
Consider $f: X\to Y$, $x\in E_\infty^{s,t}(X)$, $y\in E_\infty^{s+n,t+n}(Y)$ and $y'\in E_\infty^{s+m,t+m}(Y)$ for $m,n\ge 0$.
\begin{enumerate}
    \item There is an $f$-extension from $x$ to $y$, i.e., $d_n^{f}(x)=y$ in the $f$-ESS if and only if there is a class $[x]\in \{x\}$ such that $f[x]\in \{y\}$.
    \item An $f$-extension from $x$ to $y$ is inessential, i.e., $y$ is trivial in the $E_n$-page of the $f$-ESS if and only if there exists an element $x'\in E_\infty^{s+a,t+a}(X)$ for $0<a\le n$ such that we have an essential differential $d_{n-a}^{f}(x')=y$. Equivalently, there exists a class $[x'] \in \{x'\} \subset \pi_*X$ with AF$(x')>$AF$(x)$, such that $f[x'] \in \{y\}$.
    
    \item Suppose we have an $f$-extension from $x$ to both $y$ and $y'$.
    \begin{enumerate}
        \item If $m=n$, then $y-y'\in \BESS{f}^{s+n,t+n}_{n-1}(Y)$ and there exists an element $x'\in E_\infty^{s+a,t+a}(X)$ for $0<a\le n$ such that we have an essential differential 
        $$d_{n-a}^{f}(x') = y-y'$$
        in $f$-ESS. Equivalently, there exists a class $[x'] \in \{x'\} \subset \pi_*X$ with AF$(x')>$AF$(x)$, such that $f[x'] \in \{y-y'\}$. 
        \item If $m>n$, then the $f$-extension from $x$ to $y$ is inessential.
     \end{enumerate}
\end{enumerate}
\end{proposition}

\begin{proof}
    Part (1) is straightforward from the setup of the extension spectral sequence and Definition~\ref{def:768fba8a}. Parts (2) and (3) follow from Part (1).
\end{proof}

\begin{corollary}
    If the Adams filtration of $f$ is $k$, then
    $d^f_i=0$ for $i<k$.
\end{corollary}

\begin{proof}
    When $k>0$, it is clear that $d^f_0=0$, which means that for any $s$ and $[x]\in F_s\pi_*(X)$, we have $f([x])\in F_{s+1}\pi_*(Y)$. Since $\AF(f)=k$, we know that there exist $k$ maps
    $$X=X_0\fto{f_1}X_1\fto{f_2}\cdots\fto{f_k}X_k=Y$$
    such that each $f_i$ induces trivial map in $\HF$-homology (hence $\AF(f_i)>0$) and $f$ is homotopic to the composition $f_k\circ \cdots\circ f_1.$
    Then for any $s$ and $[x]\in F_s\pi_*(X)$, inductively we have
    $$f_1([x])\in F_{s+1}\pi_*(X_1)$$
    $$f_2\circ f_1([x])\in F_{s+2}\pi_*(X_2)$$
    $$\cdots$$
    $$f([x])=f_k\circ\cdots\circ f_1([x])\in F_{s+k}\pi_*(X_k)=F_{s+k}\pi_*(Y).$$
    By Proposition \ref{prop:8154e6f1}.(1) this means that $d_i^fx=0$ for any $x\in E_\infty^{*,*}(X)$ and $i<k$.
\end{proof}

\begin{definition} \label{def:98skj23}
\begin{enumerate}[leftmargin=*]
    \item For $x\in E_\infty^{s,t}(X)$, we say that the an $f$-extension $d_n^{f}(x)=y$ has a crossing that hits Adams filtration $p$ for some $p \leq s+n$, if there exists an element $x^\prime\in E_\infty^{s+a,t+a}(X)$ with $a>0$ and an $f$-extension $d_m^{f}(x^\prime)=y^\prime$ for $0\neq y^\prime\in E^{s+a+m, t+a+m}_\infty(Y)$ such that
    $$p = \textup{AF}(y^\prime) = s+a+m\le \textup{AF}(y) = s+n.$$
    \item We say that an $f$-extension $d_n^{f}(x)=y$ has no crossing that hits the range AF $\ge p$ if there does not exist an element $x^\prime\in E_\infty^{s+a,t+a}(X)$ with $a>0$ and an $f$-extension $d_m^{f}(x^\prime)=y^\prime$ for $0\neq y^\prime\in E^{s+a+m, t+a+m}_\infty(Y)$ such that
    $$p\le \textup{AF}(y^\prime) = s+a+m\le \textup{AF}(y) = s+n.$$
    \item We say that an $f$-extension $d_n^{f}(x)=y$ has no crossing if it has no crossing that hits the range AF$\ge \AF(x) + 1$.
\end{enumerate}
\end{definition}

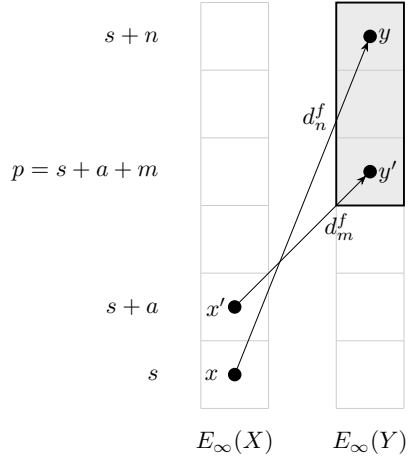
\begin{figure}[h]
\centering
\scalebox{0.9}{\begin{tikzpicture}
    % Draw grid
    \fill[llgray] (2-0.5, 5+0.5) rectangle ++(1, -3);
    \tikzgrid{0}{1}{0}{6}
    \tikzgrid{2}{3}{0}{6}
    \draw[thick] (2-0.5, 5+0.5) rectangle ++(1, -3);

    % Draw bullets
    \coordinate (x) at (0, 0);
    \coordinate (y) at (2, 5);
    \coordinate (x1) at (0, 1);
    \coordinate (y1) at (2, 3);
    
    \fill (x) circle (0.1);
    \fill (y) circle (0.1);
    \fill (x1) circle (0.1);
    \fill (y1) circle (0.1);

    % Draw differentials
    \drawarrow (x) node[left] {$x~$} -- (y) node[right] {$y$};
    \drawarrow (x1) node[left] {$x^\prime$} -- (y1) node[right] {$y^\prime$};
    
    \node[anchor=south east] at (1.5, 3.5) {$d^f_n$};
    \node[anchor=north west] at (1.2, 2.5) {$d^f_m$};

    \node[anchor=east] at (-1, 0) {$s$};
    \node[anchor=east] at (-1, 5) {$s+n$};
    \node[anchor=east] at (-1, 1) {$s+a$};
    \node[anchor=east] at (-1, 3) {$p=s+a+m$};

    \node at (0, -1.0) {$E_\infty(X)$};
    \node at (2, -1.0) {$E_\infty(Y)$};
\end{tikzpicture}}
\caption{A crossing of $d_n^f$ that hits Adams filtration $p$}
\end{figure}
   
\begin{remark}   

The following statements are equivalent.
\begin{itemize}
    \item An $f$-extension $d_n^{f}(x)=y$ has no crossing that hits the range AF$\ge p$. 
    \item If for any $a>0$ there is an $f$-extension from $x^\prime\in E_\infty^{s+a,t+a}(X)$ to some nontrivial $y^\prime$, then AF$(y^\prime)<p$ or AF$(y^\prime)>$AF$(y)$.
\end{itemize} 
\end{remark}
    
\begin{remark}   
We do not require the crossings in Definition~\ref{def:98skj23}~(1)(2) to be essential. A crossing implies the existence of an essential crossing, possibly a shorter one.
\end{remark}

\begin{proposition} \label{prop:i8r47oe}
    An $f$-extension from $x$ to $y$ has no crossing in the range $\AF\ge p$ if and only if for all $[x]\in \{x\}$ such that $\AF(f[x])\ge p$ we have $f[x]\in \{y\}$. In particular, we have
    \begin{enumerate}
        \item An $f$-extension from $x$ to $y$ has no crossing if and only if for all $[x]\in \{x\}$ we have $f[x]\in \{y\}$.
        \item An $f$-extension $d_n^{f}(x)=0$ has no crossing if and only if for all $[x]\in \{x\}$,
        $$\AF(f[x]) >\AF(x)+n.$$
        \item If $\AF(f)=n$, then all $d^f_n$-differentials have no crossing in the $f$-ESS.
    \end{enumerate}
\end{proposition}

\begin{proof}
    First we prove the only if part. Assume by contradiction that there exists $[x]\in \{x\}$ such that $f[x]=[y']\notin \{y\}$ is detected by $y'\in E_\infty(Y)$ with $\AF(y')\ge p$. Then there is an $f$ extension from $x$ to $y'$. By Proposition \ref{prop:8154e6f1}~(3), $y'$ (if $\AF(y)>\AF(y')$)  or $y-y'$ (if $\AF(y)=\AF(y')$) or $y$ (if $\AF(y)<\AF(y')$) should be hit by a shorter $d^f$ differential. This shorter differential is a crossing of the $f$-extension from $x$ to $y$ that hits $\AF\ge p$, which is a contradiction. Now we prove the if part. Assume that the $f$-extension has a crossing from $x'$ to $y'$ with $p\le \AF(y')\le \AF(y)$. Then there exists $[x']\in \{x'\}$ such that $f([x'])\in \{y'\}$. Note that $[x]+[x']\in \{x\}$ since $\AF(x')>\AF(x)$, we have
    $$f([x]+[x'])=f([x])+f([x'])=[y]+[y']\in\begin{cases}
        \{y'\} & \text{ if } \AF(y')<\AF(y)\\
        \{y+y'\} & \text{ if } \AF(y')=\AF(y).\\
    \end{cases}$$
    However, by assumption we should have $f([x]+[x'])\in \{y\}$ and it is contradictory to the equation above in either case.
    
    Hence the main statement is true.  Part (1) (2) are special cases of the main statement and Part (3) holds for degree reasons.
\end{proof}

\begin{example}
    Consider the Hopf map $\eta:S^1 \rightarrow S^0$ between 2-completed spheres. It is known that it induces a surjective map on $\pi_{46}$ (see [IWX] for example). 
    \begin{itemize}[leftmargin=1cm]
        \item In the $\eta$-ESS, we have the following nonzero differentials -- these are the essential $\eta$-extensions:
    \begin{enumerate}
        \item $d_1^\eta(h_5d_0) = h_1h_5d_0$ in AF$=6$,
        \item $d_2^\eta(\Delta h_1g) = d_0l$ in AF$=11$,
        \item $d_3^\eta(h_1g_2) = \Delta h_2c_1$ in AF$=8$,
        \item $d_4^\eta(h_3^2 h_5) = Mh_1$ in AF$=7$.\\
    \end{enumerate}
    
    \item The $\eta$-extension $d_2^\eta (h_0h_3^2h_5) = h_1h_5d_0$ is inessential.\\
    
        This is a zero differential on the $E_2$-page of the $\eta$-ESS due to the nonzero $d_1^\eta$-differential that hits $h_1h_5d_0$. Note that although this $\eta$-extension is inessential, we still have a class $[h_0h_3^2h_5] \in \{h_0h_3^2h_5\}$ such that $\eta \cdot [h_0h_3^2h_5] \in \{h_1h_5d_0\}$, due to the presence of an essential $\eta$-extension from $h_5d_0$ to $h_1h_5d_0$.\\
    
    \item The $\eta$-extension $d_4^\eta(h_3^2 h_5) = Mh_1$ has a crossing that hits AF$=6$.\\ 
    
    In fact, both the essential extension $d_1^\eta(h_5d_0) = h_1h_5d_0$ and the inessential extension $d_2^\eta (h_0h_3^2h_5) = h_1h_5d_0$ are such crossings. Due to the existence of a crossing, the following statement is NOT true:\\

    For all $[h_3^2 h_5] \in \{h_3^2 h_5\}$, we have $\eta \cdot [h_3^2 h_5] \in \{Mh_1\}$.\\

    In fact, there exists a class $[h_3^2 h_5] \in \{h_3^2 h_5\}$ such that $\eta \cdot [h_3^2 h_5] \in \{h_1h_5d_0\}$ due to the crossing $\eta$-extension from $h_5d_0$ to $h_1h_5d_0$.\\

    \item The $\eta$-extension $d_1^\eta(h_5d_0) = h_1h_5d_0$ has no crossing by part (3).
   
 \end{itemize}
\end{example}

\begin{theorem}\label{thm:4114f70c}
    Consider a homotopy commutative diagram of spectra
    $$\xymatrix{
    X \ar[r]^{f}\ar[d]_{p} & Y\ar[d]^{q}\\
    Z \ar[r]_{g} & W
    }$$
    Suppose $m,n,l\ge 0$, $0<k\le m+l-n$,
    $$\begin{array}{ll}
        x\in E^{s,t}_\infty(X) & y\in E^{s+n,t+n}_\infty(Y),\\
        z\in E^{s+m,t+m}_\infty(Z) & w\in E^{s+m+l,t+m+l}_\infty(W),
    \end{array}$$
    and
    \begin{enumerate}
        \item $d^{f}_{n}(x)=y$,
        \item $d^{p}_{m}(x)=z$,
        \item the differential in (1) or the differential in (2) has no crossing,
        \item $d^{g}_{l}(z)=w$, and has no crossing that hits the range AF $\ge s+n+k$,
        \item $d^{q}_{k-1} y=0$, and has no crossing.
    \end{enumerate}
    Then $d^{q}_{m+l-n}(y)=w$. (See Figure \ref{fig:4f154a8}.)
\end{theorem}

\begin{proof}
    First we find a representative $[x]$ of $x$ such that
    \begin{equation}\label{eq:9d45feac}
        p[x]\in \{z\}
    \end{equation}
    and
    \begin{equation}\label{eq:36e72616}
        f[x]\in \{y\}.
    \end{equation}
    If (1) has no crossing, by (2) we can pick $[x]\in \{x\}$ such that (\ref{eq:9d45feac}) holds;
    If (2) has no crossing, by (1) we can pick $[x]\in \{x\}$ such that  (\ref{eq:36e72616}) holds.
    In both cases the other equation also holds because of the no crossing condition.
    
    By (5), (\ref{eq:36e72616}) and Proposition~\ref{prop:i8r47oe}~(2) we know that $qf[x]$ has AF $\ge s+n+k$. Since the diagram on the Adams $E_\infty$-pages commutes, we have
    $$AF(gp[x])=AF(qf[x])\ge s+n+k.$$
    Combining with (4), (\ref{eq:9d45feac}) and Proposition~\ref{prop:i8r47oe} we have
    $$gp[x]\in \{w\}.$$
    Therefore, $qf[x]\in \{w\}$ and by (\ref{eq:36e72616}) there is a $q$-extension from $y$ to $w$.
\end{proof}

\begin{figure}
\centering
\scalebox{0.9}{\begin{tikzpicture}[node/.style={circle,fill=black!0.1}]
    \def\ymax{6}
    \def\ymaxp{7}
    % Define colors
    \definecolor{lightgray}{RGB}{200,200,200}
    
    \fill[llgray] (6-0.5, 6+0.5) rectangle ++(1, -2);
    % Draw grid
    \foreach \x in {0,...,7} {
        \draw[lightgray] (\x-0.5, -0.5) -- (\x-0.5, \ymax+0.5);
    }
    \foreach \x in {0,2,4,6} {
        \foreach \y in {0,...,\ymaxp} {
            \draw[lightgray] (\x-0.5, \y-0.5) -- (\x+0.5, \y-0.5);
        }
    }
    \draw[thick] (6-0.5, 6+0.5) rectangle ++(1, -2);

    % Draw bullets
    \coordinate (x) at (0, 0);
    \coordinate (y) at (4, 1);
    \coordinate (z) at (2, 4);
    \coordinate (w) at (6, 6);
    
    \fill (x) circle (0.1) node[below left] {$x$};
    \fill (y) circle (0.1) node[below right] {$y$};
    \fill (z) circle (0.1) node[above left] {$z$};
    \fill (w) circle (0.1) node[above right] {$w$};

    % Draw differentials
    \drawarrow (x) -- (y) node[midway,below] {$d^f_n$};
    \drawarrow (x) -- (z) node[midway,left] {$d^p_m$};
    \draw[dashed, -{Stealth[length=1.5mm]}, shorten >=(0.08cm)] (y) -- (6,6-1) node[midway,right] {$d^q_{\ge k}$};
    \drawarrow (z) -- (w) node[midway,above] {$d^g_{l}$};

    \node[anchor=east] (axis1) at (-.7, 0) {$s$};
    \node[anchor=east] (axis2) at (-.7, 1) {$s+n$};
    \node[anchor=east] (axis3) at (-.7, 4) {$s+m$};
    \node[anchor=east] (axis4) at (-.7, 5) {$s+n+k$};
    \node[anchor=east] (axis5) at (-.7, 6) {$s+m+l$};

    \node (EX) at (0, -1.0) {$E_\infty(X)$};
    \node (EY) at (4, -1.0) {$E_\infty(Y)$};
    \node (EZ) at (2, -1.0) {$E_\infty(Z)$};
    \node (EW) at (6, -1.0) {$E_\infty(W)$};
\end{tikzpicture}}
\caption{A demonstration of Theorem \ref{thm:4114f70c}}\label{fig:4f154a8}
\end{figure}
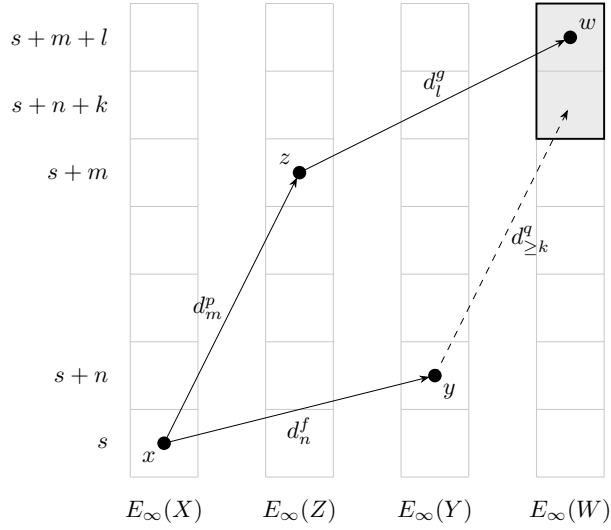

\begin{corollary} \label{cor:0012nik}
    Consider a homotopy commutative diagram of spectra
    $$\xymatrix{
    X \ar[r]^{f}\ar[d]_{p} & Y\ar[d]^{q}\\
    Z \ar[r]_{g} & W
    }$$
    Suppose $m,n,l\ge 0$, 
    $$\begin{array}{ll}
        x\in E^{s,t}_\infty(X) & y\in E^{s+n,t+n}_\infty(Y),\\
        z\in E^{s+m,t+m}_\infty(Z) & w\in E^{s+m+l,t+m+l}_\infty(W),
    \end{array}$$
    and
    \begin{enumerate}
        \item $d^{f}_{n}(x)=y$,
        \item $d^{p}_{m}(x)=z$,
        \item the differential in (1) or the differential in (2) has no crossing,
        \item $d^{g}_{l}(z)=w$, and has no crossing.
    \end{enumerate}
    Then $d^{q}_{m+l-n}(y)=w$.
\end{corollary}

\begin{corollary} \label{cor:166dc180}
    Consider a homotopy commutative diagram of spectra
    $$\xymatrix{
    X \ar[r]^{f}\ar[dr]_{p} & Y\ar[d]^{q}\\
    & Z
    }$$
    Suppose $n,m\ge 0$, 
    $$x\in E^{s,t}_\infty(X), ~y\in E^{s+n,t+n}_\infty(Y), ~z\in E^{s+m,t+m}_\infty(Z)$$
    and
    \begin{enumerate}
        \item $d^{f}_{n}(x)=y$,
        \item $d^{p}_{m}(x)=z$,
        \item the differential in (1) or the differential in (2) has no crossing.
    \end{enumerate}
    Then $d^{q}_{m-n}(y)=z$.
\end{corollary}

\begin{proof}
This is a special case of Corollary \ref{cor:0012nik} when $Z=W$ and $g=id_Z$.
\end{proof}

\begin{corollary} \label{cor:290d35ce}
    Consider a homotopy commutative diagram of spectra
    $$\xymatrix{
    X \ar[d]_{p}\ar[dr]^{q} & \\
    Z\ar[r]_{g} & W
    }$$
    Suppose $n,m\ge 0$, 
    $$x\in E^{s,t}_\infty(X), ~z\in E^{s+m,t+m}_\infty(Y), ~w\in E^{s+m+l,t+m+l}_\infty(W)$$
    and
    \begin{enumerate}
        \item $d^{p}_{m}(x)=z$,
        \item $d^{g}_{l}(z)=w$, and has no crossing.
    \end{enumerate}
    Then $d^{q}_{m+l}(x)=w$.
\end{corollary}

\begin{proof}
This is a special case of Corollary \ref{cor:0012nik} when $X=Y$ and $f=id_X$.
\end{proof}

\begin{corollary}\label{cor:e7b20ae2}
    Consider a homotopy commutative diagram of spectra
    $$\xymatrix{
    X \ar[r]^{f}\ar[d]_{p} & Y\ar[d]^{q}\\
    Z \ar[r]_{g} & W
    }.$$
    Assume $\ESS{p}_0^{*,*}=\ESS{p}_r^{*,*}$ and $\ESS{q}_0^{*,*}=\ESS{q}_r^{*,*}$ for some $r\ge 0$. Then 
    $$(d^p_r, d^q_r): E_\infty^{*,*}(X)\oplus E_\infty^{*,*}(Y)\to E_\infty^{*,*}(Z)\oplus E_\infty^{*,*}(W)$$
    induces a map from the $f$-ESS to the $g$-ESS.
\end{corollary}

\begin{proof}
    Since $\ESS{p}_0^{*,*}=\ESS{p}_r^{*,*}$ and $\ESS{q}_0^{*,*}=\ESS{q}_r^{*,*}$, we know that $d^p_r$, $d^q_r$ have no crossings. By Corollary \ref{cor:0012nik}, we have $d^g_0\circ d^p_r = d^q_r\circ d^f_0$ and therefore $(d^p_r, d^q_r)$ induces a map from $\ESS{f}_1^{*,*}$ to $\ESS{g}_1^{*,*}$. Inductively we can apply Corollary \ref{cor:0012nik} again and show that $(d^p_r, d^q_r)$ induces a map from $\ESS{f}_n^{*,*}$ to $\ESS{g}_n^{*,*}$ and $d^g_n\circ d^p_r = d^q_r\circ d^f_n$ for all $n$.
\end{proof}

\begin{corollary}\label{cor:aed3d1a4}
    If the composition $g\circ f$ of two maps $f:X\to Y$, $g:Y\to Z$ is trivial, and $d_n^f(x)=y$ for $x\in E_\infty^{s,t}(X)$ and $y\in E_\infty^{s+n,t+n}(Y)$, then $y$ is a permanent cycle in the $g$-ESS, i.e., we have
    $d^g_m(y)=0$ for all $m\ge 0$.
\end{corollary}

\begin{proof}
    This follows immediately from the following commutative diagram
    $$\xymatrix{
    X \ar[r]^{f}\ar[d] & Y\ar[d]^{g}\\
    0 \ar[r] & Z
    }$$
    and Corollary \ref{cor:0012nik}.

    There is a second short proof using homotopy groups. By Proposition \ref{prop:8154e6f1}, we know that there exists $[x]\in \{x\}$ such that $f([x])\in \{y\}$. Let $[y]=f([x])\in \{y\}$ and we have $g([y])=g(f([x]))=0$. Hence $y$ is a permanent $d^g$ cycle.
\end{proof}

\begin{proposition}\label{prop:cfe810af}
    Suppose that the sequence of spectra
    $$X\fto{f}Y\fto{g}Z$$
    induces a sequence of homotopy groups
    $$\pi_*X\fto{\pi_*f}\pi_*Y\fto{\pi_*g}\pi_*Z$$
    that is exact in $\pi_*Y$. Then all permanent cycles in $E^{*,*}_\infty(Y)$ of the $g$-ESS are boundaries in the $f$-ESS.
\end{proposition}

\begin{proof}
    If we consider the following sequence
    \begin{equation}\label{eq:8c7eec14}
        0\to \pi_*X\fto{\pi_*f}\pi_*Y\fto{\pi_*g}\pi_*Z\to 0
    \end{equation}
    and treat it as a chain complex filtered by the Adams filtration, we obtain an extension spectral sequence comprising $d^f$ and $d^g$ differentials. The abutment of this spectral sequence, when projected onto the $Y$ component, is zero. Therefor, all permanent $d^g$-cycles are killed by $d^f$-differentials.

    Again we offer another proof via homotopy groups. Assume that $y\in E^{*,*}_\infty(Y)$ is a permanent $d^g$ cycle. By the convergence of ESS we know that it detects an $g$-torsion $[y]\in \{y\}$ with $g([y])=0$. Hence there exists $[x]\in \pi_*X$ such that $f([x])=[y]$ by exactness. Assume that $[x]$ is detected by $x\in E^{*,*}_\infty(X)$. Then there exists an $f$ extension from $x$ to $y$.
\end{proof}

Corollary \ref{cor:aed3d1a4} and Proposition \ref{prop:cfe810af} are particularly useful because they provide numerous potential applications whenever we have a cofiber sequence of the form
$$X\fto{f}Y\fto{g}Cf\fto{h}\Sus X$$
These results can also be implemented in a computer program to facilitate their application.

\section{$\HF$-synthetic spectra} \label{sec:synthetic}

In Section~\ref{sec:e6be737c}, we explored classical extension problems in the context of commutative diagrams and cofiber sequences of spectra, framed in terms of the extension spectral sequence. The power of the extension spectral sequence can be further enhanced by generalizing and applying it to the category of $\HF$-synthetic spectra, a topic we will discuss in Section~\ref{sec:synext}.

In this section, we recall some recent work of Pstragowski \cite{Pst} and Appendix A of Burklund--Hahn--Senger \cite{BHS} on $E$-synthetic spectra, which is strongly motivated by the work of Hu--Kriz--Ormsby \cite{HuKrizOrmsby}, Isaksen \cite{Isaksen}, Gheorghe--Wang--Xu \cite{GWX}, and Gheorghe--Isaksen--Krause--Ricka \cite{GIKR} on the $\mathbb{C}$-motivic spectra. We will specialize some of these results to $\HF$-synthetic spectra, refining and highlighting specific properties.

The stable, presentable, symmetric monoidal $\infty$-category $\mathrm{Syn}_{\HF}$ of $\HF$-synthetic spectra is constructed in \cite{Pst}, along with a functor
$$\nu: \mathrm{Sp}\to \mathrm{Syn}_{\HF},$$
from the $\infty$-category of spectra $\mathrm{Sp}$. This functor is lax symmetric monoidal and preserves filtered colimits. However, $\nu$ does not generally preserve cofiber sequences. Instead, the following holds.

\begin{proposition}[{\cite[Lemma 4.23]{Pst}}]\label{prop:1f7950df}
    Suppose that
    $$X\fto{f} Y\fto{g} Z$$
    is a cofiber sequence of spectra. Then
    $$\nu X\fto{\nu f}\nu Y\fto{\nu g} \nu Z$$
    is a cofiber sequence of ${\HF}$-synthetic spectra if and only if
    $$0\to {\HF}_*X\fto{{\HF}_*f} {\HF}_*Y\fto{{\HF}_*g} {\HF}_*Z\to 0$$
    is a short exact sequence.
\end{proposition}

\begin{definition}[{\cite[Definition 4.6]{Pst}}]
The bigraded spheres $S^{s,w}$ are the synthetic spectra defined by $\Sigma^{s-w}\nu (S^w)$.
\end{definition}

\begin{definition}[{\cite[Definition 4.27]{Pst}}]
    The natural comparison map
    $$S^{0,-1}=\Sus \nu (S^{-1})\to \nu(\Sus S^{-1})=S^{0,0}$$
    in the category of $\HF$-synthetic spectra is denoted by $\lambda\in \pi_{0,-1}S^{0,0}$. We denote by $S^{0,0}/\lambda$ the cofiber of $\lambda$. In general for synthetic $X$ the following
    $$\lambda\wedge id_X: \Sus^{0,-1}X\to X.$$
    is a natural transformation and by abuse of notation this map is also denoted by $\lambda$.
    We denote by $X/\lambda$ the cofiber of $\lambda$
    which is equivalent to $(S^{0,0}/\lambda) \wedge X$.
\end{definition}

\begin{remark}
For every $n\ge 1$, $S^{0,0}/\lambda^n$ is an $E_\infty$-object in $\mathrm{Syn}_{\HF}$ (see \cite[Appendix B, C]{BHSmot} for example).
\end{remark}

\begin{remark} \label{rem:deformation}
The element $\lambda$ can be viewed as a parameter of a categorical deformation on $\mathrm{Syn}_{\HF}$ (see \cite{Pst}). 
$$\xymatrix{\textup{Sp}^\wedge_2 & & \mathrm{Syn}_{\HF}\ar[ll]_-{\lambda^{-1}}^-{\textup{generic fiber}} \ar[rr]^-{mod ~  \lambda}_-{\textup{special fiber}} & & \mathcal{D}(A_*\textup{-Comod})}$$
Inverting $\lambda$, the generic fiber is equivalent to the classical category of 2-completed spectra $\textup{Sp}^\wedge_2$. Modding out by $\lambda$, the special fiber is equivalent to Hovey's derived category \cite{Hovey} of comodules over the mod 2 dual Steenrod algebra $\mathcal{D}(A_*\textup{-Comod})$. Historically, this synthetic deformation is motivated by, and analogous to, the motivic deformation conjectured by Isaksen and proven by Gheorghe--Wang--Xu \cite{GWX}: Specifically, there exists a map
$$\tau: \widehat{S^{0,-1}} \rightarrow \widehat{S^{0,0}}$$
between p-completed motivic spheres. This element $\tau$
can be viewed as a parameter of a categorical deformation on the category of cellular objects over $\widehat{S^{0,0}}$. 
$$\xymatrix{\textup{Sp}^\wedge_p & & \widehat{S^{0,0}}\textup{-Mod}\ar[ll]_-{\tau^{-1}}^-{\textup{generic fiber}} \ar[rr]^-{mod ~  \tau}_-{\textup{special fiber}} & & \mathcal{D}(BP_*BP\textup{-Comod})}$$
Inverting $\tau$, the generic fiber is equivalent to the classical category of p-completed spectra $\textup{Sp}^\wedge_p$. Modding out by $\tau$, the special fiber is equivalent to Hovey's derived category of comodules over the Hopf algebroid $BP_*BP$. 
\end{remark}

\begin{theorem}[{\cite[Theorem A.8]{BHS}}] \label{thm:rigid}
    The synthetic Adams spectral sequence for $\nu X$
     $$\SynSS_2^{s,t,w}(\nu X)\Longrightarrow \pi_{t-s,w}(\nu X)$$
     has an $E_2$-page 
    $$\SynSS_2^{*,*,*}(\nu X) \cong E_2^{*,*}(X)\otimes \mathbb{F}_2[\lambda]$$
    with differentials
    $$d_r: \SynSS_r^{s,t,w}\to \SynSS_r^{s+r,t+r-1,w}.$$
    Here $E_2^{*,*}(X) \cong \Ext_A^{*,*}(X)$ is the classical Adams $E_2$-page of $X$. An element in $E_2^{s,t}(X)$ is viewed as having tridegree $(s, t, t)$ in $\SynSS_2^{s,t,t}(\nu X)$, and $\lambda$ is in tridegree $(0, 0, -1)$. Given a classical Adams differential $d_r^{\mathrm{cl}}(x)=y$, the corresponding synthetic differential is $d_r(x)=\lambda^{r-1}y$, which is $\lambda$-linear, and all synthetic Adams differentials arise in this way.
\end{theorem}

\begin{remark}
    Here, we use a third grading of the synthetic Adams $E_2$-page that differs from \cite{BHS} but more compatible with \cite{BurklundIsaksenXu}. Specifically, we view elements in $E_2^{s,t}(X)$ as having tridegree $(s,t,t)$, instead of $(s,t,s)$ as in \cite{BHS}. This choice ensures that differentials and maps preserve the third degree, which proves convenient for the extension spectral sequence in the $\HF$-synthetic category in Section~\ref{sec:synext}.
\end{remark}

\begin{theorem}[{\cite[Theorem A.1]{BHS}}]\label{thm:17e90ac0}
    The synthetic Adams spectral sequence for $\nu X$ is isomorphic to the $\lambda$-Bockstein spectral sequence (with no sign difference because we are working over $\bF_2$ here). More specifically, we have
    \begin{enumerate}
        \item $E_2^{s,t}(X)\iso \pi_{t-s,t}(\nu X /\lambda)$.
        \item If we have a classical Adams differential $d_rx=y$ for 
        $$x\in E_2^{s,t}(X)\iso \pi_{t-s,t}(\nu X /\lambda),$$
        then $x$ admits a lift to $\pi_{t-s,t}(\nu X/\lambda^{r-1})$ whose image under the $\tau$-Bockstein
        $$\nu X/\lambda^{r-1}\to \Sus^{1,-r+1}\nu X/\lambda$$
        is equal to $d_r(x)$.
    \end{enumerate}
\end{theorem}

\begin{remark}
    The isomorphism with the $\lambda$-Bockstein spectral sequence in Theorem~\ref{thm:17e90ac0} and the correspondence of differentials between the classical and synthetic Adams spectral sequences in Theorem~\ref{thm:rigid} can be interpreted as manifestations of the rigidity of the synthetic Adams spectral sequence. This type of rigidity was first observed by Hu--Kriz--Ormsby \cite{HuKrizOrmsby} and Isaksen \cite{Isaksen} in the context of the motivic Adams--Novikov spectral sequence, and it was utilized extensively in the work of Gheorghe--Wang--Xu \cite{GWX}. Subsequently, Burklund--Xu \cite{BurklundXu} uncovered the rigidity of the motivic Cartan--Eilenberg spectral sequence and performed computations based on this property. See also \cite{mmf}, \cite{Burklund} for applications of this type of rigidity.

    Together with Remark~\ref{rem:deformation}, $\HF$-synthetic spectra can be seen as a categorification of the classical Adams spectral sequence, just as cellular motivic spectra can be viewed a categorification of the classical Adams--Novikov spectral sequence.
\end{remark}

\begin{notation}
    Consider the classical Adams spectral sequence of $X$. Let $B_r^{s,t}(X)$ denote the subgroups of $E_2^{s,t}(X)$ generated by the sum of images of Adams differentials $d_2,\dots,d_r$. Let $Z_r^{s,t}(X)$ denote the subgroup of $E_2^{s,t}(X)$ consisting of elements on which $d_2,\dots,d_r$ vanish. For $r=1$, we define $Z_1^{s,t}(X)=E_2^{s,t}(X)$ and $B_1^{s,t}(X)=0$. For $r=\infty$, let $Z_\infty^{s,t}(X)$ be the intersection of all $Z_r^{s,t}(X)$, and $B_\infty^{s,t}(X)$ be the union of all $B_r^{s,t}(X)$. We then have the following inclusions
    $$0 = B_1 \subset B_2\subset B_3\subset \cdots\subset B_\infty\subset Z_\infty \subset \cdots\subset Z_3\subset Z_2\subset Z_1 = E_2.$$
    In this article, we interpret a classical Adams differential $d_r$ as the map
    $$d_r: Z_{r-1}^{s,t}(X)\to E_2^{s+r,t+r-1}(X)/B_{r-1}^{s+r,t+r-1}(X),$$
    and follow the pattern in Notation \ref{nota:6fb333a2} for writing differentials.
\end{notation}

The following two propositions are taken from \cite[Corollary A.9, A.11]{BHS}. Note that our third grading, $w$, differs from that used in the reference.

\begin{proposition}\label{prop:30e8b746}
    The $E_\infty$-page of the synthetic Adams spectral sequence for $\nu X$ is given by
    $$E_\infty^{s,t,w}(\nu X)\iso\begin{cases}
        Z_\infty^{s,t}(X)/B_{1+t-w}^{s,t}(X) & \text{ if } t \ge w,\\
        0 & \text{ otherwise }.
    \end{cases}$$
\end{proposition}

\begin{proposition}\label{prop:59f111f}
    The $E_\infty$-page of the synthetic Adams spectral sequence for $\nu X/\lambda^r$ for $r\ge 2$ is given by
    $$E_\infty^{s,t,w}(\nu X/\lambda^r)\iso\begin{cases}
        Z_{r-t+w}^{s,t}(X)/B_{1+t-w}^{s,t}(X) & \text{if } 0\le t-w < r,\\
        0 & \text{otherwise}.
    \end{cases}$$
\end{proposition}

\begin{notation}
    For convenience, whenever $n\le 0$ or $m\le 0$, we will define $Z_n(X)/B_m(X)$ as a trivial group. This is a bit counterintuitive to the definition of $Z_n$ and $B_m$ but it allows us to consistently express the following
    $$E_\infty^{s,t,w}(\nu X)\iso Z_\infty^{s,t}(X)/B_{1+t-w}^{s,t}(X)$$
    and
    $$E_\infty^{s,t,w}(\nu X/\lambda^r)\iso Z_{r-t+w}^{s,t}(X)/B_{1+t-w}^{s,t}(X).$$
\end{notation}

\begin{remark}
    The isomorphisms in Propositions~\ref{prop:30e8b746} and \ref{prop:59f111f} are compatible with the synthetic maps $\lambda^{r'-r}: \Sus^{-r'+r}\nu X/\lambda^r\to \nu X/\lambda^{r'}$ and $\rho: \nu X/\lambda^r\to \nu X/\lambda^{r''}$ for $r'\ge r\ge r''$. The induced map
    $$\lambda^{r'-r}:E_\infty^{s,t,w}(\nu X/\lambda^r)\to E_\infty^{s,t,w-r'+r}(\nu X/\lambda^{r'})$$
    corresponds to the quotient map (when all subscripts are positive):
    $$\lambda^{r'-r}:Z_{r-t+w}^{s,t}(X)/B_{1+t-w}^{s,t}(X)\to Z_{r-t+w}^{s,t}(X)/B_{1+t-w+r'-r}^{s,t}(X)$$
    Similarly, the induced map
    $$\rho:E_\infty^{s,t,w}(\nu X/\lambda^r)\to E_\infty^{s,t,w}(\nu X/\lambda^{r''})$$
    corresponds to the embedding map (when all subscripts are positive):
    $$\rho:Z_{r-t+w}^{s,t}(X)/B_{1+t-w}^{s,t}(X)\to Z_{r''-t+w}^{s,t}(X)/B_{1+t-w}^{s,t}(X).$$
    These data do not extend beyond classical information. However, when we consider a map $f: X\to Y$, we can obtain interesting extensions, as will be discussed in Section~\ref{sec:synext}.
\end{remark}

\begin{example} \label{exam:synEinfty}
    Consider the first few nonzero differentials in the classical Adams spectral sequence from the 15-stem to the 14-stem: 
    $$d_2(h_4) = h_0h_3^2,  \ d_3(h_0h_4) = h_0d_0, \ d_3(h_0^2h_4) = h_0^2d_0,$$
    where $$h_0 \in \Ext_A^{1,1}, \ h_3 \in \Ext_A^{1,8}, \ h_4 \in \Ext_A^{1,16}, \ d_0 \in \Ext_A^{4,18}.$$
    The $E_2$-page of the synthetic Adams spectral sequence for $S^{0,0}$ has the form
    $$\SynSS_2^{*,*,*}(S^{0,0}) \cong \Ext_A^{*,*}\otimes \mathbb{F}_2[\lambda],$$
    and we have
    $$h_0 \in \SynSS_2^{1,1,1}, \ h_3 \in \SynSS_2^{1,8,8}, \ h_4 \in \SynSS_2^{1,16,16}, \ d_0 \in \SynSS_2^{4,18,18}.$$
    The corresponding nonzero synthetic Adams differentials for $S^{0,0}$ are
    $$d_2(h_4) = \lambda h_0h_3^2,  \ d_3(h_0h_4) = \lambda^2 h_0d_0, \ d_3(h_0^2h_4) = \lambda^2 h_0^2d_0,$$
    along with their $\lambda$-multiples. As a result, the 14-stem of the synthetic $E_\infty$-page is generated by $\lambda$-free elements $h_3^2$ and $d_0$, along with the following $\lambda$-torsion elements:
    $$h_0h_3^2, \ h_0d_0, \ \lambda h_0d_0, \ h_0^2d_0, \ \lambda h_0^2d_0.$$
    On the other hand, since $h_0h_3^2 \in \Ext_A^{3,17}$ is a $d_2$-cycle, we have 
    $$0 = B_1^{3,17} \subset B_2^{3,17} = Z_\infty^{3,17} = E_2^{3,17} \cong \mathbb{F}_2\{h_0h_3^2\},$$
    which matches the description:
    $$E_\infty^{3,17,w}(S^{0,0})\iso\begin{cases}
        Z_\infty^{3,17}/B_{18-w}^{3,17} \cong \mathbb{F}_2\{h_0h_3^2\} & \text{ if } w = 17,\\
        Z_\infty^{3,17}/B_{18-w}^{3,17} = 0 & \text{ if } w \le 16,\\
        0 & \text{ if } w \ge 18.
    \end{cases}$$
    The element $h_0h_3^2$ detects homotopy classes of $\AF=3$ in $\pi_{14,17}$. Due to the presence of the $\lambda$-free element $d_0$ in the synthetic $E_\infty$-page, the element $\lambda d_0$ also detects homotopy classes in $\pi_{14,17}$, but with $\AF=4$. This illustrates why it is more convenient to describe the $E_\infty$-page as in Proposition~\ref{prop:30e8b746}, rather than in terms of synthetic homotopy groups.
    
    As a comparison, the only nonzero synthetic Adams differential in the 14-stem for $S^{0,0}/\lambda^2$ is 
    $$d_2(h_4) = \lambda h_0h_3^2,$$
    and the elements $\lambda h_4, \ h_0h_4, \ h_0^2h_4$ are all permanent cycles. 

    For $S^{0,0}/\lambda^3$, the nonzero synthetic Adams differentials in the 14-stem are
    $$d_2(h_4) = \lambda h_0h_3^2, \  d_2(\lambda h_4) = \lambda^2 h_0h_3^2,$$ 
    $$d_3(h_0h_4) = \lambda^2 h_0d_0, \ d_3(h_0^2h_4) = \lambda^2 h_0^2d_0,$$
    and the elements $\lambda^2 h_4, \ \lambda h_0h_4, \ \lambda^2 h_0h_4, \ \lambda h_0^2h_4, \lambda^2 h_0^2h_4$ are all permanent cycles. 

    These results can be compared with the statements in Proposition~\ref{prop:59f111f}.
\end{example}

\begin{proposition}[{\cite[Lemma 9.15]{BHS}}]\label{prop:ef21f9bc}
    If a map $f:X\to Y$ has Adams filtration $\AF(f)=k$, then there exists a factorization
    $$\xymatrix{
     & \Sus^{0,-k}\nu Y\ar[d]^{\lambda^k}\\
     \nu X \ar[r]_{\nu f}\ar@{-->}[ru]^{\Sus^{0,-k}\tilde f} & \nu Y
    }$$
    where $\tilde f: \Sus^{0,k}\nu X\to \nu Y$ is called a synthetic lift of $f$.
\end{proposition}

In Proposition \ref{prop:ef21f9bc} the map $\nu f$ can be also factorized as follows:
$$\xymatrix{
 \nu X \ar[r]^{\nu f}\ar[d]_{\lambda^k} & \nu Y\\
 \Sus^{0,k}\nu X\ar@{-->}[ru]_{\tilde f}
}$$

%Proposition \ref{prop:1f7950df} and Proposition \ref{prop:ef21f9bc} are specializations of the original statements about general $E$-synthetic spectra to the case $E=\HF$. 

\begin{proposition}\label{prop:41561db2}
    Suppose that we have a  distinguished triangle of spectra
    $$X\fto{f} Y\fto{g} Z\fto{h} \Sus X$$
    with $\AF(h)>0$, and consequently a short exact sequence on $\HF$-homology
    $$0\to H_*X\fto{H_*f} H_*Y\fto{H_*g} H_*Z\to 0.$$
    Then there exists a distinguished triangle of synthetic spectra
    $$\nu X\fto{\nu f}\nu Y\fto{\nu g} \nu Z\fto{\Sus^{0,-1}\hat h} \Sus^{0,-1}\nu \Sus X=\Sus^{1,0}\nu X$$
    such that $\nu h=\lambda \hat h$.
\end{proposition}
\begin{proof}
    The proof is contained in the proof of \cite[Lemma 9.15]{BHS}.
\end{proof}

\begin{remark}
For $h$ in Proposition \ref{prop:41561db2}, let $\tilde h$ be a synthetic map such that $\nu h=\lambda^k\tilde h$, then $\hat h$ is equal to $\lambda^{k-1}\tilde h$ up to some $\lambda$-torsion. 
\end{remark}

\begin{notation} \label{not:fhat}
     For a map $f: X\to Y$ which is part of a distinguished triangle
    $$X\fto{f} Y\fto{g}Cf\to\Sus X,$$
    we define
    $$e(f)=\begin{cases}
     0 & \text{ if }\AF(f)=0,\\
     1 & \text{ if }\AF(f)>0.
    \end{cases}$$
    When $\AF(f)=0$, we also denote $\nu f$ by $\hat f$. In both cases, we have 
    $$\hat f: \Sus^{0,e(f)}\nu X\to \nu Y$$
    and $\nu f=\lambda^{e(f)}\hat f$. Further more, $C\hat f\simeq \Sus^{0,-e(g)}\nu Cf$ or equivalently,
    \begin{itemize}
        \item if $f$ induces a trivial map or an injection on $\HF$-homology, then $C\hat f\simeq \nu Cf$;
        \item if $f$ induces a surjection on $\HF$-homology, then $C\hat f \simeq \Sus^{0,-1}\nu Cf$.
    \end{itemize}
\end{notation}

With the notation above, we can rewrite Proposition \ref{prop:41561db2} into the following form which will be used in Theorem \ref{thm:158d451a} and help us unify different cases.

\begin{proposition}
    Suppose that we have a distinguished triangle of spectra
    $$X\fto{f} Y\fto{g} Z\fto{h} \Sus X$$
    such that $e(f)+e(g)+e(h)=1$. Then we have a distinguished triangle of synthetic spectra
    $$\nu X\fto{\hat f} \Sus^{0,-e(f)}\nu Y\fto{\hat g} \Sus^{0,-e(f)-e(g)}\nu Z\fto{\hat h} \Sus^{0,-1}\nu \Sus X$$
    (For simplicity we often omit the suspensions before maps but keep the suspensions before spectra.)
\end{proposition}

\begin{remark}
    When $\AF(f)>1$, it is often more advantageous to work with $C\tilde f$ rather than $C\hat f$, as the $\lambda$-Bockstein spectral sequence for $C\tilde f$ corresponds to a modified Adams spectral sequence, which differs from the classical one but may provide additional information. On the other hand, given the equivalence  $C\hat{f}\simeq \nu Cf$ from Notation~\ref{not:fhat}, the rigidity Theorems~\ref{thm:rigid} and \ref{thm:17e90ac0} imply that the synthetic Adams spectral sequence for $C\hat f$ is isomorphic to the $\lambda$-Bockstein spectral sequence, and aligns with the classical Adams spectral sequence for $Cf$.  Since the computational data used in this work is based on the classical Adams spectral sequence, we defer the exploration of $C\tilde f$ to future work.
    
    To illustrate this in more detail, consider the example of the synthetic lift $\tilde{\eta^3} \in \pi_{3,6}$ of $\eta^3 \in \pi_3$, which has $\AF = 3$. We can examine several $\lambda$-multiples of $\tilde{\eta^3}$:
    $$\lambda \tilde{\eta^3} \in \pi_{3,5}, \ \ \hat{\eta^3} = \lambda^2 \tilde{\eta^3} \in \pi_{3,4}, \ \ \nu(\eta^3) = \lambda^3 \tilde{\eta^3} \in \pi_{3,3}.$$

   Now, consider the synthetic 2-cell complexes formed by taking the cofibers of these maps:
    $$C(\tilde{\eta^3}), \ C(\lambda\tilde{\eta^3}), \ C(\lambda^2\tilde{\eta^3}) \simeq C(\hat{\eta^3}) \simeq \nu (C\eta^3), \ C(\lambda^3\tilde{\eta^3}) \simeq C\nu(\eta^3).$$

    The rigidity Theorems~\ref{thm:rigid} and \ref{thm:17e90ac0} apply only to $C(\hat{\eta^3})$, as it is equivalent to $\nu (C\eta^3)$. In this case, its synthetic Adams $E_2$-page is $\lambda$-free over $\Ext_A^{*,*}(C\eta^3)$; in particular the identity element from the top cell has tridegree $(s,t,w) = (0,4,4)$. There is a correspondence of differentials between the classical and synthetic Adams spectral sequence, and a synthetic differential generates $\lambda$-torsion corresponding to its length.

    For $C(\tilde{\eta^3}), C(\lambda\tilde{\eta^3})$ and $C(\lambda^3\tilde{\eta^3})$, their synthetic Adams $E_2$-pages remain $\lambda$-free, but the identity element from the top cell has tridegree 
    $$(s,t,w) =(0,4,6), (0,4,5), (0,4,7)$$ 
    respectively. As a result, each supports a nonzero Adams $d_3$-differential killing $h_1^3$ from the bottom cell when multiplied by $1, \lambda, \lambda^3$, respectively. This leaves no $\lambda$-torsion, $\lambda^1$-torsion and $\lambda^3$-torsion in $\pi_{4,6}$, respectively.

    On the other hand, the $\lambda$-Bockstein spectral sequences for $C(\tilde{\eta^3})$ and $ C(\lambda\tilde{\eta^3})$ correspond to a modified Adams spectral sequence in the sense of \cite[Section 3]{BHHM}, with the AF of $\eta^3$ ``modified" to $\AF=1$ and $\AF=2$, respectively. In particular, the $E_2$-page of the $\lambda$-Bockstein spectral sequence for $C(\tilde{\eta^3})$, i.e., $\pi_{*,*}C(\tilde{\eta^3})/\lambda$, may contain more multiplicative information than $\Ext_A^{*,*}(C\eta^3)$.

 %   For the synthetic Adams $E_2$-page of $C(\lambda\tilde{\eta^3})$, it remains $\lambda$-free, but the contributions from the top cell will increase their AF by 1, and the first nonzero Adams differential is a $d_2$. As for $C(\tilde{\eta^3})$, its synthetic Adams $E_2$-page is also $\lambda$-free, and the contributions from the top cell will increase their AF by 2. In particular, $h_1^3$ from the bottom cell is already zero on the $E_2$-page, as if it were killed by a $d_1$. 
    
 %   In the other direction, for $C(\lambda^3\tilde{\eta^3}) \simeq C\nu(\eta^3)$, its synthetic Adams $E_2$-page is no longer $\lambda$-free, and the contributions from the top cell will lower their AF by 1, leaving $h_1^3$ from the bottom cell as a $\lambda^3$-torsion on the $E_2$-page, as if it were killed by a $d_4$.

\end{remark}

\section{Synthetic Extensions} \label{sec:synext}
The extension spectral sequence introduced in Section \ref{sec:e6be737c} can be generalized for $\HF$-synthetic spectra. Consider a map $f: X\to Y$ between two connective finite $\HF$-synthetic spectra\footnote{In our solution to the Kervaire invariant problem, we only consider finite synthetic spectra. However, the arguments in this section apply to any synthetic spectra for which the synthetic Adams spectral sequence strongly converges.}. By filtering
$$0\to \pi_{*,*}X\fto{\pi_{*,*}f} \pi_{*,*}Y\to 0$$
using the Adams filtration, we obtain an $f$-extension spectral sequence:
$$\ESS{f}_0^{s,t,w} \cong E_\infty^{s,t,w}(X) \oplus E_\infty^{s,t,w}(Y) \Longrightarrow \Ker(\pi_{*,*}f)\oplus \Coker(\pi_{*,*}f)$$
with differential
$$d_n^f: \ZESS{f}_{n-1}^{s,t,w}(X)\to E_\infty^{s+n,t+n,w}(Y)/\BESS{f}_{n-1}^{s+n,t+n,w}(Y)$$
where $\ZESS{f}_*^{*,*,*}$ and $\BESS{f}_*^{*,*,*}$ are  analogous to the notations in Notation \ref{nota:6fb333a2}, with the addition of an extra degree.

In fact all the definitions and results in Section \ref{sec:e6be737c} remain valid in this synthetic setting, provided we include the extra weight degree in addition to the topological degree and ensure that all maps also preserve the weight degree. For convenience, we will not restate these results here but will occasionally refer to them and apply their synthetic versions when needed.

\begin{notation} \label{not:deltaandrho}
There are certain special maps between synthetic spectra that we want to consider. For any $n<m\le \infty$ and a spectrum $X$, we have the following  distinguished triangles of $\HF$-synthetic spectra
$$\Sus^{0,-n}\nu X/\lambda^{m-n}\fto{\lambda^n} \nu X/\lambda^m\fto{~\rho_{n,m}~} \nu X/\lambda^n\fto{~\delta_{n,m}~} \Sus^{1,-n}\nu X/\lambda^{m-n}.$$
We simply write $\rho=\rho_{n,m}$, $\delta=\delta_{n,m}$ by abuse of notation if $n,m$ is understood in the context.
When $m=\infty$, this sequence is interpreted  as
$$\Sus^{0,-n}\nu X\fto{\lambda^n} \nu X\fto{~\rho~} \nu X/\lambda^n\fto{~\delta~} \Sus^{1,-n}\nu X.$$

Many arguments for finite $m$ implies the corresponding statement when $m=\infty$ because of
$$\nu X=\varprojlim_{m} \nu X/\lambda^m$$
as shown in \cite[Proposition A.13]{BHS}.

We also note that our grading for the triangulation translation functor is smashing with $S^{1,0}$, which is consistent with \cite[Appendix A]{BHS} but is different from \cite[Section 7]{BurklundXu}.
\end{notation}

%First we show that there are no $\lambda^n$ or $\rho$-hidden extensions.

\begin{proposition}\label{prop:9770ae6e} 
The only nonzero differentials in the extension spectral sequences for the maps $\lambda^n$ and $\rho$ from Notation~\ref{not:deltaandrho} are the $d_0$'s:
$$d^{\lambda^n}_0=\lambda^n \ \text{and} \ d^{\rho}_0=\rho.$$ 
As a result, these $d_0$'s have no crossings.
\end{proposition}
\begin{proof}
    First, we show that the the following sequence is exact in the middle:
    \begin{equation}\label{eq:d8ee7ffe}
        \scalebox{1}{$E_\infty^{s,t,w}(\Sus^{0,-n}\nu X/\lambda^{m-n})\fto{\lambda^n} E_\infty^{s,t,w}(\nu X/\lambda^{m})\fto{\rho} E_\infty^{s,t,w}(\nu X/\lambda^{n})$}.
    \end{equation}
    We apply Proposition \ref{prop:30e8b746} and prove this case by case.

    When $t-w<0$ or $t-w\ge m$, the middle group $E_\infty^{s,t,w}(\nu X/\lambda^{m}) = 0$, so the sequence is exact in the middle.
    
    When $0\le t-w<n$, the sequence is isomorphic to
    $$0\to Z^{s,t}_{m-t+w}(X)/B^{s,t}_{1+t-w}(X)\to Z^{s,t}_{n-t+w}(X)/B^{s,t}_{1+t-w}(X)$$
    which is exact in the middle because the second map is clearly injective.

    When $n\le t-w<m$, the sequence is isomorphic to
    $$Z^{s,t}_{m-t+w}(X)/B^{s,t}_{1+t-w-n}(X)\to Z^{s,t}_{m-t+w}(X)/B^{s,t}_{1+t-w}(X)\to 0$$
    which is exact in the middle because the first map is clearly surjective.

    Thus, we have shown that (\ref{eq:d8ee7ffe}) is always exact in the middle. Combined with Corollary \ref{cor:aed3d1a4} we conclude that all $d^{\lambda^n}_i$, $d^{\rho}_i$ are trivial for $i>0$.
\end{proof}

\begin{remark}
    From the proof above, we see that in (\ref{eq:d8ee7ffe}), $\lambda^n$ is surjective or trivial, while $\rho$ is injective or trivial.
\end{remark}

However, the $\delta$-extension spectral sequence for
$$\delta:\nu X/\lambda^n\to \Sus^{1,-n}\nu X/\lambda^{m-n}$$
is more complicated, as it encodes classical Adams differentials $d_2$ through $d_m$.

\begin{remark}
    For the convenience of readers to check gradings, whenever we write 
    $$d_{n}^{f}(x)=\lambda^k y$$
    for $f: \Sus^{m,w} \nu X\to \nu Y$,
    $x\in E_\infty^{s_1,t_1,w_1}(\nu X)$ and
    $y\in E_\infty^{s_2,t_2,w_2}(\nu Y)$,
    the following conditions must hold: 
    $$s_2=s_1+n, \ t_2-s_2=t_1-s_1+m, \ \text{and} \ w_2=w_1+w.$$
\end{remark}

\begin{proposition}\label{prop:6de7d130}
    Suppose in the classical Adams spectral sequence of $X$ we have $d_r(x)=y$, where $x\in Z_{r-1}^{s,t}(X)$ and $y\in Z_\infty^{s+r,t+r-1}(X)$. Consider the map
    $$\delta:\nu X/\lambda^n\to \Sus^{1,-n}\nu X/\lambda^{m-n}.$$
    \begin{enumerate}
        \item If $r\ge n+1$, then we view $x$ as an element of 
        $$E_\infty^{s,t,t}(\nu X/\lambda^n)\iso Z_n^{s,t}(X),$$
        and $\lambda^{r-n-1} y$ as an element of 
        $$E_\infty^{s+r,t+r-1, t+n}(\nu X/\lambda^{m-n})\iso Z_{m-r+1}^{s+r,t+r-1}(X)/B_{r-n}^{s+r,t+r-1}(X).$$
        We then have
        $$d^{\delta}_r(x)=\lambda^{r-n-1}y,$$
        which is trivial if $r>m$.
        \item If $r<n+1$, then we view $\lambda^{n+1-r}x$ as an element of 
        $$E_\infty^{s,t,t-n-1+r}(\nu X/\lambda^n)\iso Z^{s,t}_{r-1}(X)/B^{s,t}_{n+2-r}(X),$$
        and $y$ as an element of 
        $$E_\infty^{s+r,t+r-1, t+r-1}(\nu X)\iso Z_\infty^{s+r,t+r-1}(X).$$
        In this case, we have
        $$d^{\delta}_r(\lambda^{n+1-r}x)=y.$$
    \end{enumerate}
\end{proposition}
\begin{proof}
    Since $\delta=\delta_{n,m}$ is the composition of $\rho$ and $\delta_{n,\infty}$ as the following
     $$\xymatrix{
    \nu X/\lambda^n\ar[d]_-{\delta_{n,\infty}} \ar[rd]^{\delta_{n,m}} & \\
    \Sus^{1,-n}\nu X\ar[r]_-{\rho} & \Sus^{1,-n}\nu X/\lambda^{m-n}
    }$$
    it suffices to prove the case when $m=\infty$ by Corollary \ref{cor:290d35ce}. In the rest of the proof we will write $\delta_n=\delta_{n,\infty}$.

    First, we prove by induction on $n$ that
    \begin{equation}\label{eq:24e1e3f6}
        d^{\delta_n}_r(x)\equiv \lambda^{r-n-1}y\mod B_{r-1}^{s+r,t+r-1}(X)
    \end{equation}
    when $r\ge n+1$, and
    \begin{equation}\label{eq:2f8e024e}
        d^{\delta_n}_r(\lambda^{n+1-r}x)\equiv y\mod B_{r-1}^{s+r,t+r-1}(X)
    \end{equation}
    when $r\le n+1$. (These two expressions coincide when $r=n+1$.) For $n=1$, the claim holds since the $\tau$-Bockstein spectral sequence is isomorphic to the classical Adams spectral sequence. Now, assume $n\ge 2$ and the claim holds for $n-1$.
    
    Consider the following commutative diagram.
    $$\xymatrix{
        \Sus^{0,-1}\nu X/\lambda^{n-1} \ar[rd]_-{\delta_{n-1}} \ar[r]^-\lambda &  \nu X/\lambda^{n}\ar[d]^-{\delta_{n}}\\
        & \Sus^{1,-n}\nu X\\
    }$$
    By Corollary \ref{cor:166dc180}, if $r\le n$, we have
    $$d^{\delta_{n}}_r(\lambda^{n+1-r}x)=d^{\delta_{n-1}}_r(\lambda^{n-r}x)\equiv y\mod B_{r-1}^{s+r,t+r-1}(X).$$
    If $r\ge n+1$, $x$ can be viewed as an element of the $E_\infty$-pages of $\nu X/\lambda^{n-1}$ or $\nu X/\lambda^n$. We then have
    $$\lambda d^{\delta_{n}}_r(x)=d^{\delta_{n}}_r(\lambda x)=d^{\delta_{n-1}}_r(x)\equiv \lambda^{r-n}y \mod B_{r-1}^{s+r,t+r-1}(X)$$
    which implies
    $$d^{\delta_{n}}_r(x)\equiv\lambda^{r-n-1}y\mod B_{r-1}^{s+r,t+r-1}(X),$$
    since $r-n+1\le r-1$ and hence the indeterminacy $B_{r-n+1}$ introduced by dividing $\lambda$ is contained in $B_{r-1}$. The induction for (\ref{eq:24e1e3f6}) and (\ref{eq:2f8e024e}) is now complete.

    We will show that
    $B_{r-1}^{s+r,t+r-1}(X)$ in both equations are actually equal to the sum of images of $d^{\delta_{n}}_0$ through $d^{\delta_{n}}_{r-1}$.
    Consider any $y'\in B_{r-1}^{s+r,t+r-1}(X)$ and assume that $d_{r'}x'=y'$ is an essential classical Adams differential for $2\le r'\le r-1$. If $r\ge n+1$, using (\ref{eq:24e1e3f6}), we have
    $$d^{\delta_n}_{r'}(\lambda^{r-r'}x)\equiv \lambda^{r-n-1}y'\mod B_{r'-1}^{s+r,t+r-1}(X),$$
    and if $r\leq r'+1$, using (\ref{eq:2f8e024e}), we have
    $$d^{\delta_n}_{r'}(\lambda^{n+1-r'}x')\equiv y'\mod B_{r'-1}^{s+r,t+r-1}(X).$$
    Notice that the extra indeterminacy here is $B_{r'-1}$ instead of $B_{r-1}$. By induction on $r$, this shows that $B_{r-1}^{s+r,t+r-1}(X)$ equals the sum of images of $d^{\delta_{n}}_0$ through $d^{\delta_{n}}_{r-1}$.
    
    Therefore, we can omit $B_{r-1}^{s+r,t+r-1}(X)$ and simply write
    $$d^{\delta_n}_r(x)= \lambda^{r-n-1}y$$
    when $r\ge n+1$, and
    $$d^{\delta_n}_r(\lambda^{n+1-r}x)= y$$
    when $r\le n+1$.
\end{proof}

\begin{remark}
    The $d^{\delta}_r(x)$ we calculated in Proposition \ref{prop:6de7d130} may be inessential.
\end{remark}

\begin{corollary}\label{cor:2a636737}
    For $x,y,\delta$ in Proposition \ref{prop:6de7d130} we always have
    \begin{equation}\label{eq:841425e9}
        d^{\delta}_r(\lambda^ax)=\lambda^{a+r-n-1}y
    \end{equation}
    if $0\le a\le n$ and $0\le a+r-n-1<m-n$ (the differential is trivial if $a$ exceeds this range).
\end{corollary}

\begin{remark}\label{rmk:qvoewfj}
    As indicated in the proof, the right-hand side of equation (\ref{eq:841425e9}) (considered as a subset of $Z_\infty^{s+r,t+r-1}(X)$) is a coset of
    $$B_{r-1}^{s+r,t+r-1}(X)$$
    which is the same as the value of the classical Adams differential $d_r(x)=y$. This implies that the equation (\ref{eq:841425e9}) holds if and only if $d_r(x)=y$. Therefore the $\delta$-ESS encodes the same information as the classical Adams spectral sequence.
\end{remark}

\begin{example}
    We continue the discussion from Example~\ref{exam:synEinfty} regarding the implications of the classical Adams differentials in the 14-stem:
    $$d_2(h_4) = h_0h_3^2, \ d_3(h_0h_4) = h_0d_0.$$
    The reader is advised to begin by identifying the elements on the $E_\infty$-pages of the synthetic Adams spectral sequences for $S^{0,0}$ and $S^{0,0}/\lambda^k$ for $k=1,2,3$ in the relevant tridegrees, as illustrated in Example~\ref{exam:synEinfty}.
    
    For $\delta_1: S^{0,0}/\lambda \rightarrow S^{1,-1}$, we have
    $$d_2^{\delta_1}(h_4) = h_0h_3^2, \ d_3^{\delta_1}(h_0h_4) = \lambda h_0d_0.$$
    For $\delta_2: S^{0,0}/\lambda^2 \rightarrow S^{1,-2}$, we have
    $$d_2^{\delta_2}(\lambda h_4) = h_0h_3^2, \ d_3^{\delta_2}(h_0h_4) = h_0d_0, \ d_3^{\delta_2}(\lambda h_0h_4) = \lambda h_0d_0$$
    For $\delta_3: S^{0,0}/\lambda^3 \rightarrow S^{1,-3}$, we have
    $$d_2^{\delta_3}(\lambda^2 h_4) = h_0h_3^2, \ d_3^{\delta_3}(\lambda h_0h_4) = h_0d_0, \ d_3^{\delta_3}(\lambda^2 h_0h_4) =\lambda h_0d_0.$$
\end{example}

\begin{definition} \label{def:classicalcrossdiff}
    Suppose $r\ge n+1$ and $d_r(x)=y$, where
    $$x\in E_2^{s,t}(X), \hspace{0.5cm} y\in E_2^{s+r,t+r-1}(X).$$
    A crossing of $d_r(x)=y$ on the $E_{n+1}$-page refers to an essential Adams differential
    $$d_{r-a-b}(x^\prime)=y^\prime,$$
    where
    $$x^\prime\in E_2^{s+a,t+a}(X), \hspace{.5cm} y^\prime\in E_2^{s+r-b,t+r-b-1}(X),$$
    with $0<a\le n-1$ and $0\le b\le r-n-1$. See Figure \ref{fig:f5eae42a}.
\end{definition}

\begin{remark}
    The crossing defined here is opposite to crossings in Moss's theorem.
\end{remark}
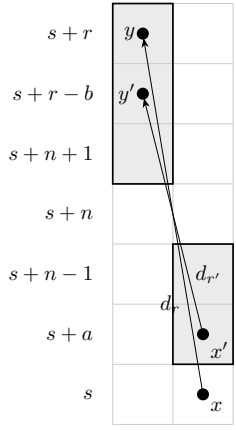
\begin{figure}[h]
\centering
\scalebox{0.8}{\begin{tikzpicture}
    % Draw grid
    \fill[llgray] (1-0.5, 1-0.5) rectangle ++(1, 2);
    \fill[llgray] (0-0.5, 4-0.5) rectangle ++(1, 3);
    \tikzgrid{0}{2}{0}{7}
    \draw[thick] (1-0.5, 1-0.5) rectangle ++(1, 2);
    \draw[thick] (0-0.5, 4-0.5) rectangle ++(1, 3);

    % Draw bullets
    \coordinate (x) at (1, 0);
    \coordinate (y) at (0, 6);
    \coordinate (x1) at (1, 1); 
    \coordinate (y1) at (0, 5);
    
    \fill (x) circle (0.1) node[below right] {$x~$};
    \fill (x1) circle (0.1) node[below right] {$x'$};
    \fill (y) circle (0.1) node[left] {$y$};
    \fill (y1) circle (0.1) node[left] {$y'$};

    % Draw differentials
    \drawarrow (x) -- (y) node[near start, left] {$d_r$};
    \drawarrow (x1) -- (y1) node[near start, right] {$d_{r'}$}; 

    \node[anchor=east] at (-.7, 0) {$s$};
    \node[anchor=east] at (-.7, 1) {$s+a$};
    \node[anchor=east] at (-.7, 2) {$s+n-1$};
    \node[anchor=east] at (-.7, 3) {$s+n$};
    \node[anchor=east] at (-.7, 4) {$s+n+1$};
    \node[anchor=east] at (-.7, 5) {$s+r-b$};
    \node[anchor=east] at (-.7, 6) {$s+r$};
\end{tikzpicture}}
\caption{A crossing of $d_r(x)=y$ on the $E_{n+1}$-page}\label{fig:f5eae42a}
\end{figure}

The emphasis on a crossing occurring ``on the $E_{n+1}$-page" in Definition~\ref{def:classicalcrossdiff} may seem counter-intuitive. However, this is clarified in Proposition~\ref{prop:cross-dr-En} and Example~\ref{exam:classcrossdiff} that follow.

\begin{proposition}\label{prop:cross-dr-En}
    The Adams differential $d_r(x)=y$ has a crossing on the $E_{n+1}$-page if and only if the corresponding $\delta_n$-extension
    $$d^{\delta_n}_r(x)=\lambda^{r-n-1}y$$
    for
    $$\delta_{n}:\nu X/\lambda^{n}\to \Sus^{1,-n}\nu X$$
    has a crossing.
\end{proposition}
\begin{proof}
    By Propositions \ref{prop:30e8b746}, \ref{prop:59f111f} and \ref{prop:6de7d130}, a crossing of $d^{\delta_n}_r(x)=\lambda^{r-n-1}y$ takes the form
    $$d^{\delta_n}_{r-a-b}(\lambda^a x')=\lambda^{r-n-1-b}y',$$
    where $0<a\le n-1$, $0\le b\le r-n-1$, 
    $$\lambda^a x'\in E_\infty^{s+a,t+a,t}(\nu X/\lambda^n)\iso Z_{n-a}^{s+a,t+a}(X)/B_{1+a}^{s+a,t+a}(X),$$
    and
    $$y'\in E_\infty^{s+r-b, t+r-b-1}(\nu X).$$
    By Corollary \ref{cor:2a636737}, we see that this crossing corresponds to the classical Adams differential
    $$d_{r-a-b}(x')=y'.$$
\end{proof}

\begin{remark} \label{nocrossE2}
    From Definition~\ref{def:classicalcrossdiff}, it immediately follows that there are no crossings of any differential on the $E_2$-page, as this would require $0< a \le n-1 = 0$. According to Proposition~\ref{prop:cross-dr-En}, this reflects the fact that $\delta_1$-extensions have no crossings for degree reasons.
\end{remark}

\begin{example} \label{exam:classcrossdiff}
    We examine the following classical Adams differentials from stem 38 to stem 37:
    $$d_3(e_1) = h_1 t, \ d_4(h_0h_3h_5) = h_0^2 x,$$
    where both targets, $h_1 t$ and $h_0^2 x$, reside within the same $(s,t)$-bidegree in $\Ext_A^{7,44}$.

    According to Definition~\ref{def:classicalcrossdiff}, this nonzero $d_3$-differential is a crossing of the $d_4$-differential on both the $E_3$-page and the $E_4$-page, but not on the $E_2$-page or any $E_r$-page for $r \ge 5$.

    This corresponds to the following facts: 
    \begin{enumerate}
        \item For the map $\delta_2: S^{0,0}/\lambda^2 \rightarrow S^{1,-2}$, the differential
        $$d_3^{\delta_2} (\lambda e_1) = \lambda h_1 t$$
        is a crossing for the differential
        $$d_4^{\delta_2} (h_0h_3h_5) = \lambda h_0^2 x.$$

        \item For the map $\delta_3: S^{0,0}/\lambda^3 \rightarrow S^{1,-3}$, the differential
        $$d_3^{\delta_3} (\lambda e_1) = h_1 t$$
        is a crossing for the differential
        $$d_4^{\delta_3} (h_0h_3h_5) = h_0^2 x.$$
        In both cases (1) and (2), the target elements have the same weight. 

        \item In comparison, for the map $\delta_1: S^{0,0}/\lambda \rightarrow S^{1,-1}$, we have the differential
        $$d_4^{\delta_1} (h_0h_3h_5) = \lambda^2 h_0^2 x,$$
        which does not have any crossing differential. The only differential whose target resides in the same $(s,t)$-bidegree is 
        $$d_3^{\delta_1} (e_1) = \lambda h_1 t,$$
        but its target has a different weight. 

        \item For later pages, specifically for $r \ge 4$, and considering the synthetic spectrum $S^{0,0}/\lambda^r$, there is still a crossing for the differential
        $$d_4^{\delta_r} (\lambda^{r-3} h_0h_3h_5) = h_0^2 x.$$
        However, the element $h_0h_3h_5$ itself does not survive to the synthetic Adams $E_\infty$-page, as it supports a nonzero Adams differential:
        $$d_4 (h_0h_3h_5) = \lambda^3 h_0^2 x.$$
    \end{enumerate}
\end{example}

\section{Extensions on a classical $E_r$-page} \label{sec:extenpage}

To state the Generalized Leibniz Rule in terms of the classical Adams spectral sequence, we need to define $f$-extensions not only on homotopy groups but \emph{on the $E_r$-page} as well.

\begin{notation}
For a map between classical spectra $f: X\to Y$, consider the associated synthetic map from Notation~\ref{not:fhat}
$$\hat f: \Sus^{0,e(f)}\nu X\to \nu Y.$$
For any $2\le r\le \infty$, we denote the following mod $\lambda^{r-1}$ reduction maps by $\hat f_{r-1}$:
$$\hat f_{r-1}: \Sus^{0,e(f)}\nu X/\lambda^{r-1}\to \nu Y/\lambda^{r-1}.$$
\end{notation}

The $E_0$-page of the $\hat f_{r-1}$-ESS is isomorphic to
$$\begin{aligned}
    \ESS{\hat f_{r-1}}_0^{s,t,t-k+e(f)}\iso & E_\infty^{s,t,t-k}(\nu X/\lambda^{r-1})\oplus E_\infty^{s,t,t-k+e(f)}(\nu Y/\lambda^{r-1})\\
    \iso & \big(Z_{r-1-k}^{s,t}(X)/B_{1+k}^{s,t}(X)\big)\oplus \big(Z_{r-1-k+e(f)}^{s,t}(Y)/B_{1+k-e(f)}^{s,t}(Y)\big).
\end{aligned}$$
A nontrivial $d^{\hat f_{r-1}}_n$ differential can be interpreted as a map from the subgroup
\begin{equation}\label{eq:eae450fe}
  \ZESS{\hat f_{r-1}}_{n-1}^{s,t,t-k}(X) \subset  Z_{r-1-k}^{s,t}(X)/B_{1+k}^{s,t}(X)
\end{equation}
to the quotient group 
\begin{equation}\label{eq:334248d2}
    \big(Z_{r-1-k-n+e(f)}^{s+n,t+n}(Y)/B_{1+k+n-e(f)}^{s+n,t+n}(Y) \big)/\BESS{\hat f_{r-1}}_{n-1}^{s+n,t+n,t-k}(Y).
\end{equation}
The differential $d^{\hat f_{r-1}}_n$ is trivial for degree reasons when 
$$n<\AF(f)\text{ or }n>r-2-k+e(f).$$

\begin{definition}\label{def:6c076a33}
    Let $x\in Z^{s,t}_{r-1}(X)$ and $y\in Z^{s+n,t+n}_{r-1-n+e(f)}(Y)$ for some 
    $$e(f)\le n\le r-2+e(f).$$
    We say that there is an $(f,E_r)$-extension from $x$ to $y$, denoted by
    \begin{equation}\label{eq:b98be76a}
    d_{n}^{f,E_r}(x)=y
    \end{equation}
    if there exists a synthetic $\hat f_{r-1}$-extension
    \begin{equation}\label{eq:a020bad3}
        d_{n}^{\hat f_{r-1}}(x)=\lambda^{n-e(f)} y.
    \end{equation}
    where $x$ is viewed as an element of the subgroup (\ref{eq:eae450fe}) with $k=0$, and $\lambda^{n-e(f)} y$ is viewed as an element of the quotient group (\ref{eq:334248d2}) with $k=0$. 

   We say that this $(f,E_r)$-extension in (\ref{eq:b98be76a}) is \emph{essential} if the corresponding synthetic $\hat f_{r-1}$-extension in (\ref{eq:a020bad3}) is an essential differential in the $\hat f_{r-1}$-ESS.
    
    For $r=\infty$, we similar define an $(f,E_\infty)$-extension using the corresponding synthetic $\hat f$-extension.
\end{definition}

\begin{figure}[h]
\scalebox{0.85}{\begin{tikzpicture}
    % Draw grid
    \foreach \x in {0,...,5} {
        \draw[lightgray] (\x-0.5, -0.5) -- (\x-0.5, 5.5);
    }
    \foreach \y in {0,...,6} {
        \draw[lightgray] (-0.5, \y-0.5) -- (1.5, \y-0.5);
        \draw[lightgray] (2.5, \y-0.5) -- (4.5, \y-0.5);
    }

    % Draw bullets
    \coordinate (x) at (1, 0);
    \coordinate (dx) at (0, 5);
    \coordinate (y) at (4, 3);
    \coordinate (dy) at (3, 5);
    
    \fill (x) circle (0.1) node[below right] {$x$};
    \fill (y) circle (0.1) node[above right] {$y$};

    % Draw differentials
    \draw[dashed, -{Stealth[length=1.5mm]}, shorten >=(0.08cm)] (x) -- (dx) node[midway,right] {$d_{\ge r}$};
    \draw[dashed, -{Stealth[length=1.5mm]}, shorten >=(0.08cm)] (y) -- (dy) node[midway,right] {$d_{\ge r-n+e(f)}$};
    \drawarrow (x) -- (y) node[midway,below right] {$d^{f,E_r}_{n}$};

    \node[anchor=east] (axis1) at (-0.6, 0) {$s$};
    \node[anchor=east] (axis2) at (-0.6, 3) {$s+n$};
    \node[anchor=east] (axis3) at (-0.6, 5) {$s+r$};

    \node (EX) at (.5, -1.0) {$E_2(X)$};
    \node (EY) at (3.5, -1.0) {$E_2(Y)$};
\end{tikzpicture}
\hspace{0.5cm}
\begin{tikzpicture}
    % Draw grid
    \foreach \x in {0,1,3,4} {
        \draw[lightgray] (\x-0.5, -0.5) -- (\x-0.5, 5.5);
    }
    \foreach \y in {0,...,6} {
        \draw[lightgray] (-0.5, \y-0.5) -- (0.5, \y-0.5);
        \draw[lightgray] (2.5, \y-0.5) -- (3.5, \y-0.5);
    }

    % Draw bullets
    \coordinate (x) at (0, 0);
    \coordinate (y) at (3, 3);
    
    \fill (x) circle (0.1) node[below right] {$x$};
    \fill (y) circle (0.1) node[above right] {$\lambda^{n-e(f)} y$};

    % Draw differentials
    \drawarrow (x) -- (y) node[midway,below right] {$d^{\hat f_{r-1}}_{n}$};

    \node[anchor=east] at (-0.6, 0) {$s$};
    \node[anchor=east] at (-0.6, 3) {$s+n$};
    \node[anchor=east] at (-0.6, 5) {$s+r$};
    
    \node[anchor=east] at (1.5, -1.0) {$E_\infty^{*,t-s+*,t}(\nu X/\lambda^{r-1})$};
    \node[anchor=west] at (1.5, -1.0) {$E_\infty^{*,t-s+*,t}(\nu Y/\lambda^{r-1})$};
\end{tikzpicture}
}
\caption{$(f,E_r)$-extension}\label{fig:3d966cc4}
\end{figure}
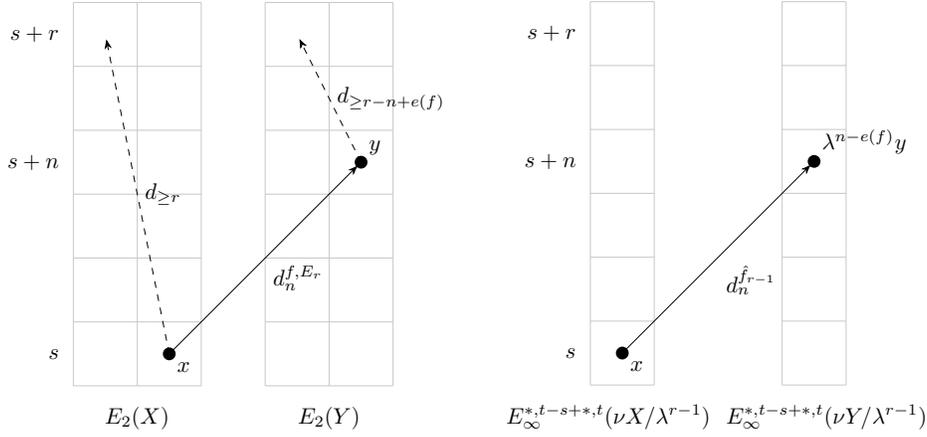

\begin{remark} \label{def:fErextess}
Consider an $(f,E_r)$-extension $d_{n}^{f,E_r}(x)=y$ in (\ref{eq:b98be76a}).
    The element $y$ should be interpreted as a coset with indeterminacy given by $B_{1+n-e(f)}^{s+n,t+n}(Y)$, plus the sum of images of $d^{\hat f_{r-1}}_{<n}$ differentials.   This $(f,E_r)$-extension is \emph{essential} if this coset does not contain 0.
\end{remark}

\begin{example} \label{exam:extonEn}
We present the following examples of $(f, E_r)$-extensions.
    \begin{enumerate}
        \item Consider $f = \nu: S^3 \rightarrow S^0$. We have an $(f, E_2)$-extension:
        $$d_1^{f, E_2} (h_5) = h_2 h_5,$$
        where $h_5 \in \Ext_A^{1,32}$ and $h_2 h_5 \in \Ext_A^{2,36}$.
        
        In fact, the classical homotopy class $\nu \in \pi_3$ is detected by $h_2$ with $\AF=1$ in Ext, and this $(f, E_2)$-extension simply states that, in Ext, the product of $h_5$ by $h_2$ is $h_2 h_5$.

        \item For the same map $f$ as in $(1)$, we have an $(f, E_3)$-extension:
        $$d_2^{f, E_3} (h_0h_4^2) = h_0 p,$$
        where $h_0 h_4^2 \in \Ext_A^{3,33}$ and $h_0 p \in \Ext_A^{5,38}$.
        
        By definition, $\hat{f} = [h_2]$, detected by $h_2$ in tridegree $(s,t,w) = (3,4,4)$ of the synthetic Adams $E_2$-page. This $(f, E_3)$-extension states that for the synthetic map 
        $$\hat{f_{2}} = [h_2]/\lambda^2: S^{3,4}/\lambda^2 \rightarrow S^{0,0}/\lambda^2,$$ 
        there exists a synthetic $\hat{f_{2}}$-extension
        $$d_2^{\hat{f_{2}}} (h_0h_4^2) = \lambda h_0 p.$$
        This is equivalent to the existence of a synthetic homotopy class $[h_0 h_4^2]$ in $\pi_{30+3, 33+4} S^{3,4}/\lambda^2$, detected by $h_0h_4^2$ on the synthetic Adams $E_\infty$-page, that maps to the unique class $[\lambda h_0 p]$ in $\pi_{33, 37} S^{0,0}/\lambda^2$, detected by $\lambda h_0 p$ on the synthetic Adams $E_\infty$-page. In other words, we have
        $$[h_0 h_4^2] \cdot [h_2] = [\lambda h_0 p] \ \text{in} \ \pi_{33, 37} S^{0,0}/\lambda^2.$$

        We will justify this $(f, E_3)$-extension later on in Example~\ref{exam:Mahowald} using the Generalized Mahowald Trick, Theorem~\ref{thm:158d451a}.

        \item For the same map $f$ as in $(1)$, we also have an $(f, E_\infty)$-extension:
        $$d_2^{f, E_\infty} (h_0h_4^2) = h_0 p.$$
        This is equivalent of saying that the relation in $\pi_{33, 37} S^{0,0}/\lambda^2$ from $(2)$, can be lifted to a relation
        $$[h_0 h_4^2] \cdot [h_2] = [\lambda h_0 p] \ \text{in} \ \pi_{33, 37} S^{0,0},$$
        for some classes $[h_0 h_4^2]$ in $\pi_{30+3, 33+4} S^{3,4}$ and $[\lambda h_0 p]$ in $\pi_{33, 37} S^{0,0}$.

        We will justify this $(f, E_\infty)$-extension later on in Example~\ref{exam:stretchext} using Corollary~\ref{cor:dfc6043e}, which allows us to stretch extensions across pages.

        \item Consider $f = 2: S^0 \rightarrow S^0$. We have an $(f, E_\infty)$-extension:
        $$d_2^{f, E_\infty} (h_0 h_3^2) = 0.$$
        By definition, $\hat{f} = [h_0]$, detected by $h_0$ in tridegree $(s,t,w) = (1,1,1)$ of the synthetic Adams $E_2$-page. We remark that there is a relation $\lambda \cdot [h_0] = 2$ in $\pi_{0,0}S^{0,0}$. For the synthetic map $[h_0]:S^{0,1} \rightarrow S^{0,0}$,
        this $(f, E_\infty)$-extension is equivalent to the existence of a synthetic homotopy class $[h_0 h_3^2]$ in $\pi_{14, 17} S^{0,0}$, such that
        $$[h_0 h_3^2] \cdot [h_0] = 0 \ \text{in} \ \pi_{14,18}S^{0,0}.$$
        In fact, by inspection, we learn that $\pi_{14,18}S^{0,0} \cong \mathbb{Z}/2$, generated by $[\lambda h_0 d_0]$, which is an $[h_0]$-multiple of a class $[\lambda d_0]$ in a higher Adams filtration than $h_0h_3^2$. Therefore, there exists a class $[h_0 h_3^2]$ whose $[h_0]$-multiple is zero.

        In other words, we have an essential $(f, E_\infty)$-extension:
        $$d_1^{f, E_\infty} (d_0) = h_0 d_0,$$
        which makes the following alternative $(f, E_\infty)$-extension inessential: 
        $$d_2^{f, E_\infty} (h_0 h_3^2) = h_0 d_0.$$
    \end{enumerate}
\end{example}

\begin{definition}\label{def:41d51149}
    A crossing of the $(f,E_r)$-extension $d_{n}^{f,E_r}(x)=y$ in (\ref{eq:b98be76a}) is defined as an essential $(f, E_{r-a})$-extension from some $x'\in Z_{r-1-a}^{s+a,t+a}(X)$ to 
    $$y'\in Z_{r-1-n+b+e(f)}^{s+n-b,t+n-b}(Y)\backslash B_{1+n-b-e(f)}^{s+n-b,t+n-b}(Y)$$
    (where $\backslash$ denotes the difference of sets and it means that $y'$ should survive to the classical Adams $E_{r-n+b+e(f)}$ page while it should not be hit by an Adams differential of length at most $1+n-b-e(f)$) for $0<a \le r-2$ and $0\le b\le n-a-e(f)$. See Figure \ref{fig:66241bf6}.
\end{definition}

\begin{figure}[h]
\scalebox{0.8}{\begin{tikzpicture}
    % Draw grid
    \tikzgrid{0}{2}{0}{6}
    \tikzgrid{3}{5}{0}{6}

    % Draw bullets
    \coordinate (x) at (1, 0);
    \coordinate (x1) at (1, 1);
    \coordinate (dx) at (0, 5);
    \coordinate (y) at (4, 3);
    \coordinate (y1) at (4, 2);
    \coordinate (dy) at (3, 5);
    
    \fill (x) circle (0.1) node[below right] {$x$};
    \fill (y) circle (0.1) node[above right] {$y$};
    \fill (x1) circle (0.1) node[below right] {$x^\prime$};
    \fill (y1) circle (0.1) node[right] {$y^\prime$};

    % Draw differentials
    \draw[dashed, -{Stealth[length=1.5mm]}, shorten >=(0.08cm)] (x) -- (dx) node[midway,left] {$d_{\ge r}$};
    \draw[dashed, -{Stealth[length=1.5mm]}, shorten >=(0.08cm)] (x1) -- (dx) node[midway,right] {$d_{\ge {r-a}}$};
    \draw[dashed, -{Stealth[length=1.5mm]}, shorten >=(0.08cm)] (y) -- (dy) node[midway,right] {$d_{\ge r-n+e(f)}$};
    \draw[dashed, -{Stealth[length=1.5mm]}, shorten >=(0.08cm)] (y1) -- (dy) node[midway,left] {$d_{\ge r-n+b+e(f)}$};
    \drawarrow (x) -- (y) node[near start,below right] {$d^{f,E_r}_{n}$};
    \drawarrow (x1) -- (y1) node[near start,above] {$d^{f,E_{r-a}}_{n-a-b}$};

    \node[anchor=east] at (-0.6, 0) {$s$};
    \node[anchor=east] at (-0.6, 1) {$s+a$};
    \node[anchor=east] at (-0.6, 2) {$s+n-b$};
    \node[anchor=east] at (-0.6, 3) {$s+n$};
    \node[anchor=east] at (-0.6, 5) {$s+r$};

    \node at (.5, -1.0) {$E_2(X)$};
    \node at (3.5, -1.0) {$E_2(Y)$};
\end{tikzpicture}
\hspace{0.5cm}
\begin{tikzpicture}
    % Draw grid
    \tikzgrid{0}{1}{0}{6}
    \tikzgrid{3}{4}{0}{6}

    % Draw bullets
    \coordinate (x) at (0, 0);
    \coordinate (x1) at (0, 1);
    \coordinate (y) at (3, 3);
    \coordinate (y1) at (3, 2);
    
    \fill (x) circle (0.1) node[below left] {$x$};
    \fill (y) circle (0.1) node[above right] {$\lambda^{n-e(f)} y$};
    \fill (x1) circle (0.1) node[below] {$\lambda^ax^\prime$};
    \fill (y1) circle (0.1) node[above right] {$\lambda^{n-b-e(f)} y^\prime$};

    % Draw differentials
    \drawarrow (x) -- (y) node[near start,below right] {$d^{\hat f_{r-1}}_{n}$};
    \drawarrow (x1) -- (y1) node[near start,above] {$d^{\hat f_{r-1}}_{n-a-b}$};

    \node[anchor=east] at (-0.6, 0) {$s$};
    \node[anchor=east] at (-0.6, 1) {$s+a$};
    \node[anchor=east] at (-0.6, 2) {$s+n-b$};
    \node[anchor=east] at (-0.6, 3) {$s+n$};
    \node[anchor=east] at (-0.6, 5) {$s+r$};
    
    \node[anchor=east] at (1, -1.0) {$E_\infty^{*,t-s+*,t}(\nu X/\lambda^{r-1})$};
    \node[anchor=west] at (1, -1.0) {$E_\infty^{*,t-s+*,t}(\nu Y/\lambda^{r-1})$};
\end{tikzpicture}
}
\caption{A crossing of an $(f,E_r)$-extension}\label{fig:66241bf6}
\end{figure}
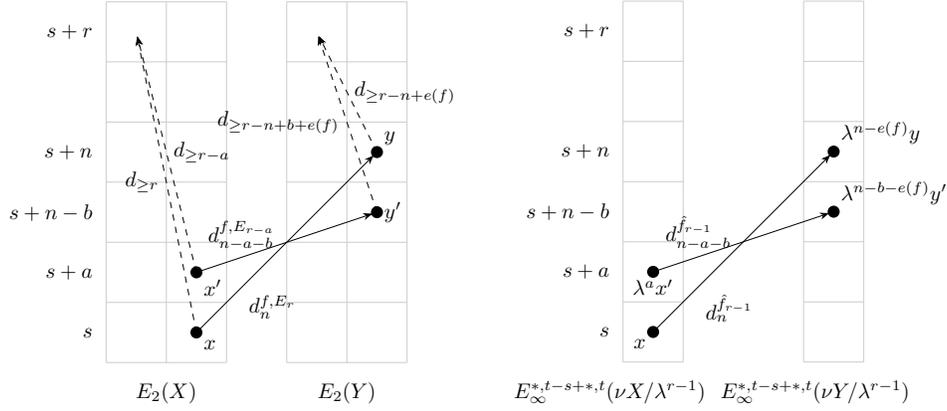

\begin{proposition}\label{prop:cross-f-Er}
    An $(f,E_r)$-extension $d_{n}^{f,E_r}(x)=y$ (\ref{eq:b98be76a}) has a crossing if and only if the synthetic $\hat f_{r-1}$-extension $d_{n}^{\hat f_{r-1}}(x)=\lambda^{n-e(f)} y$ (\ref{eq:a020bad3}) has a crossing.
\end{proposition}

\begin{proof}
    By definition a crossing of the synthetic extension (\ref{eq:a020bad3}) has form
    \begin{equation}\label{eq:f7a7a418}
        d_{n-a-b}^{\hat f_r}(\lambda^ax')=\lambda^{n-b-e(f)}y'
    \end{equation}
    for some $x'\in Z_{r-1-a}^{s+a,t+a}(X)$ and $y'\in Z_{r-1-n+b+e(f)}^{s+n-b,t+n-b}(Y)$. Consider the following commutative diagram
    $$\xymatrix{
        \Sus^{0,e(f)}\nu X/\lambda^{r-1-a} \ar[r]^-{\hat f_{r-1-a}} \ar[d]_{\lambda^a} & \nu Y/\lambda^{r-1-a} \ar[d]^{\lambda^a}\\
        \Sus^{0,e(f)+a}\nu X/\lambda^{r-1}\ar[r]_-{\hat f_{r-1}} & \Sus^{0,a}\nu Y/\lambda^{r-1}\\
    }$$
    We see that (\ref{eq:f7a7a418}) lifts to
    \begin{equation}\label{eq:5742281}
        d_{n-a-b}^{\hat f_{r-1-a}}(x')= \lambda^{n-a-b-e(f)}y''
    \end{equation}
    for some $y''\in Z_{r-1-n+b+e(f)}^{s+n-b,t+n-b}(Y)$ such that
    $$y''\equiv y'\mod B_{1+n-b-e(f)}^{s+n-b,t+n-b}(Y)$$
    where the indeterminacy $B_{1+n-b-e(f)}^{s+n-b,t+n-b}(Y)$ is introduced by dividing $\lambda^a$.
    Therefore, the differential (\ref{eq:f7a7a418}) is essential if and only if the differential (\ref{eq:5742281}) is essential and $y''\notin B_{1+n-b-e(f)}^{s+n-b,t+n-b}(Y)$. This completes the proof.
\end{proof}

The above Definition~\ref{def:41d51149}, and Proposition~\ref{prop:cross-f-Er} can be similarly defined and stated for the case $r=\infty$.  

\begin{example} \label{exam:Ercross}
    For the four $(f,E_r)$-extensions in Example~\ref{exam:extonEn}, we consider their crossings.
    \begin{enumerate}
        \item For $f = \nu: S^3 \rightarrow S^0$, with an $(f, E_2)$-extension:
        $$d_1^{f, E_2} (h_5) = h_2 h_5,$$
        there are no crossings for degree reasons. In fact, since the map $\hat{f}_1 = [h_2]/\lambda$ has $\AF=1$, $d_1^{\hat{f}_1}$-extensions have no crossings.

        \item For $f = \nu: S^3 \rightarrow S^0$, with an $(f, E_3)$-extension:
        $$d_2^{f, E_3} (h_0h_4^2) = h_0 p,$$
        we can verify that it has no crossings from two perspectives.
        
        First, by Definition~\ref{def:41d51149}, a crossing of this $(f, E_3)$-extension would have $n=2$, $e(f)=1$, which implies $a=1, b=0$. This corresponds to an essential $(f, E_2)$-extension from some 
        $$x' \in Z_1^{4,37}(S^3) \cong \Ext_A^{4, 37}(S^3) \cong \mathbb{F}_2,$$
        generated by $h_0^2h_4^2$, to some 
        $$y' \in Z_1^{5,38}(S^0) \cong \Ext_A^{5, 38}(S^0) \cong \mathbb{F}_2,$$
        generated by $h_0p$. However, since in Ext we have
        $$h_0^2 h_4^2 \cdot h_2 = 0 \neq h_0 p,$$
        this $(f, E_3)$-extension has no crossing.

        Alternatively, as seen in Example~\ref{exam:extonEn}(2), this is equivalent to a synthetic $\hat{f_{2}}$-extension
        $$d_2^{\hat{f_{2}}} (h_0h_4^2) = \lambda h_0 p,$$
        where 
        $$\hat{f_{2}} = [h_2]/\lambda^2: S^{3,4}/\lambda^2 \rightarrow S^{0,0}/\lambda^2.$$
        By Proposition~\ref{prop:cross-f-Er}, this $(f, E_3)$-extension has a crossing if and only if the synthetic $\hat{f_{2}}$-extension has a crossing, which is a non-trivial differential of the form
        $$d_1^{\hat{f_{2}}} (x') = y'.$$
        For degree reasons, the only possible candidates are $$x'= \lambda h_0^2 h_4^2, \ y' = \lambda h_0 p.$$ This again contradicts the relation in Ext, confirming that there is no crossing.

        \item For $f = \nu: S^3 \rightarrow S^0$, with an $(f, E_\infty)$-extension:
        $$d_2^{f, E_\infty} (h_0h_4^2) = h_0 p,$$
        we can similarly confirm, using both perspectives outlined in $(2)$, that it has no crossings.

        \item For $f = 2: S^0 \rightarrow S^0$, with an $(f, E_\infty)$-extension:
        $$d_2^{f, E_\infty} (h_0 h_3^2) = 0,$$
        the discussion in Example~\ref{exam:extonEn}(4) can be rephrased to indicate a crossing of the form
        $$d_1^{f, E_\infty} (d_0) = h_0 d_0.$$
        We comment that this $d_1$-differential is also a crossing of the inessential $(f, E_\infty)$-extension:
        $$d_2^{f, E_\infty} (h_0 h_3^2) = h_0 d_0.$$
    \end{enumerate}
\end{example}

\begin{proposition}
    Consider $f: X\to Y$, $x\in Z_\infty^{s,t}(X)$ and $y\in Z_\infty^{s+n,t+n}(Y)$.
    Then
    $$d^{f,E_\infty}_n(x)=y$$
    implies that
    $$d^{f}_n(x+B_\infty^{s,t}(X))=y+B_\infty^{s,t}(Y).$$
\end{proposition}
(The implied differential could be inessential.)
\begin{proof}
  %  By \cite[Proposition 4.40]{Pst} we consider the $\lambda$-inverting functor $\lambda^{-1}:\calS yn\to \calS p$ such that $\lambda^{-1}\circ \nu=id$. The localization map induces a map from the $\hat f$-ESS to the $f$-ESS whose $E_0$ pages are isomorphic to the following
  This follows directly from inverting $\lambda$, which induces a map of spectral sequences from the $\hat f$-ESS to the $f$-ESS. The induced map on the $E_0$-pages are of the following form:
    $$\xymatrix{
    Z_\infty^{s,t}(X)/B_{1+t-w}^{s,t}(X)\ar[r]^-{d^{\hat f}_0}\ar[d]_{\lambda^{-1}} & Z_\infty^{s,t}(Y)/B_{1+t-w-e(f)}^{s,t}(Y)\ar[d]^{\lambda^{-1}}\\
    E_\infty^{s,t}(X)\ar[r]^-{d^{f}_0} & E_\infty^{s, t}(Y)
    }$$
 %   The proposition follows from the naturality of this map between spectral sequences.
\end{proof}
\begin{remark}
    Since the synthetic Adams $E_\infty$-page contains more information than the classical Adams $E_\infty$-page, the $(f,E_\infty)$-extensions similarly provide more information compared to the classical $f$-extensions.
\end{remark}

\section{The Generalized Leibniz Rule and Generalized Mahowald Trick} \label{sec:rules}

Now, we introduce the theorem of the Generalized Leibniz Rule, a valuable tool for computing new Adams differentials.

\begin{theorem}[Generalized Leibniz Rule]\label{thm:e73f481e}
    Let $f: X\to Y$ be a map between two classical spectra.
    Suppose that $2\le n\le r$, $e(f)\le m\le n-2+e(f)$, $l\ge e(f)$, and we have
    $$x\in Z_{r-1}^{s,t}(X), \hspace{0.5cm} y\in Z_{r-1-m+e(f)}^{s+m,t+m}(Y)$$
    $$x_\infty\in Z_\infty^{s+r,t+r-1}(X), \hspace{0.5cm} y_\infty\in Z_\infty^{s+r+l,t+r+l-1}(Y)$$
    and the following conditions hold:
    \begin{enumerate}
        \item $d_{r}(x)=x_\infty$,
        \item $d_{m}^{f,E_n}(x)=y$,
        \item $d_{l}^{f,E_\infty}(x_\infty)=y_\infty$,
        \item the differential in (1) has no crossing on the $E_n$-page or (2) has no crossing.
        \item the differential in (3) has no crossing.% on the $E_\infty$-page.
    \end{enumerate}
    Then we have an Adams differential
    $$d_{r+l-m}(y)= y_\infty.$$
\end{theorem}

\begin{proof}
    Consider the commutative diagram of synthetic spectra
    $$\xymatrix{
        \Sus^{0,e(f)}\nu X/\lambda^{n-1} \ar[rr]^{\hat f_{n-1}}\ar[d]_{\delta_X} && \nu Y/\lambda^{n-1}\ar[d]_{\delta_Y}\\
        \Sus^{1,-n+1+e(f)} \nu X \ar[rr]_{\hat f} && \Sus^{1,-n+1}\nu Y.
    }$$
    By condition (1) and Proposition \ref{prop:6de7d130}, we have
    \begin{equation}\label{eq:ca3fdd92}
        d^{\delta_X}_r(x)=\lambda^{r-n}x_\infty.
    \end{equation}
    By conditions (2) and (3), and Definition \ref{def:6c076a33}, we have
    \begin{equation}\label{eq:b2aa5032}
        d^{\hat f_{n-1}}_m(x)=\lambda^{m-e(f)} y
    \end{equation}
    and
    \begin{equation*}
        d^{\hat f}_l(x_\infty)=\lambda^{l-e(f)} y_\infty, 
    \end{equation*}
    which implies
    \begin{equation}\label{eq:886aee9b}
         d^{\hat f}_l(\lambda^{r-n}x_\infty)=\lambda^{r+l-n-e(f)}y_\infty.
    \end{equation}
    
    Applying Proposition \ref{prop:cross-dr-En} and Proposition \ref{prop:cross-f-Er} to conditions (4) and (5), we know that the differential in (\ref{eq:ca3fdd92}) or (\ref{eq:b2aa5032}) has no crossing, and that the differential in (\ref{eq:886aee9b}) has no crossing.
    
    Therefore, we can apply Corollary \ref{cor:0012nik} and obtain
    $$d^{\delta_Y}_{r+l-m}(\lambda^{m-e(f)} y)=\lambda^{r+l-n-e(f)} y_\infty.$$
    By Remark \ref{rmk:qvoewfj}, this is equivalent to
    $$d_{r+l-m}(y)= y_\infty.$$
\end{proof}

\begin{remark}
    We can further generalize Theorem \ref{thm:e73f481e} by using the conditions in Theorem \ref{thm:4114f70c} rather than those Corollary \ref{cor:0012nik}. We leave this generalization to the reader.
\end{remark}

We emphasize that the no-crossing conditions in Theorem \ref{thm:e73f481e}, the Generalized Leibniz Rule, are crucial. We first demonstrate the strength of the Generalized Leibniz Rule with Example~\ref{exam:Leibnizyes}, which uses a $d_2$-differential to prove a $d_3$-differential in the classical Adams spectral sequence. Additionally, Example~\ref{exam:Leibnizno} shows that, without the no-crossing condition, the conclusion could be false.

\begin{remark}
    A version of the Generalized Leibniz Rule without the no-crossing conditions is presented in the synthetic setting in Chua's work \cite[Theorem~12.9]{Chua}. However, there is no doubt that this version is incorrect. Indeed, Example~\ref{exam:Leibnizno} is a counter-example to Chua's theorem. For further details, see Remark~\ref{rem:chuaerror}.
\end{remark}

\begin{example} \label{exam:Leibnizyes}
 We show that the Generalized Leibniz Rule can be used to prove the following classical Adams differential
    $$d_3(h_2 h_5) = h_0 p,$$
    using the classical Hopf invariant one differential
    $$d_2(h_5) = h_0 h_4^2,$$
    and the $(f, E_3)$-extension
        $$d_2^{f, E_3} (h_0h_4^2) = h_0 p,$$
        for the map $f = \nu: S^3 \rightarrow S^0$ from Example~\ref{exam:extonEn}.

        Specifically, we have:
        $$n=2, \ r=2, \ s=1, \ t=35, \ m=1, \ l=2, \ e(f)=1,$$
        and 
        $$x=h_5 \in \Ext_A^{1,32} \cong Z_1^{1,35}(S^3), \ y=h_2h_5 \in \Ext_A^{2,36} \cong Z_1^{2,36}(S^0),$$
        $$x_\infty=h_0h_4^2 \in \Ext_A^{3,33} \cong Z_\infty^{3,36}(S^3), \ y_\infty=h_0p \in \Ext_A^{5,38} \cong Z_\infty^{5,38}(S^0).$$
In terms of the conditions in Theorem \ref{thm:e73f481e}, we have:
      \begin{enumerate}
      \item $d_2(h_5) = h_0h_4^2$,
      \item $d_1^{f, E_2}(h_5) = h_2h_5$, from Example~\ref{exam:extonEn}(1),
      \item $d_2^{f, E_\infty} (h_0h_4^2) = h_0 p$, from Example~\ref{exam:extonEn}(3),
      \item 
      \begin{enumerate}
          \item the differential in (1) has no crossing, from Remark~\ref{nocrossE2}, and
          \item the differential in (2) has no crossing, from Example~\ref{exam:Ercross}(1), 
      \end{enumerate}
      \item the differential in (3) has no crossing, from Example~\ref{exam:Ercross}(3).
      \end{enumerate}
      Since all conditions of Theorem \ref{thm:e73f481e} are satisfied, we conclude that there is a classical Adams differential
      $$d_3(h_2 h_5) = h_0 p.$$
\end{example}

\begin{example} \label{exam:Leibnizno}
    We show that, without the no crossing conditions, the Generalized Leibniz Rule could lead to incorrect conclusions. 

    Consider the map $f=2: S^0 \rightarrow S^0$, and the classical Hopf invariant one differential
    $$d_2(h_4) = h_0 h_3^2.$$
    Set $$n=2, \ r=2, \ s=1, \ t=16, \ m=1, \ l=2, \ e(f)=1,$$
        and 
        $$x=h_4 \in \Ext_A^{1,16} \cong Z_1^{1,16}(S^0), \ y=h_0h_4 \in \Ext_A^{2,17} \cong Z_1^{2,17}(S^0),$$
        $$x_\infty=h_0h_3^2 \in \Ext_A^{3,17} \cong Z_\infty^{3,17}(S^0), \ y_\infty=0 \in \Ext_A^{5,19} \cong Z_\infty^{5,19}(S^0).$$
In terms of the conditions in Theorem \ref{thm:e73f481e}, we have:
 \begin{enumerate}
      \item $d_2(h_4) = h_0h_3^2$,
      \item $d_1^{f, E_2}(h_4) = h_0h_4$, since the product of $h_4$ by $h_0$ is $h_0h_4$ in Ext,
      \item $d_2^{f, E_\infty} (h_0h_3^2) = 0$, from Example~\ref{exam:extonEn}(4),
      \item 
      \begin{enumerate}
          \item the differential in (1) has no crossing, from Remark~\ref{nocrossE2}, and
          \item the differential in (2) has no crossing, for degree reasons, 
      \end{enumerate}
      \item the differential in (3) has a crossing:
      $$d_1^{f, E_\infty} (d_0) = h_0 d_0,$$
      from Example~\ref{exam:Ercross}(4).
      \end{enumerate}
      For condition $(5)$, if there were no crossing for the differential in $(3)$, then the Generalized Leibniz Rule would imply a classical Adams differential
      $$d_3(h_0h_4) = 0,$$
      which is incorrect. This example illustrates that the no-crossing conditions are essential when applying the Generalized Leibniz Rule.
\end{example}

\begin{remark}
    As noted in Example~\ref{exam:extonEn}(4) and Example~\ref{exam:Ercross}(4), the essential $(f, E_\infty)$-extension:
        $$d_1^{f, E_\infty} (d_0) = h_0 d_0,$$
        is a crossing for both the $(f, E_\infty)$-extension,
        $$d_2^{f, E_\infty} (h_0 h_3^2) = 0,$$
        and the inessential $(f, E_\infty)$-extension: 
        $$d_2^{f, E_\infty} (h_0 h_3^2) = h_0 d_0.$$
        This indicates that, in Example~\ref{exam:Leibnizno}, if we were to disregard the no-crossing condition (5), and apply the Generalized Leibniz Rule to these two cases of the $(f, E_\infty)$-extensions, we would arrive at two conflicting classical statements:
$$d_3(h_0h_4) = 0, \ d_3(h_0h_4) = h_0 d_0.$$
\end{remark}

\begin{remark} \label{rem:chuaerror}
    In Chua's work \cite[Theorem~12.9]{Chua}, it is stated that for a synthetic map $\alpha: X\rightarrow Y$, and an element $x \in \pi_{*,*}X/\lambda$, there exists a differential from a maximal $\alpha$-extension of $x$ to a maximal $\alpha$-extension of $d_r(x)$. According to  \cite[Definition~12.5]{Chua}, a maximal $\alpha$-extension of $x'$ is defined as $\alpha[x']$, where $[x']$ represents a lift of $x'$ to the $E_{r'}$-page, chosen such that $\alpha[x']$ is the most $\lambda$-divisible among all such lifts. 

    In the context of Example~\ref{exam:Leibnizno}, let $r=2, r'=\infty$,
    $$\alpha = [h_0]: S^{0,1} \rightarrow S^{0,0},$$
    and consider $x = h_4$ in Ext, with $x' = d_2(x) = h_0h_3^2$.

    As noted in Example~\ref{exam:extonEn}(4), the $(f, E_\infty)$-extension:
        $$d_2^{f, E_\infty} (h_0 h_3^2) = 0$$
         is equivalent to the existence of a synthetic homotopy class $[h_0 h_3^2]$ in $\pi_{14, 17} S^{0,0}$, such that
        $$[h_0 h_3^2] \cdot [h_0] = 0 \ \text{in} \ \pi_{14,18}S^{0,0}.$$
        Similarly, the inessential $(f, E_\infty)$-extension: 
        $$d_2^{f, E_\infty} (h_0 h_3^2) = h_0 d_0$$
        implies the existence of another synthetic homotopy class $[h_0 h_3^2]$ in $\pi_{14, 17} S^{0,0}$, such that
        $$[h_0 h_3^2] \cdot [h_0] = [h_0d_0] \ \text{in} \ \pi_{14,18}S^{0,0}.$$
        Between these two lifts of $h_0h_3^2$, the first $[h_0h_3^2]$ is clearly the maximal $[h_0]$-extension according to Chua’s definition \cite[Definition~12.5]{Chua}. Consequently, the incorrect version of the Generalized Leibniz Rule in \cite[Theorem~12.9]{Chua}, would lead to an incorrect conclusion:
        $$d_3(h_0h_4) = 0.$$
\end{remark}

Next, we discuss the Generalized Mahowald Trick. 

In order to apply the Generalized Leibniz Rule, we need to provide a method for computing extensions on specific Adams $E_k$-pages, such as the $(f, E_3)$-extension $d_2^{f, E_3} (h_0h_4^2) = h_0 p$ in Example~\ref{exam:Leibnizyes}. This is provided by Theorem~\ref{thm:158d451a} (the Generalized Mahowald Trick). The crux of the proof of the Generalized Mahowald Trick lies in the following lemma.
\begin{lemma}[May \cite{May01}]\label{lem:452d218c}
    Let $X\to Y\to Z\to \Sus X$ and $X'\to Y'\to Z'\to \Sus X'$ be distinguished triangles of (synthetic) spectra. By smashing these distinguished triangles together, we obtain the following commutative diagram of cofiber sequences:
    $$\xymatrix{
    X\sma X' \ar[r] \ar[d] & Y\sma X' \ar[r] \ar[d] & Z\sma X' \ar[d]\\
    X\sma Y' \ar[d] \ar[r] & Y\sma Y' \ar[d] \ar[r] & Z\sma Y' \ar[d]\\
    X\sma Z' \ar[r] & Y\sma Z' \ar[r] & Z\sma Z'\\
    }$$
    If $a\in \pi_n(X\sma Z')$ and $b\in\pi_n(Y\sma Y')$ map to the same element in $\pi_n(Y\sma Z')$, then there exists $c\in \pi_n(Z\sma X')$ such that
    \begin{enumerate}
        \item $b$ and $c$ map to the same element in $\pi_n(Z\sma Y')$, and
        \item $a$ and $c$ map to the same element via boundary maps in $\pi_{n-1}(X\sma X')$.
    \end{enumerate}
\end{lemma}

\begin{proof}
The proof follows directly from \cite[Lemma 4.6]{May01}.

In fact, consider $V$ in Axiom TC3 of \cite[Section 4]{May01} and the corresponding commutative diagram. By \cite[Lemma 4.6]{May01} we know that $V$ is the pull back of the following diagram.
$$\xymatrix{
\ar@{}[dr]|\lrcorner V\ar[d]\ar[r]& Y\sma Y'\ar[d]\\
X\sma Z'\ar[r] & Y\sma Z'
}$$
Since $a$ and $b$ map to the same element in $\pi_n(Y\sma Z')$, we know that we can find $v\in V$ such that $v$ maps to $a$ in $\pi_n(X\sma Z')$ and $b$ in $\pi_n(Y\sma Y')$. Then we let $c=j_3(v)\in \pi_n(Z\sma X')$, where $j_3: V \rightarrow Z \sma X'$ is the map in the commutative diagram in Axiom TC3 of \cite[Section 4]{May01}. This lemma follows from the commutativity of the diagram.
\end{proof}
 
\begin{theorem}[Generalized Mahowald Trick]\label{thm:158d451a} 
    Consider a distinguished triangle of spectra
    $$X\fto{f} Y\fto{g} Z\fto{h} \Sus X$$
    with $e(f)+e(g)+e(h)=1$.
    Suppose that $r=n+m+l$, $n_1=n-e(f)\ge 1$, $m_1=m-e(g)\ge 0$, $l_1=l-e(h)\ge 0$, and
    $$\begin{array}{ll}
        x\in Z_{n_1}^{s+l,t+l-1}(X), & y\in Z_{m_1+1}^{s+n+l,t+n+l-1}(Y),\\
        \bar x\in Z_{r-1}^{s,t}(Z), & \bar y\in Z_\infty^{s+r,t+r-1}(Z),
    \end{array}$$
    such that
    \begin{enumerate}
        \item $d^{h,E_{r'}}_{l}\bar x=x$, where $r'=r-m_1=n_1+l_1+1$,
        \item $d_r \bar x=\bar y$,
            \item the $(h,E_{r'})$-extension in (1) has no crossing, or the Adams differential (2) has no crossing on the $E_{r'}$-page,
        \item $d^{g,E_{m_1+2}}_{m}y=\bar y$.
    \end{enumerate}
    Then we have $x\in Z_{n+m+e(h)}^{s+l,t+l-1}(X)$ and
    $$d^{f,E_{n+m+1+e(h)}}_{n}x\equiv y \mod B_{r'}^{s+n+l,t+n+l-1}(Y).$$
    (See Figure \ref{fig:2aa45b21}.)
\end{theorem}
\begin{figure}[h]
\centering
\scalebox{0.9}{\begin{tikzpicture}[node/.style={circle,fill=black!0.1}]
    \def\ymax{6}
    \def\ymaxp{7}
    
    % Draw grid
    \foreach \x in {0,...,8} {
        \draw[lightgray] (\x-0.5, -0.5) -- (\x-0.5, \ymax+0.5);
    }
    \foreach \x in {0,2,4,5,7} {
        \foreach \y in {0,...,\ymaxp} {
            \draw[lightgray] (\x-0.5, \y-0.5) -- (\x+0.5, \y-0.5);
        }
    }

    % Draw bullets
    \coordinate (x) at (0, 1);
    \coordinate (y) at (2, 5);
    \coordinate (y1) at (4, 6);
    \coordinate (x1) at (5, 0);
    \coordinate (sx) at (7, 1);
    
    \fill (x) circle (0.1) node[below left] {$x$};
    \fill (y) circle (0.1) node[below right] {$y$};
    \fill (x1) circle (0.1) node[below left] {$\bar x$};
    \fill (y1) circle (0.1) node[above right] {$\bar y$};
    \fill (sx) circle (0.1) node[above right] {$\Sus x$};

    % Draw differentials
    \draw[dashed, -{Stealth[length=1.5mm]}, shorten >=(0.08cm] (x) -- (y) node[near start,below right] {$d^{f, E_{n+m+e(h)+1}}_n(\mathrm{mod}~B_{r'})$};
    \drawarrow (y) -- (y1) node[midway,left] {$d^{g,E_{m-e(g)+2}}_{m}$};
    \drawarrow (x1) -- (y1) node[midway,right] {$d_r$};
    \drawarrow (x1) -- (sx) node[midway,below] {~\hspace{2em}$d^{h,E_{r'}}_{l}$};

    \node[anchor=east] at (-.7, 0) {$s$};
    \node[anchor=east] at (-.7, 1) {$s+l$};
    \node[anchor=east] at (-.7, 5) {$s+l+n$};
    \node[anchor=east] at (-.7, 6) {$s+r$};

    \node at (0, -1.0) {$E_2(X)$};
    \node at (2, -1.0) {$E_2(Y)$};
    \node at (4.5, -1.0) {$E_2(Z)$};
    \node at (7, -1.0) {$E_2(\Sus X)$};
\end{tikzpicture}}
\caption{A demonstration of Theorem \ref{thm:158d451a}}\label{fig:2aa45b21}
\end{figure}
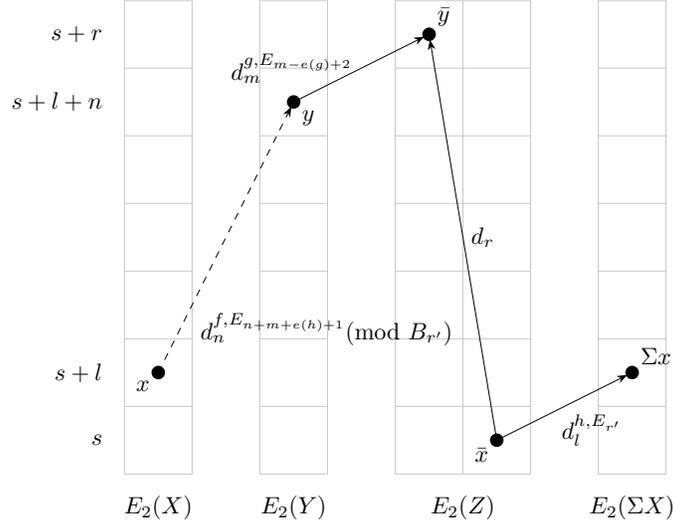

\begin{figure}
\centering
    \scalebox{0.85}{$\xymatrix@C=1em{
        \bigwedge & \nu X \ar[r]^-{\hat f} & \Sus^{0,-\!e(f)}\nu Y \ar[r]^-{\hat g} &  \Sus^{0,-\!e(f)\!-\!e(g)}\nu Z \ar[r]^-{\hat h} & \Sus^{1,0}\nu X\\
        S^{0,0}/\lambda^{r} \ar[d]_{\rho} & & & & [\lambda^{l_1}x] \ar[d]\\
        S^{0,0}/\lambda^{r'-1} \ar[d]_{\delta_{r'-1}} & & & [\bar x] \ar[r] \ar[d] & [\lambda^{l_1}x]\\
        S^{1,-r'+1}/\lambda^{m_1+1} \ar[d]_{\lambda^{r'-1}} & & [y] \ar[r] \ar[d] & [\lambda^{m_1}\bar y] & \\
        S^{1,0}/\lambda^{r} & [\lambda^{l_1}x] \ar[r] & [\lambda^{r'-1}y] & & \\
    }$}
\caption{Elements in the homotopy groups of the smash products}\label{fig:2d013742} 
\end{figure}

\begin{proof}
    Consider the following two distinguished triangles of synthetic spectra
    \begin{gather*}
        \nu X\fto{\hat f} \Sus^{0,-e(f)}\nu Y\fto{\hat g} \Sus^{0,-e(f)-e(g)}\nu Z\fto{\hat h} \Sus^{0,-1}\nu \Sus X=\Sus^{1,0}\nu X\\
        S^{0,0}/\lambda^{r}\fto{\rho} S^{0,0}/\lambda^{r'-1} \fto{\delta_{r'-1}} S^{1,-r'+1}/\lambda^{m_1+1} \fto{\lambda^{r'-1}} S^{1,0}/\lambda^{r}
    \end{gather*}
    and their smash product. See Figure \ref{fig:2d013742}.

    By condition (1), we have
    \begin{equation}\label{eq:70327019}
        d^{\hat h_{r'-1}}_l(\bar x)=\lambda^{l_1}x
    \end{equation}
    where $\hat h_{r'-1}$ is the map $\Sus^{0,e(h)}\nu Z/\lambda^{r'-1}\to \nu X/\lambda^{r'-1}$ induced by $\hat h$.
    By condition (2), we have
    \begin{equation}\label{eq:e129d90c}
        d^{\delta_{r'-1}}_r(\bar x)=\lambda^{m_1}\bar y.
    \end{equation}

    Applying Proposition \ref{prop:cross-dr-En} and Proposition \ref{prop:cross-f-Er} to condition (3), we know that the differential in (\ref{eq:70327019}) or the differential in (\ref{eq:e129d90c}) has no crossing. Hence, there exists
    $$[\bar x]\in \{\bar x\}\subset \pi_{t-s,t}(\nu Z/\lambda^{r'-1})$$
    such that
    $$\hat h_{r'-1}([\bar x])\in \{\lambda^{l_1}x\}\subset \pi_{t-s-1, t-1+e(h)}(\nu X/\lambda^{r'-1})$$
    and
    $$\delta_{r'-1}([\bar x])\in \{\lambda^{m_1}\bar y\}\subset \pi_{t-s-1,t+r'-1}(\nu Z/\lambda^{m_1+1}).$$

    By condition (4), we have
    \begin{equation*}
        d^{\hat g_{m_1+1}}_m(y)=\lambda^{m_1}\bar y.
    \end{equation*}
    This implies that there exists
    $$[y]\in \{y\}\subset \pi_{t-s-1,t+n+l-1}(\nu Y/\lambda^{m_1+1})$$
    such that
    $$\hat g_{m_1+1}([y])\in \{\lambda^{m_1}\bar y\}\subset \pi_{t-s-1, t+r'-1}(\nu Z/\lambda^{m_1+1}),$$
    where $\hat g_{m_1+1}$ is the map $\Sus^{0,e(g)}\nu Y/\lambda^{m_1+1}\to \nu Z/\lambda^{m_1+1}$ induced by $\hat g$.
    
    Due to degree reasons (Proposition \ref{prop:59f111f}), $\lambda^{m_1}\bar y$ detects a unique element in homotopy,
    $$[\lambda^{m_1}\bar y]\in \pi_{t-s-1, t+r'-1}(\nu Z)/\lambda^{m_1+1}.$$
    Therefore, we have
    $$\delta_{r'-1}([\bar x])=\hat g_{m_1+1}([y])=[\lambda^{m_1}\bar y],$$
    By Lemma \ref{lem:452d218c}, there exists
    $$[\lambda^{l_1}x]\in \pi_{t-s-1,t-1+e(h)}(\nu X/\lambda^r),$$
    such that
    \begin{gather}
        \rho ([\lambda^{l_1}x])=\hat h_{r'-1}([\bar x]),\label{eq:f4848f29}\\
        \hat f_r([\lambda^{l_1}x])=\lambda^{r'-1}[y],\label{eq:fba35e4e}
    \end{gather}
    where $\rho$ is the map $\nu X/\lambda^r\to \nu X/\lambda^{r'-1}$, and $\hat f_r$ is the map $\Sus^{0,e(f)}\nu X/\lambda^r\to \nu Y/\lambda^r$ induced by $\hat f$.
    
    The equality in (\ref{eq:f4848f29}) indicates that $[\lambda^{l_1}x]$ can be lifted to the $E_\infty$-page of $\nu X/\lambda^r$, so we have
    $$x\in Z_{n+m+e(h)}^{s+l,t+l-1}(X).$$
    The equation in (\ref{eq:fba35e4e}) indicates that
    $$d^{\hat f_r}_n(\lambda^{l_1}x)=\lambda^{r'-1}y.$$
    Therefore, by dividing $\lambda^{l_1}$, we have
    $$d^{\hat f_{n+m+e(h)}}_n(x)\equiv \lambda^{n_1}y \mod B_{r'}^{s+n+l,t+n+l-1}(Y)$$
    for $x$ in the $E_\infty$ page of $\nu X/\lambda^{r-l_1}=\nu X/\lambda^{n+m+e(h)}$. By definition, this is equivalent to
    $$d^{f, E_{n+m+1+e(h)}}_nx\equiv y \mod B_{r'}^{s+n+l,t+n+l-1}(Y).$$
\end{proof}

    We refer to Theorem~\ref{thm:158d451a} as the Generalized Mahowald Trick, as this approach was first utilized by Mahowald and his collaborators in various works (see, for example, \cite{MahowaldTangora}), particularly in the case where $m_1 = l_1 = 0$. The synthetic setting advances this method by allowing for the consideration of cases where $m_1>0, l_1>0$ as well. 
    
\begin{remark}
    Several versions of classical generalizations of Mahowald's original trick appear in the literature and are often referred to as geometric boundary theorems. Notable examples include the works of Behrens \cite{Behrens}, and Ma \cite{Ma}.
\end{remark}

\begin{remark}
  A synthetic generalization of the Mahowald Trick, again lacking any no-crossing conditions, is also presented in Chua's work \cite[Theorem~12.11]{Chua}. This version is also incorrect for similar reasons.
\end{remark}

\begin{example} \label{exam:Mahowald}
    We apply the Generalized Mahowald Trick to prove the claim in Example~\ref{exam:extonEn}(2) regarding an $(f, E_3)$-extension:
        $$d_2^{f, E_3} (h_0h_4^2) = h_0 p,$$
        for the map $f=\nu:S^3 \rightarrow S^0$.

        Specifically, let $X = S^3, Y = S^0, Z = S^0/\nu$, where $Z$ is the cofiber of $\nu \in \pi_3$, we have a distinguished triangle
        $$\xymatrix{S^3 \ar[r]^{\nu} & S^0 \ar[r]^-{i} & S^0/\nu \ar[r]^{q} & S^4.}$$
We set
$$f=\nu, \ g=i, \ h=q, \ e(f)=1, \ e(g)=0, \ e(h)=0.$$
The short exact sequence on $\HF$-homology
$$\xymatrix{0 \ar[r] & {\HF}_*S^0 \ar[r]^-{i_*} & {\HF}_*S^0/\nu \ar[r]^-{q_*} & {\HF}_*S^4 \ar[r] & 0,}$$
induces a long exact sequence on Ext-groups
$$\xymatrix{\cdots \ar[r]^-{\cdot h_2} & {\Ext}_A^{*,*}(S^0) \ar[r]^-{i_*} & {\Ext}_A^{*,*}(S^0/\nu) \ar[r]^-{q_*} & {\Ext}_A^{*,*}(S^4) \ar[r]^-{\cdot h_2} & \cdots}.$$
We use Atiyah--Hirzebruch notation (as in \cite[Notation 3.3]{WX61, WangXu51ext}) to denote elements in ${\Ext}_A^{*,*}(S^0/\nu)$. Specifically, if an element $x$ in ${\Ext}_A^{*,*}(S^0/\nu)$ satisfies 
$$q_*(x) = a \neq 0 \in \Ext_A^{*,*}(S^4),$$ 
we denote $x$ by $a[4]$; otherwise, due to exactness, 
$$x = i_*(b) \ \text{for some} \ b \in \Ext_A^{*,*}(S^0),$$
and in this case, we denote $x$ by $b[0]$. 

We have
$$m=0, \ l=0, \ n=2, \ r=2, \ r'=2, \ n_1=1, \ m_1=0, \ l_1=0,$$
and

% \begin{equation*}
\begin{align*}
x & =h_0h_4^2 && \in \Ext_A^{3,33} && \cong Z_1^{3,36}(S^3),\\
\bar{x} & =h_0h_4^2[4] && \in \Ext_A^{3,37}(S^0/\nu) && \cong Z_1^{3,37}(S^0/\nu), \\
y & =h_0p && \in \Ext_A^{5,38} && \cong Z_1^{5,38}(S^0),\\
        \bar{y} & =h_0p[0] && \in \Ext_A^{5,38}(S^0/\nu) && \cong Z_\infty^{5,38}(S^0/\nu). 
  \end{align*}      
%  \end{equation*}
For conditions in Theorem~\ref{thm:158d451a}, we have:
\begin{enumerate}
      \item $d_0^{h, E_2}(h_0h_4^2[4]) = h_0h_4^2$, since in Ext $h_2 \cdot h_0h_4^2 = 0$.
      \item $d_2(h_0h_4^2[4]) = h_0p[0]$. This is a classical Adams $d_2$-differential for $S^0/\nu$, and can be computed using Lin's computer program via the method of the secondary Steenrod algebra \cite{BJ, Nassau, Chua}.
      \item 
      \begin{enumerate}
          \item The differential in (1) has no crossing, as it is a $d_0$-differential, and
          \item the differential in (2) has no crossing, from Remark~\ref{nocrossE2}. 
      \end{enumerate}
      \item 
      $d_0^{g, E_2} (h_0p) = h_0p[0]$, as $h_0p$ is not divisible by $h_2$ in Ext.
      \end{enumerate}
    Since all conditions of  Theorem~\ref{thm:158d451a} are  satisfied, we conclude that $x=h_0h_4^2$ survives to the $E_3$-page, and there is an $(f, E_3)$-extension:
        $$d_2^{f, E_3} (h_0h_4^2) = h_0 p,$$
        as promised.
        
\end{example}

The outcome of the Generalized Mahowald Trick, as stated in Theorem~\ref{thm:158d451a}, is an $f$-extension on a specific Adams page, such as the $(f, E_3)$-extension $d_2^{f, E_3} (h_0h_4^2) = h_0 p$ in Example~\ref{exam:extonEn}(2) and Example~\ref{exam:Mahowald}. In practice, the source of an $(f, E_r)$-extension may survive to later Adams pages, prompting interest in the $(f, E_{>r})$-extensions, such as the $(f, E_\infty)$-extension $d_2^{f, E_\infty} (h_0h_4^2) = h_0 p$ in Example~\ref{exam:extonEn}(3), which is applied in Example~\ref{exam:Leibnizyes}. The following Propositions~\ref{prop:dec738d3} and \ref{cor:dfc6043e} describe the relationships between extensions across different pages.

\begin{proposition}\label{prop:dec738d3}
    Suppose that we have an $(f,E_r)$-extension $d^{f,E_r}_n(x)=y$, where $x\in Z_{r-1}^{s,t}(X)$ and $y\in Z_{r-1-n+e(f)}^{s+n,t+n}(Y)$. Then for all $2\le r'<r$, we also have
    \begin{equation}\label{eq:dr'}
        d^{f,E_{r'}}_n(x)=y.
    \end{equation}

    Furthermore, if $d^{f,E_r}_n(x)=y$ is essential and $n\le r'-2+e(f)$, then (\ref{eq:dr'}) is inessential if and only if there exists some $0<a'\le n-e(f)$ and an element
    $$x'\in Z_{r'-1-a'}^{s+a',t+a'}(X)\backslash Z_{r-1-a'}^{s+a',t+a'}(X)$$
    such that
    $$d^{\hat f_{r'-1}}_{n-a'-b}(\lambda^{a'} x')=\lambda^{n-b-e(f)} y'$$
    for some $b\ge 0$.
    
\end{proposition}

\begin{proof}
    Consider the following commutative diagram:
    $$\xymatrix{
        \Sus^{0,e(f)}\nu X/\lambda^{r-1} \ar[r]^-{\hat f_{r-1}}\ar[d]_{\rho_X} & \nu Y/\lambda^{r-1} \ar[d]_{\rho_Y}\\
        \Sus^{0,e(f)}\nu X/\lambda^{r'-1} \ar[r]^-{\hat f_{r'-1}} & \nu Y/\lambda^{r'-1}
    }$$
    By Corollary \ref{cor:e7b20ae2}, $(\rho_X,\rho_Y)$ induces a map from the $\hat f_{r-1}$-ESS to the $\hat f_{r'-1}$-ESS. Therefore, by naturality,
    $$d^{\hat f_{r-1}}_n(x)=\lambda^{n-e(f)} y\ \text{implies} \ d^{\hat f_{r'-1}}_n(x)=\lambda^{n-e(f)} y,$$
    which, by definition, is
    $$d^{f,E_{r'}}_n(x)=y.$$
   
    Next, we prove the second part of the proposition. By Definition \ref{def:6c076a33}, the $(f,E_r)$-extension (\ref{eq:dr'}) is inessential if and only if 
    there exists $0<a\le n-e(f)$ and 
    $$\lambda^a x'\in E_\infty^{s+a,t+a,t}(\nu X/\lambda^{r'-1})\iso Z_{r'-1-a}^{s+a,t+a}(X)/B_{1+a}^{s+a,t+a}(X),$$
    such that
    \begin{equation}\label{eq:dcd62a19}
        d^{\hat f_{r'-1}}_{n-a}(\lambda^ax')=\lambda^{n-e(f)} y,
    \end{equation}
    and this differential is not induced by $(\rho_X,\rho_Y)$, as we assume that $d^{f,E_r}_n(x)=y$ is essential.
    
    There are two scenarios where this differential is not induced by $(\rho_X,\rho_Y)$. The first case occurs when $\lambda^ax'$ is not in the image of $\rho_X$ at all, which is equivalent to $x'\notin Z_{r-1-a}^{s+a,t+a}(X)$. (See Figure \ref{fig:b22d961f}.) 
    
    \begin{figure}[h]
    \centering
    \scalebox{0.9}{\begin{tikzpicture}
        % Draw grid
        \fill[llgray] (0-0.5, 4-0.5) rectangle ++(1, 2);
        \tikzgrid{0}{2}{0}{7}
        \tikzgrid{3}{5}{0}{7}
        \draw[thick] (0-0.5, 4-0.5) rectangle ++(1, 2);
    
        % Draw bullets
        \coordinate (x) at (1, 0);
        \coordinate (x1) at (1, 1);
        \coordinate (dx1) at (0, 5);
        \coordinate (dx) at (0, 6);
        \coordinate (y) at (4, 3);
        \coordinate (y1) at (4, 2);
        \coordinate (dy) at (3, 6);
        
        \fill (x) circle (0.1) node[below right] {$x$};
        \fill (x1) circle (0.1) node[below right] {$x'$};
        \fill (y1) circle (0.1) node[right] {$y$};
        \fill (dx1) circle (0.1) node[right] {};
    
        % Draw differentials
        \draw[dashed, -{Stealth[length=1.5mm]}, shorten >=(0.08cm)] (x) -- (dx) node[near end, above right] {$d_{\ge r'}$};
        \drawarrow (x1) -- (dx1) node[near end,left
        ] {$d_*$};
        \draw[dashed, -{Stealth[length=1.5mm]}, shorten >=(0.08cm)] (x) -- (y) node[near end,above left=-5pt] {$d^{f,E_{r'}}_{>n}$};
        \drawarrow (x) -- (y1) node[midway,below right=-5pt] {$d^{f,E_{r}}_{n}$};
        \drawarrow (x1) -- (y1) node[near start,above=-3pt] {$d^{f,E_{r'-a}}_{n-a}$};
    
        \node[anchor=east] at (-0.6, 0) {$s$};
        \node[anchor=east] at (-0.6, 1) {$s+a$};
        \node[anchor=east] at (-0.6, 2) {$s+n$};
        \node[anchor=east] at (-0.6, 4) {$s+r'$};
        \node[anchor=east] at (-0.6, 6) {$s+r$};
    
        \node at (.5, -1.0) {$E_2(X)$};
        \node at (3.5, -1.0) {$E_2(Y)$};
    \end{tikzpicture}}
    \caption{Extensions across pages}\label{fig:b22d961f}
    \end{figure}
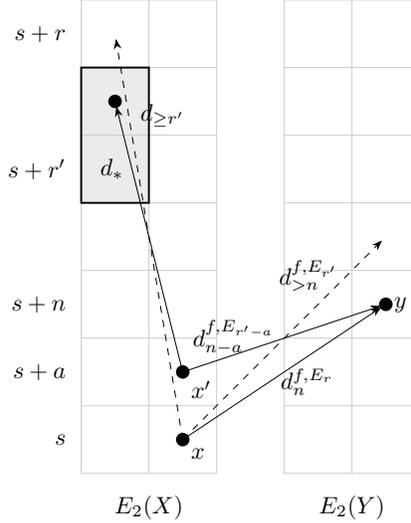
    
    The second case occurs when $\lambda^ax'$ is in the image of $\rho_X$, but it supports an essential $\hat f_{r-1}$-extension that is strictly shorter than the differential (\ref{eq:dcd62a19}). 
    
    To further explore this scenario, we replace $x$ with $\lambda^ax'$ and analyze successive essential extensions. This process is repeated iteratively until the first case is reached. Ultimately, this leads to the existence of some $0<a'\le n-e(f)$ and an element
    $$x''\in Z_{r'-1-a'}^{s+a',t+a'}(X)\backslash Z_{r-1-a'}^{s+a',t+a'}(X)$$
    such that
    $$d^{\hat f_{r'-1}}_{n-a'-b}(\lambda^{a'} x'')=\lambda^{n-b-e(f)} y'$$
    for some $b>0$. This represents a crossing of $d^{\hat f_{r'-1}}_n(x)=\lambda^{n-e(f)}y$.
    
\end{proof}

The following Corollary~\ref{cor:dfc6043e} is the contrapositive statement of the second part of Proposition~\ref{prop:dec738d3}.

\begin{corollary}\label{cor:dfc6043e}
    Suppose $r\ge r'$ and there exists an $(f,E_{r'})$-extension 
    $$d^{f,E_{r'}}_n(x)=y$$
    for $x\in Z_{r-1}^{s,t}(X)$ and $y\in Z_{r'-1-n+e(f)}^{s+n,t+n}(Y)$. Assume that this extension has no crossing of the form $d^{f,E_{r'-a}}_{n-a-b}(x')=y'$ for any $b>0$, $0<a\le n-e(f)$, and 
    $$x'\in Z_{r'-1-a}^{s+a,t+a}(X)\backslash Z_{r-1-a}^{s+a,t+a}(X).$$ 
    Under these conditions, we also have an $(f,E_{r})$-extension
    $$d^{f,E_{r}}_n(x)=y.$$
\end{corollary}

\begin{example} \label{exam:stretchext}
    In Example~\ref{exam:Mahowald}, we use the Generalized Mahowald Trick to establish the $(f, E_3)$-extension:
        $$d_2^{f, E_3} (h_0h_4^2) = h_0 p,$$
        for the map $f=\nu:S^3 \rightarrow S^0$, as discussed
in Example~\ref{exam:extonEn}(2). However, to apply the Generalized Leibniz Rule in proving the classical Adams differential 
$$d_3(h_2h_5) = h_0p$$
in Example~\ref{exam:Leibnizyes}, we need the $E_\infty$-page version of the $(f, E_3)$-extension:
$$d_2^{f, E_\infty} (h_0h_4^2) = h_0 p.$$
We apply Corollary~\ref{cor:dfc6043e} to confirm this.

%Specifically, set 
%$$r'=3, \ r=\infty, \ n=2, \ s=3, \ t=36, \ e(f)=1,$$
%$$x=h_0h_4^2 \in Z_\infty^{3,36}(S^3), \ y = h_0p \in Z_1^{5,38}(S^0)$$
As checked in Example~\ref{exam:Ercross}(2), the $(f, E_3)$-extension has no crossing. Consequently, by Corollary~\ref{cor:dfc6043e}, we obtain the required $(f, E_\infty)$-extension:
$$d_2^{f, E_\infty} (h_0h_4^2) = h_0 p.$$
\end{example}

\section{Proof of the main theorem}

In this section, we present the proof of our main Theorem~\ref{thm:h62}. 

\begin{theorem}[Theorem~\ref{thm:h62}] \label{thm:126survives}
    The element $h_6^2$ survives to the $E_\infty$-page in the Adams spectral sequence.
\end{theorem}

We first recall the following theorem, known as the inductive method, originally by Barratt--Jones--Mahowald \cite{BJMinduction} and later extended to the H$\mathbb{F}_2$-synthetic setting by Burklund--Xu \cite[Proposition~7.19]{BurklundXu}.

\begin{notation}
    Let $\theta_5 = [h_5^2]$ represent any synthetic homotopy class in $\pi_{62,62+2}S^{0,0}$ detected by $h_5^2$ in the Adams $E_2$-page. For convenience, we use the same notation, $\theta_5$, to denote its image in $\pi_{62,62+2}S^{0,0}/\lambda^r$ via the map $S^{0,0} \rightarrow S^{0,0}/\lambda^r$ for all $r \ge 1$. Similarly, let $\eta = [h_1] \in \pi_{1,1+1}S^{0,0}$.
\end{notation}

\begin{theorem}[Barratt--Jones--Mahowald, Burklund--Xu] \label{thm:bjmbx}
\begin{enumerate}
 \item[]
 \item    The element $h_6^2$ survives to the $E_{r+3}$-page of the classical Adams spectral sequence if and only if for some $\theta_5$, 
 $$\lambda \eta \theta_5^2 = 0 \ \text{in} \ \pi_{125,125+4} S^{0,0}/\lambda^{r+1}.$$
 \item   In particular, $h_6^2$ is a permanent cycle in the classical Adams spectral sequence if and only if for some $\theta_5$,
 $$\lambda \eta \theta_5^2 = 0 \ \text{in} \ \pi_{125,125+4} S^{0,0}.$$ 
 \end{enumerate}
\end{theorem}

\begin{remark}
    The statement in Theorem~\ref{thm:bjmbx}(1) was originally stated as 
    $$\eta \theta_5^2 = 0 \ \text{in} \ \pi_{125,125+5} S^{0,0}/\lambda^r$$
    in \cite[Proposition~7.19]{BurklundXu}. Upon inspection, $\pi_{125,125+5} S^{0,0}$ doesn't contain any $\lambda$-torsion classes, so this is equivalent to the version we stated, which is more consistent with the statement in part (2). 
\end{remark}

\begin{remark} \label{rem:theta5choice}
    Theorem~\ref{thm:bjmbx} is proved using the quadratic construction  on a map from the mod 2 Moore spectrum to the sphere spectrum, where the restriction on the bottom cell is $\theta_5$. Therefore, it is necessary to use a $\theta_5$ of order 2. 
    
    Notably, \cite{Xu, IWX} confirms that all classical $\theta_5$'s indeed have order 2. Furthermore, from Proposition~\ref{prop:30e8b746} and an analysis of the differentials in the classical Adams spectral sequence, we find that $\pi_{62,62+2}S^{0,0}$ doesn't contain any $\lambda$-torsion. Consequently, all synthetic $\theta_5$'s also have order 2, making it valid to apply Theorem~\ref{thm:bjmbx} to any $\theta_5$. 
    
    Additionally, since the proof of Theorem~\ref{thm:bjmbx} shows that the expression $\lambda \eta \theta_5^2$ corresponds to the total differential $\delta_1: S^{0,0}/\lambda \rightarrow S^{1,-1}$ on $h_6^2$, the value of the expression $\lambda \eta \theta_5^2$ is consistent for every choice of $\theta_5$.  (Note that our grading for the triangulation translation functor is smashing with $S^{1,0}$, which is consistent with \cite[Appendix A]{BHS} but is different from \cite[Section 7]{BurklundXu}, so the target of $\delta_1$ is $S^{1,-1}$. )
\end{remark}

We pay special attention to the following three elements on the classical Adams $E_2$-page (see Figure~\ref{Fig:near126} and Tables~\ref{Table:S125.19}, \ref{Table:S126.10}, \ref{Table:S124.12}, \ref{Table:S125.20} in the Appendix): 
\begin{align*}
    h_1h_4x_{109,12} & \in \Ext_A^{14, 125+14}, \\
    x_{126,8,4} + x_{126,8} & \in \Ext_A^{8, 126+8}, \\
    h_0^2x_{124,8} & \in \Ext_A^{10, 124+10},\\
    g^4\Delta h_1g & \in \Ext_A^{25,125+25}.
\end{align*}

For the right side of Figure ~\ref{Fig:near126}, we use dashed differentials to indicate the shortest possible nonzero differentials that these elements could support.

{
\include{112}
}

From Figure~\ref{Fig:near126} and Tables~\ref{Table:S124.12}, \ref{Table:S126.10}, \ref{Table:S125.19}, \ref{Table:S125.20} in the Appendix, we know that
\begin{fact} \label{fact:theta5sqAF}
\begin{enumerate}
   % \item[]
    \item $x_{126,8,4} + x_{126,8}$ survives to the $E_6$-page.
    
    \item $h_1h_4x_{109,12}$ is a permanent cycle, and can only be killed by 
    $$d_6(x_{126,8,4} + x_{126,8}) \ \text{or} \ d_{12}(h_6^2).$$
    
    \item $h_0^2x_{124,8}$ survives to the $E_\infty$-page.

    \item In $\Ext_A^{25,125+25}$, the element $g^4\Delta h_1g$ is the only one that survives to the classical $E_5$-page.
\end{enumerate}\end{fact}

\begin{remark}
From Table~\ref{Table:S126.10} in the Appendix, we have
  % For Fact~\ref{fact:theta5sqAF}(2), the element $x_{126,6} \in \Ext_A^{6, 126+6}$ supports a nonzero $d_3$-differential and therefore cannot kill $h_1h_4x_{109,12}$. In fact, since 
 %  $$h_2 \cdot d_3(x_{126,6}) = d_3(h_2 \cdot x_{126,6}) = h_0^2 h_3 D_2 h_6 \neq 0 
 %  \ \text{on the} \ E_3\text{-page},$$
 %  we have 
   $$d_3(x_{126,6}) = h_5x_{94,8}, \ \text{or} \ h_5x_{94,8} + h_6(\Delta e_1 + C_0 + h_0^6h_5^2) \neq 0.$$
   Therefore, the element $x_{126,6} \in \Ext_A^{6, 126+6}$ cannot kill $h_1h_4x_{109,12}$.
\end{remark}

We will apply Theorem~\ref{thm:bjmbx} to prove the following Proposition~\ref{prop:possibleh62}.

\begin{proposition} \label{prop:possibleh62}
Exactly one of the following two statements is true:
\begin{enumerate}
    \item The element $h_6^2$ survives to the $E_\infty$-page.
   
    \item There is a nonzero classical Adams differential 
    $$d_{12}(h_6^2) = h_1h_4x_{109,12}.$$
\end{enumerate}
Furthermore, statement $(2)$ is true if and only if the following three statements are all true:
\begin{enumerate} \setcounter{enumi}{2}
    \item \leavevmode\vspace*{-\dimexpr\abovedisplayskip + \baselineskip}  
    $$d_6(x_{126,8,4} + x_{126,8}) = 0.$$
    
    \item There exists a $\theta_5$ such that 
    $\theta_5^2$ is detected by $\lambda^6 h_0^2x_{124,8}$. In particular, 
    $$\theta_5^2 = \lambda^6 [h_0^2x_{124,8}] \neq 0 \in \pi_{124,124+4}S^{0,0}$$
    for some $[h_0^2x_{124,8}]$.

    \item There exists a homotopy class $[h_0^2x_{124,8}]$ such that
    $\lambda^3 \eta [h_0^2x_{124,8}]$ is detected by $\lambda^6 h_1h_4x_{109,12}$. In particular, we have 
    $$\lambda^3 \eta [h_0^2x_{124,8}] = \lambda^6 [h_1h_4x_{109,12}] \in \pi_{125,125+8} S^{0,0}$$
    for some $[h_1h_4x_{109,12}]$.
\end{enumerate}
\end{proposition}

By further analyzing classical Adams differentials, we reduce the $\eta$-extension in statement $(5)$ of Proposition~\ref{prop:possibleh62} to a specific 2-extension in stem 125 (Corollary~\ref{cor:2ext125}), and then compare it with a particular $\nu$-extension in stem 125 (Lemma~\ref{lem:nuext125}) to demonstrate that the 2-extension cannot hold. This ultimately leads to the proof of Proposition~\ref{prop:state5false}.

\begin{proposition} \label{prop:state5false}
    If statement $(3)$ is true, then statement $(5)$ in Proposition~\ref{prop:possibleh62} must be false.
\end{proposition}

\begin{proof}[Proof of Theorem~\ref{thm:126survives}]
    From Proposition~\ref{prop:state5false}, at least one of  statements $(3)$ or $(5)$ in Proposition~\ref{prop:possibleh62} is false. Consequently,  statement $(2)$ is also false, which confirms that statement $(1)$ in Proposition~\ref{prop:possibleh62} is true.  
\end{proof}

In the rest of this section, we prove Propositions~\ref{prop:possibleh62} and \ref{prop:state5false}.\\

To prove Proposition~\ref{prop:possibleh62}, we first establish Lemmas~\ref{lem:equistate4} and \ref{lem:equistate5}, which demonstrate that the existential statements $(4)$ and $(5)$ in Proposition~\ref{prop:possibleh62} are equivalent to their corresponding universal statements $(4')$ and $(5')$. 

\begin{lemma} \label{lem:equistate4}
    The statement $(4)$ in Proposition~\ref{prop:possibleh62} is equivalent to the following statement $(4')$:
    \begin{enumerate}[label=(\arabic*')]\setcounter{enumi}{3} 
    \item For every $\theta_5$, we have $\theta_5^2$ is detected by $\lambda^6 h_0^2x_{124,8}$. In particular, 
    $$\theta_5^2 = \lambda^6 [h_0^2x_{124,8}] \neq 0 \in \pi_{124,124+4}S^{0,0}$$
    for some $[h_0^2x_{124,8}]$.
    \end{enumerate}
\end{lemma}

\begin{proof}
    We only need to show that statement $(4)$ implies statement $(4')$. 
    
    According to \cite{IWX}, 
    $$\pi_{62} \cong \mathbb{Z}/2 \oplus \mathbb{Z}/2 \oplus \mathbb{Z}/2 \oplus \mathbb{Z}/2,$$ 
    and is generated by $\theta_5$ and classes of $\AF=6,8,10$. Therefore, the indeterminacy of the classical $\theta_5$, or the difference of any two choices of classical $\theta_5$, lies in $\AF \ge 6$.
    
    As explained in Remark~\ref{rem:theta5choice}, $\pi_{62,62+2}S^{0,0}$ contains no $\lambda$-torsion, and therefore, the indeterminacy of the synthetic $\theta_5$ also belongs to $\AF \ge 6$. Since every $\theta_5$ has order 2, the indeterminacy of the synthetic $\theta_5^2$ lies in $\AF \ge 12$.

    Therefore, if for some $\theta_5$, $\theta_5^2$ is detected by this specific element $\lambda^6 h_0^2x_{124,8}$, which is nonzero in $\AF=10$, then for any $\theta_5$, $\theta_5^2$ is nonzero and detected by $\lambda^6 h_0^2x_{124,8}$.

\end{proof}

\begin{lemma} \label{lem:equistate5}
    The statement $(5)$ in Proposition~\ref{prop:possibleh62} is equivalent to the following statement $(5')$:
    \begin{enumerate}[label=(\arabic*')]\setcounter{enumi}{4} 
    \item For every homotopy class $[h_0^2x_{124,8}]$, we have
    $\lambda^3 \eta [h_0^2x_{124,8}]$ is detected by $\lambda^6 h_1h_4x_{109,12}$. In particular, we have 
    $$\lambda^3 \eta [h_0^2x_{124,8}] = \lambda^6 [h_1h_4x_{109,12}] \in \pi_{125,125+8} S^{0,0}$$
    for some $[h_1h_4x_{109,12}]$.
    \end{enumerate}
\end{lemma}

\begin{proof}
    We only need to show that statement $(5)$ implies statement $(5')$, which is sufficient to show that the indeterminacy of $[h_0^2x_{124,8}]$, or the difference between any two homotopy classes $[h_0^2x_{124,8}]$, when multiplied by $\lambda^3 \eta$, belongs to $\AF \ge 15$, given that $[h_1h_4x_{109,12}]$ has $\AF=14$.

    Since $[h_0^2x_{124,8}]$ has $\AF=10$, the indeterminacy is generated by cycles in $\AF \ge 11$ of stem 124. We observe that classical cycles in $\AF=11, 12$ of stem 124 are all killed by Adams $d_2$ or $d_3$-differentials, and cycles in $\AF=13$ are all annihilated by $h_1$ in Ext. Therefore, we conclude that the indeterminacy, when multiplied by $\lambda^3 \eta$, belongs to $\AF \ge 15$.
\end{proof}

\begin{remark}
    From Fact~\ref{fact:theta5sqAF}(2) and the rigidity Theorems~\ref{thm:rigid} and \ref{thm:17e90ac0} for the synthetic Adams spectral sequence for $S^{0,0}$, the right side of the equation in statement $(5')$ is nonzero if and only if statement $(3)$ is true.
\end{remark}

By Lemmas~\ref{lem:equistate4} and \ref{lem:equistate5}, we will freely interchange statements $(4)$ and $(4')$, as well as $(5)$ and $(5')$, depending on the context.\\

Now we prove Proposition~\ref{prop:possibleh62}.

\begin{proof}[Proof of Proposition~\ref{prop:possibleh62}]
    From the proof of Theorem~\ref{thm:bjmbx} \cite{BurklundXu} we know that
    $$\delta_1(h_6^2) = \lambda \eta \theta_5^2,$$
    where $\delta_1$ is the map $S^{0,0}/\lambda \rightarrow S^{1,-1}$. Suppose that $\eta \theta_5^2$ is detected by $\lambda^{n-5} T_n$ in $\AF = n$ of the synthetic Adams spectral sequence for some $T_n \in \Ext_A^{n, 125+n}$. This implies a differential in the $\delta_1$-ESS:
    $$d_{n-2}^{\delta_1} (h_6^2) = \lambda^{n-2} T_n.$$
    By Theorems~\ref{thm:rigid}, \ref{thm:17e90ac0} and Corollary~\ref{cor:2a636737}, this is equivalent to a 
    synthetic Adams differential
    $$d_{n-2} (h_6^2) = \lambda^{n-3} T_n$$
    and a classical Adams differential
    $$d_{n-2} (h_6^2) = T_n.$$
    This differential is nonzero if and only if $T_n$ is nonzero on the classical $E_{n-2}$-page, i.e., not the image of a differential $d_{\le n-3}$. In particular, when $\lambda \eta \theta_5^2=0$, we conclude that $h_6^2$ is a permanent cycle.

    We first show that statements $(3)$, $(4')$ and $(5')$ together imply statement $(2)$.

    By statements $(4')$ and $(5')$, we have that for any $\theta_5$,
    $$\lambda \eta \theta_5^2 = \lambda \eta \cdot \lambda^6 [h_0^2x_{124,8}] = \lambda^4 \cdot \lambda^3 \eta [h_0^2x_{124,8}] = \lambda^{10} [h_1h_4x_{109,12}].$$
Recall from Fact~\ref{fact:theta5sqAF}(2) that $h_1h_4x_{109,12}$ is a permanent cycle, and can only be killed classically by 
    $$d_6(x_{126,8,4} + x_{126,8}) \ \text{or} \ d_{12}(h_6^2).$$
    Therefore, statement $(3)$ implies that the expression $\lambda^{10} [h_1h_4x_{109,12}]$ is nonzero in synthetic homotopy, leading to a nonzero classical $d_{12}$-differential:
    $$d_{12}(h_6^2) = h_1h_4x_{109,12}.$$
    
We will complete the proof of
    Proposition~\ref{prop:possibleh62} by showing that if one the statements $(3)$, $(4)$ or $(5)$ is false, then statement $(2)$ is also false, and in this case, 
    $$\eta \theta_5^2 = 0 \in \pi_{125,125+5} S^{0,0},$$ 
    making statement $(1)$ true. This conclusion is reached by estimating the Adams filtration of $\theta_5^2$ and subsequently that of the expression $\eta \theta_5^2$.

We begin by estimating the Adams filtration of $\theta_5^2$ in   $\pi_{124,124+4} S^{0,0}$. First, we observe that this group  $\pi_{124,124+4} S^{0,0}$ does not contain any $\lambda$-torsion classes. 
This is because, in the 125-stem, the group (from Table~\ref{Table:S125.19} in the Appendix) 
$$\Ext_A^{i,125+i} = 0 \ \text{for} \ i\leq 4,$$ 
and thus, by the rigidity Theorems~\ref{thm:rigid} and \ref{thm:17e90ac0} for the synthetic Adams spectral sequence for $S^{0,0}$, in the 124-stem, $\lambda$-torsion classes can only appear in $\pi_{124,124+j} S^{0,0}$ for $j \geq 7$.

Additionally, by analyzing classical Adams differentials, the Adams filtration of $\theta_5^2$ is at least 10. In the case where it is of Adams filtration 10, it is detected by the element $h_0^2 x_{124,8}$, which, according to Fact~\ref{fact:theta5sqAF}(3), is a permanent cycle and cannot be killed. 

Upon further inspection of the differentials associated with elements in stem 124 and filtration between 10 and 13, we are left with three possibilities:
\begin{enumerate}[label=(\alph*)]
    \item $\theta_5^2 = \lambda^6 \cdot [h_0^2 x_{124,8}]$, which is statement (4), or
    \item $\theta_5^2 = \lambda^9 \cdot [e_0 \Delta h_6 g]$, where $e_0 \Delta h_6 g$ is permanent cycle in $\AF=13$, or
    \item $\theta_5^2$ is a $\lambda^{10}$-multiple.
\end{enumerate}

Since in Ext, we have $h_1 \cdot e_0 \Delta h_6 g = 0$, we deduce that in either possibilities (b) or (c), $\eta \theta_5^2$ is a $\lambda^{10}$-multiple. In other words, $\eta \theta_5^2$ has $\AF \ge 15$. In the group $\pi_{125,125+5} S^{0,0}$, the only class that is a $\lambda^{10}$-multiple is actually $\lambda$-free: By Fact~\ref{fact:theta5sqAF}(4), it is $\lambda^{20} g^4 \Delta h_1 g$ in $\AF=25$ and is detected by tmf (see \cite{tmf} for the Hurewicz image of tmf). Since $\theta_5$ maps to zero in tmf, we have that $\eta \theta_5^2$ maps to zero in tmf and therefore must be zero in this case. 

This shows that if statement (4) is false, then $\eta \theta_5^2 = 0$, and therefore $h_6^2$ is a permanent cycle.

Hence, we focus on the remaining possibility (a), assuming statement (4) holds:
$$\theta_5^2 = \lambda^6 \cdot [h_0^2 x_{124,8}].$$

By the proof of Lemma~\ref{lem:equistate5}, the only nonzero possibility for $\lambda^3 \eta [h_0^2x_{124,8}]$ is $\lambda^6 [h_1h_4x_{109,12}]$. Therefore, if statement (5) is false, then using the same tmf detection argument, we conclude that $\eta \theta_5^2 = 0$, and consequently, $h_6^2$ is a permanent cycle.

Finally, if statement (3) is false, an inspection reveals the only alternative classical differential:
$$d_6(x_{126,8,4} + x_{126,8}) = h_1h_4x_{109,12}.$$
This, combined with the tmf detection argument, would again imply $\eta \theta_5^2 = 0$, and consequently, $h_6^2$ is a permanent cycle. This completes the proof.
\end{proof}

%From the same tmf detection argument and inspection of differentials associated to elements in stem 125 and Adams filtration 12 and above tells us the only nonzero possibility for $\eta \theta_5^2$ is $\lambda^9 \cdot [h_1h_4x_{109,12}]$.\\

%Next, we focus on the element $h_1h_4x_{109,12}$ in stem 125, Adams filtration 14. It is known that this is a permanent cycle, and there are three possibilities of which element could kill it (classically).
%\usepackage{enumitem}
%\begin{enumerate}[label=(\alph*)]
%    \item $d_6(x_{126,8,4} + x_{126,8})  = \lambda^5 h_1h_4x_{109,12}$, or
%    \item $d_{12} (h_6^2) = \lambda^{11} h_1h_4x_{109,12}$, or
%    \item $h_1h_4x_{109,12}$ survives to a nonzero homotopy class in stem 125 classically. 
%\end{enumerate}

%Let $y$ denote a homotopy class in $\pi_{125,125+14} S^{0,0}$ that is detected by $h_1h_4x_{109,12}$ and maps to zero in tmf. Inspection of differentials associated to elements in stem 125 and Adams filtration 14 and above tells us possibility (a) implies that $\lambda^5 \cdot y = 0$, so it won't be relevant to $\eta \theta_5^2$.

%For either possibilities (b) or (c), we have that
%$$\lambda^{10} \cdot y\neq 0.$$
%This is the statement we will work with from now on, for the remaining possibilities (b) and (c).\\

Before we prove Proposition~\ref{prop:state5false}, we first state and prove a few lemmas.

For Lemma~\ref{lem:x1239}, we draw attention to the following element in Ext:
$$x_{123,9}+h_0x_{123,8} \in \Ext_A^{9, 123+9}.$$
From Tables~\ref{Table:S123}, \ref{Table:S125.19} in the Appendix, we have
\begin{fact} \label{fact:x1239}
    \begin{enumerate}

   \item  $x_{123,9}+h_0x_{123,8}$ survives to the Adams $E_{12}$-page, and is not killed by any classical differential.
      % \item In Ext, $h_1 \cdot (x_{123,9}+h_0x_{123,8}) = h_1 x_{123,9}$.
    \item We have a classical nonzero differential 
    $$d_2 (x_{125,8}) = h_1 (x_{123,9}+h_0x_{123,8}) + h_0^2x_{124,8}.$$
   \end{enumerate}
\end{fact}

\begin{lemma} \label{lem:x1239}
    There exists a homotopy class 
    $$\alpha_1 = [x_{123,9}+h_0x_{123,8}] \in \pi_{123,123+9} S^{0,0}/\lambda^9$$
    with the following properties:
    \begin{enumerate}
    \item For any homotopy class $[h_0^2 x_{124,8}]$, there exist homotopy classes 
    $$\alpha_2 \in \pi_{124,124+13} S^{0,0}/\lambda^9, \ \alpha_3 \in \pi_{125,125+15} S^{0,0}/\lambda^9,$$ 
    such that 
    \begin{align*}
        \lambda^3 \eta \cdot \alpha_1 & = \lambda^3 [h_0^2 x_{124,8}] + 
\lambda^6 \alpha_2 && \in \pi_{124,124+7} S^{0,0}/\lambda^9,\\
\eta \cdot \alpha_2 & = \lambda \cdot \alpha_3 && \in \pi_{125,125+14} S^{0,0}/\lambda^9,
    \end{align*}
\item \leavevmode\vspace*{-\dimexpr\abovedisplayskip + \baselineskip} 
$$\lambda^3 \cdot \alpha_1 \cdot [h_0] = 0 \in \pi_{123,123+7} S^{0,0}/\lambda^9.$$
    \end{enumerate}
\end{lemma}

\begin{proof}
From Fact~\ref{fact:x1239}(1), the element $x_{123,9}+h_0x_{123,8}$ survives to the Adams $E_{12}$-page, and is not killed by any classical differential.

Let $\alpha_1 = [x_{123,9}+h_0x_{123,8}] \in \pi_{123,123+9} S^{0,0}/\lambda^{11}$ denote any homotopy class detected by $x_{123,9}+h_0x_{123,8}$. For simplicity, we also use  $\alpha_1$ to denote its images in $\pi_{123,123+9} S^{0,0}/\lambda^r$ for $1 \leq r \leq 10$, under the following sequence of maps:
$$
\xymatrix{
S^{0,0}/\lambda^{11} \ar[r] & S^{0,0}/\lambda^{10} \ar[r] & \cdots \ar[r] & S^{0,0}/\lambda \\
\alpha_1 = [x_{123,9}+h_0x_{123,8}] \ar@{|->}[r] & \alpha_1 \ar@{|->}[r] & \cdots \ar@{|->}[r] & x_{123,9}+h_0x_{123,8}
}$$

From the nonzero $d_2$-differential in Fact~\ref{fact:x1239},  we have 
$$h_1 \cdot (x_{123,9}+h_0x_{123,8}) = h_0^2x_{124,8},$$
on the classical $E_3$-page. This implies synthetically, for $2 \leq r \leq 11$,
$$\lambda \eta \cdot \alpha_1 + \lambda [h_0^2 x_{124,8}] \in \pi_{124,124+9} S^{0,0}/\lambda^r$$
lies in $\AF \ge 11$ for any $[h_0^2 x_{124,8}]$.

%We claim that synthetically 
%$$\lambda^3 \eta \cdot x = \lambda^3 [h_0^2 x_{124,8}] + 
%\lambda^6 w \ \ \text{in} \ \ \pi_{124,124+7} S^{0,0}/\lambda^9,$$
%where $w \in \pi_{124,124+13} S^{0,0}/\lambda^9$  with the property that $\eta w$ is $\lambda$-divisible. 

By inspection, as in the proof of Lemma~\ref{lem:equistate5}, $$\lambda^3 \eta \cdot \alpha_1 + \lambda^3 [h_0^2 x_{124,8}] \in \pi_{124,124+7} S^{0,0}/\lambda^9$$
has $\AF \ge 13$. The only possibility for it to have $\AF=13$ is that it is detected by the element $\lambda^6 e_0 \Delta h_6g$. (Note that the class $[\lambda^6 h_4 x_{109,12}]$ is irrelevant due to the nonzero differential $d_3(\lambda^6 h_4 x_{109,12}) = \lambda^8 h_1x_{122,15,2}$.) Since in Ext we have 
$$h_1 \cdot e_0 \Delta h_6g = 0,$$
we may choose $\alpha_2 \in \pi_{124,124+13} S^{0,0}/\lambda^9$ to be any class detected by $\lambda^6 e_0 \Delta h_6g$, with the property that $\eta \alpha_2$ is $\lambda$-divisible. Thus, there exists an $\alpha_3$ such that $\eta \alpha_2 = \lambda  \alpha_3$.

This proves the required property $(1)$.

For the relation in property $(2)$, we will first prove it in $\pi_{123,123+7} S^{0,0}/\lambda^{11}$ and then map it to $\pi_{123,123+7} S^{0,0}/\lambda^9$.

By Proposition~\ref{prop:59f111f} for the synthetic Adams spectral sequence for $S^{0,0}/\lambda^{11}$, the expression 
$$\lambda^3 \cdot \alpha_1 \cdot [h_0] \in  \pi_{123,123+7} S^{0,0}/\lambda^{11}$$
has $\AF \ge 17$. %(one can actually show that it has Adams filtration at least 17 but we only need it to be at least 16). 
In particular, the values $$\lambda^4(x_{123,11,2}+x_{123,11}+h_0h_6B_4) \ \text{in} \ \AF=11,$$ 
$$\lambda^8 h_0^2x_{123,13,2} \ \text{in} \ \AF=15$$
can be ruled out due to the nonzero Adams differentials
$$d_7( \lambda^4 \cdot (x_{123,11,2}+x_{123,11}+h_0h_6B_4) \ ) = \lambda^{10} h_1x_{121,17},$$
$$d_3 (\lambda^8 \cdot h_0^2x_{123,13,2}) = \lambda^{10} h_0^2 x_{122,16},$$
in the synthetic Adams spectral sequence for $S^{0,0}/\lambda^{11}$, which are zero in the spectral sequence for $S^{0,0}/\lambda^{9}$.

Therefore, the expression $\lambda^3 \cdot \alpha_1 \cdot [h_0]$ is $\lambda^{10}$-divisible in $\pi_{123,123+7} S^{0,0}/\lambda^{11}$. Mapping it further to $\pi_{123,123+7} S^{0,0}/\lambda^9$, we conclude that it is zero. 
\end{proof}

%we also need to exclude $\lambda^4[x_{123,11,2}+x_{123,11}+h_0h_6[B_4]]$, but it is impossible that $\lambda [h_0] x = \lambda^2[x_{123,11,2}+x_{123,11}+h_0h_6[B_4]]$ in $S^{0,0}/\lambda^{9}$
%}

%\underline{Step 2-2.} \ Now we return to our main proof.

%We will conclude our proof by showing that for any class $[h_0^2 x_{124,8}]$,
%$$\lambda^3 \eta \cdot [h_0^2 x_{124,8}] \neq \lambda^6 \cdot y  \ \ \text{in} \ \ \pi_{125,125+8} S^{0,0}.$$

%In fact, for a specific class $[h_0^2 x_{124,8}]$ that satisfies $\theta_5^2 = \lambda^6 \cdot [h_0^2 x_{124,8}]$, this would imply that
%$$\eta \theta_5^2 = \lambda^6 \eta \cdot [h_0^2 x_{124,8}] \neq \lambda^9 \cdot y,$$
%{there might be another term with filtraion $\geq 13$} {\color{blue} Not sure what you mean} \Comment{Weinan}{When both sides involve choices of homotopy elements, strict equality is hard to reach.} {\color{blue}Will modify.}
%which is the only nonzero possibility by previous discussion. Therefore, we must have $\eta \theta_5^2 = 0$.\\

For Lemma~\ref{lem:toda2ext}, we draw attention to the following element in Ext:
$$h_0^2x_{125,9,2} \in \Ext_A^{11, 125+11}.$$
From Table~\ref{Table:S125.19} in the Appendix, we have
\begin{fact} \label{fact:h02x1259}
    The element $h_0^2x_{125,9,2}$ survives to the Adams $E_5$-page, and is not killed by any classical differential.
\end{fact}

\begin{lemma} \label{lem:toda2ext}
    Assuming that both statements $(3)$ and $(5')$ in Proposition~\ref{prop:possibleh62} are true, the synthetic Toda bracket
$$\langle \lambda^3 \alpha_1, [h_0], \eta \rangle \subset \pi_{125,125+7} S^{0,0}/\lambda^9$$
    does not contain zero, and is detected by $\lambda^4 h_0^2x_{125,9,2}$. Here $\alpha_1 = [x_{123,9}+h_0x_{123,8}]$ refers to the homotopy class described in Lemma~\ref{lem:x1239}.
\end{lemma}

Note that the synthetic Toda bracket in Lemma~\ref{lem:toda2ext} is well defined, as the homotopy class $\alpha_1$ in Lemma~\ref{lem:x1239} satisfies the relation $\lambda^3 \alpha_1 \cdot [h_0] = 0$.

\begin{remark} \label{rem:h02x1259}
    According to Fact~\ref{fact:h02x1259}, it is not yet known whether the element $h_0^2x_{125,9,2}$ supports a nonzero $d_5$-differential. Assuming that both statements $(3)$ and $(5')$ in Proposition~\ref{prop:possibleh62} are true, Lemma~\ref{lem:toda2ext} specifically implies that $\lambda^4 h_0^2x_{125,9,2}$ detects a nonzero homotopy class in $\pi_{125,125+7} S^{0,0}/\lambda^9$. Therefore, under these assumptions, we would have $d_5(h_0^2x_{125,9,2}) = 0$. 
\end{remark}

\begin{proof}[Proof of Lemma~\ref{lem:toda2ext}]
    We assume both statements $(3)$ and $(5')$ in Proposition~\ref{prop:possibleh62} are true. From statement $(5')$, we have 
    $$\lambda^3 \eta [h_0^2x_{124,8}] = \lambda^6 [h_1h_4x_{109,12}] \in \pi_{125,125+8} S^{0,0}.$$
Mapping this relation to $S^{0,0}/\lambda^9$, and applying the following relations from Lemma~\ref{lem:x1239}
\begin{align*}
        \lambda^3 \eta \cdot \alpha_1 & = \lambda^3 [h_0^2 x_{124,8}] + 
\lambda^6 \alpha_2 && \in \pi_{124,124+7} S^{0,0}/\lambda^9,\\
\eta \cdot \alpha_2 & = \lambda \cdot \alpha_3 && \in \pi_{125,125+14} S^{0,0}/\lambda^9,
    \end{align*}
we have
\begin{align*}
    \eta \cdot \lambda^3 \eta \alpha_1 & = \eta \cdot \lambda^3 [h_0^2 x_{124,8}] + \eta \cdot \lambda^6 \alpha_2\\
   & = \lambda^6 [h_1h_4x_{109,12}] + \lambda^7 \alpha_3 \in \pi_{125,125+8} S^{0,0}/\lambda^9,
 \end{align*}
 which, from statement $(3)$ and Remark~\ref{rem:h02x1259}, is nonzero and detected by $\lambda^6 h_1h_4x_{109,12}$ in $\AF=14$.
 
 On the other hand, since $\eta^2 = \langle [h_0], \eta, [h_0] \rangle$, we have
 $$\eta \cdot \lambda^3 \eta \alpha_1 = \lambda^3 \alpha_1 \cdot \langle [h_0], \eta, [h_0] \rangle = \langle \lambda^3 \alpha_1, [h_0], \eta \rangle \cdot [h_0].$$
 Therefore, the synthetic Toda bracket
$\langle \lambda^3 \alpha_1, [h_0], \eta \rangle$
    does not contain zero, and its $[h_0]$-multiple is detected by $\lambda^6 h_1h_4x_{109,12}$ in $\AF=14$. Since $h_1h_4x_{109,12}$ is not $h_0$-divisible in Ext, the synthetic Toda bracket is detected by an element in $\AF \le 12$.

This synthetic Toda bracket $\langle \lambda^3 \alpha_1, [h_0], \eta \rangle$ lies in $\pi_{125,125+7} S^{0,0}/\lambda^9$, whose $\AF \le 12$ part is generated by
\begin{align*}
[h_0^2 x_{125,5}] \ & \text{in} \ \AF=7,\\ 
\lambda^2 [h_6(\Delta e_1 +C_0+h_0^6h_5^2)] \ & \text{in} \ \AF=9,\\ 
[\lambda^4 h_0^2 x_{125,9,2}] \ & \text{in} \ \AF=11.
\end{align*}
%\ [\lambda^6 h_3 x_{118,12}], \ \lambda^7 y, \ [\lambda^7x_{125,14}].$$

%also need to check the output data to have $d_3(x_{126,6})$

%Note that, %due to the nonzero differential 
%$$d_3(h_2 \cdot x_{126,6}) = \lambda^2 h_2 \cdot h_5 x_{94,8},$$
%we deduce that 
From Table~\ref{Table:S126.10} in the Appendix, we have
$$0 \neq d_3(x_{126,6}) = \lambda^2 h_5 x_{94,8} + \ \text{possibly} \ \lambda^2 h_6(\Delta e_1 +C_0+h_0^6h_5^2).$$ 
In both scenarios, $\lambda^2 [h_6(\Delta e_1 +C_0+h_0^6h_5^2)]$ remains and is in $\AF=9$.

For the rest of the proof, we only need to rule out the cases $[h_0^2 x_{125,5}]$ in $\AF=7$ and $\lambda^2 [h_6(\Delta e_1 +C_0+h_0^6h_5^2)]$ in $\AF=9$. 

For $[h_0^2 x_{125,5}]$ in $\AF=7$, due to the nonzero $d_3$-differential
$$d_3(x_{126,4}) = \lambda^2 h_0^2 x_{125,5},$$
it can be chosen to be annihilated by $\lambda^2$. 

However, from statement $(3)$ and Remark~\ref{rem:h02x1259}, $\lambda^8 [h_1h_4x_{109,12}]$ remains nonzero in the homotopy of $S^{0,0}/\lambda^9$. Therefore, the synthetic Toda bracket is not annihilated by $\lambda^2$. As discussed earlier, its $[h_0]$-multiple is detected by $\lambda^6 h_1h_4x_{109,12}$, and thus, this case of $[h_0^2 x_{125,5}]$ can be ruled out.

For $\lambda^2 [h_6(\Delta e_1 +C_0+h_0^6h_5^2)]$ in $\AF=9$, we first consider a classical Toda bracket in stem 125:
$$\langle \theta_5, 2, [\Delta e_1 +C_0+h_0^6h_5^2] \rangle.$$

From \cite{Xu, IWX}, $\pi_{62} \cong \mathbb{Z}/2 \oplus \mathbb{Z}/2 \oplus \mathbb{Z}/2 \oplus \mathbb{Z}/2$, so in particular both $\theta_5$ and $[\Delta e_1 +C_0+h_0^6h_5^2]$ have order 2, and this classical Toda bracket is well defined.

From classical $d_2$-differentials:
$$d_2(h_6) = h_0h_5^2, \ d_2(h_0^6h_6)= h_0(\Delta e_1 +C_0+h_0^6h_5^2),$$
we obtain the following Massey product on the $E_3$-page
$$h_6(\Delta e_1 +C_0+h_0^6h_5^2) = \langle h_5^2, h_0, \Delta e_1 +C_0+h_0^6h_5^2 \rangle,$$
and we check that it has zero indeterminacy. Further analysis shows that there are no crossing differentials, as per the criteria of Moss's theorem \cite[Theorem~1.2]{Moss} (noting that crossing differentials in Moss's theorem have a different meaning from our definition). Therefore, we have a classical Toda bracket 
$$[h_6(\Delta e_1 +C_0+h_0^6h_5^2)] \in \langle \theta_5, 2, [\Delta e_1 +C_0+h_0^6h_5^2] \rangle.$$

Synthetically, by inspection, we also have $2\theta_5=0$ and $2 [\Delta e_1 +C_0+h_0^6h_5^2] = 0$. It follows that there is a corresponding synthetic Toda bracket
$$[h_6(\Delta e_1 +C_0+h_0^6h_5^2)] \in \langle \theta_5, 2, [\Delta e_1 +C_0+h_0^6h_5^2] \rangle.$$
Multiplying by $\lambda^2 [h_0]$, we get:
\begin{align*}
    \lambda^2 [h_0] \cdot [h_6(\Delta e_1 +C_0+h_0^6h_5^2)]& = \lambda^2 [h_0] \cdot \langle \theta_5, 2, [\Delta e_1 +C_0+h_0^6h_5^2] \rangle \\
    & = \lambda \cdot \langle 2, \theta_5, 2 \rangle [\Delta e_1 +C_0+h_0^6h_5^2]\\
    & = \lambda^3 \eta \theta_5 [\Delta e_1 +C_0+h_0^6h_5^2].
\end{align*}
Note that all expressions in the above equation have zero indeterminacy. 

If the synthetic Toda bracket $\langle \lambda^3 \alpha_1, [h_0], \eta \rangle$ were $\lambda^2 [h_6(\Delta e_1 +C_0+h_0^6h_5^2)]$, mapping the above equation to $S^{0,0}/\lambda^9$, we would then have a nonzero equation in $\pi_{125,125+8} S^{0,0}/\lambda^9$.
\begin{align*}
    \lambda^3 \eta [h_0^2 x_{124,8}] = \lambda^6 [h_1h_4x_{109,12}] & = \langle \lambda^3 \alpha_1, [h_0], \eta \rangle [h_0] + \lambda^6 \eta \alpha_2 \\ & = \lambda^2 [h_6(\Delta e_1 +C_0+h_0^6h_5^2)] [h_0] + \lambda^6 \eta \alpha_2\\ & = \lambda^3 \eta \theta_5 [\Delta e_1 +C_0+h_0^6h_5^2] + \lambda^6 \eta \alpha_2.
\end{align*}
For this equation to be nonzero, we must have a nonzero equation
$$\lambda^3 [h_0^2 x_{124,8}] = \lambda^3 \theta_5 [\Delta e_1 +C_0+h_0^6h_5^2] + \lambda^6 \alpha_2 \ \ \text{in} \ \ \pi_{124,124+7} S^{0,0}/\lambda^9.$$
However, in Ext, we have 
$$h_5^2 (\Delta e_1 +C_0+h_0^6h_5^2) = 0 \neq h_0^2 x_{124,8} \in \Ext_A^{10,124+10},$$
so this equation is not possible.

We have ruled out the possibility that the synthetic Toda bracket $\langle \lambda^3 \alpha_1, [h_0], \eta \rangle$ is detected by $\lambda^2 [h_6(\Delta e_1 +C_0+h_0^6h_5^2)]$ or $[h_0^2 x_{125,5}]$. Therefore, we conclude that it must be detected by $[\lambda^4 h_0^2 x_{125,9,2}]$.
\end{proof}

From the proof of Lemma~\ref{lem:toda2ext}, we have the following $[h_0]$-extension.

\begin{corollary} \label{cor:2ext125}
    Assuming that both statements $(3)$ and $(5')$ in Proposition~\ref{prop:possibleh62} are true, we have a relation: For any homotopy class $[\lambda^4 h_0^2x_{125,9,2}]$,
    $$[\lambda^4 h_0^2x_{125,9,2}] \cdot [h_0] = \lambda^6 [h_1h_4x_{109,12}] \neq 0 \in \pi_{125,125+8} S^{0,0}/\lambda^9,$$
    for some $[h_1h_4x_{109,12}]$.
\end{corollary}

\begin{proof}
    From the proof of Lemma~\ref{lem:toda2ext}, there exists a homotopy class $[\lambda^4 h_0^2x_{125,9,2}]$ contained in the synthetic Toda bracket $ \langle \lambda^3 \alpha_1, [h_0], \eta \rangle$, and we have
    $$[\lambda^4 h_0^2x_{125,9,2}] \cdot [h_0] = \langle \lambda^3 \alpha_1, [h_0], \eta \rangle \cdot [h_0] = \lambda^6 [h_1h_4x_{109,12}] + \lambda^7 \alpha_3.$$
    Since $\lambda^7 \alpha_3$ is in a strictly higher filtration than $\lambda^6 [h_1h_4x_{109,12}]$, we may choose another class $[h_1h_4x_{109,12}]$ such that the right-hand side of the above equation is simply $\lambda^6 [h_1h_4x_{109,12}]$. From the discussion of this synthetic Toda bracket in the proof of Lemma~\ref{lem:toda2ext}, we know that the difference of any two classes detected by $\lambda^4 h_0^2x_{125,9,2}$, when multiplied by $[h_0]$, is not detected by $\lambda^6 [h_1h_4x_{109,12}]$. Therefore, the corollary holds for any choice of $[\lambda^4 h_0^2x_{125,9,2}]$.
\end{proof}

%Since in $\pi_{123,123+7} S^{0,0}/\lambda^{9}$ we also have $\lambda^3 \cdot x \cdot [h_0] = 0$, the Toda bracket 
% by consider this relation in $S^{0,0}/\lambda^9$. (EDIT and EXPAND

%these kind of argument show that there exists some representative of $[x_{126,4}]$ which do not have the extension
%}

%{ $\theta_5[\Delta e_1+C_0]$ has filtration at least $10$, and ext data could show that it is at least filtration $11$, but we have already excluded any $\eta$ extension in filtration $\geq11$}
%By inspection of differentials associated to stem 124, filtration between 11 and 13, we learn that there must be a relation in Ext: $h_0^2 x_{124,8} = h_5^2 (\Delta e_1 +C_0)$. This is, however, not true. 

%In summary, we have shown that the assumption
%$$\lambda^3 \eta \cdot [h_0^2 x_{124,8}] = \lambda^6 \cdot y  \ \ \text{in} \ \ \pi_{125,125+8} S^{0,0}$$
%implies that
%$$0 \neq \lambda^6 y + \lambda^6 \eta w = [\lambda^4 h_0^2 x_{125,9,2}] [h_0]  \ \ \text{in} \ \ \pi_{125,125+8} S^{0,0}/\lambda^9.$$

We have one more Lemma~\ref{lem:nuext125} before we prove Proposition~\ref{prop:state5false}.

We draw attention to the following element in Ext:
$$h_1x_{121,7} \in \Ext_A^{8,122+8}.$$
From Table~\ref{Table:S122} in the Appendix, we have
\begin{fact} \label{fact:h1x1217}
    The element $h_1x_{121,7}$ survives to the Adams $E_6$-page, and is not killed by any classical differential.
\end{fact}

\begin{lemma} \label{lem:nuext125}
    There exists a homotopy class $[\lambda^4 h_1 x_{121,7}] \in \pi_{122,122+4}S^{0,0}/\lambda^9$, such that 
    $$[\lambda^4 h_1 x_{121,7}] \cdot [h_2] = \lambda [\lambda^5 h_0^2 x_{125,9,2}] \in \pi_{125,125+5} S^{0,0}/\lambda^9.$$
\end{lemma}

\begin{proof}
    From Fact~\ref{fact:h1x1217}, we know that the element $h_1x_{121,7}$ may only support a nonzero $d_r$-differential for $r \ge 6$. By Proposition~\ref{prop:6de7d130}, $\lambda^4 h_1 x_{121,7}$ detects nonzero homotopy classes in $\pi_{122,122+4}S^{0,0}/\lambda^9$, and hence the required existence of such a homotopy class.

    For the desired relation, we follow Example~\ref{exam:Mahowald} and apply the Generalized Mahowald Trick Theorem~\ref{thm:158d451a}. Consider the distinguished triangle
        $$\xymatrix{S^3 \ar[r]^{\nu} & S^0 \ar[r]^-{i} & S^0/\nu \ar[r]^{q} & S^4.}$$
The short exact sequence on $\HF$-homology
$$\xymatrix{0 \ar[r] & {\HF}_*S^0 \ar[r]^-{i_*} & {\HF}_*S^0/\nu \ar[r]^-{q_*} & {\HF}_*S^4 \ar[r] & 0,}$$
induces a long exact sequence on Ext-groups
$$\xymatrix{\cdots \ar[r]^-{\cdot h_2} & {\Ext}_A^{*,*}(S^0) \ar[r]^-{i_*} & {\Ext}_A^{*,*}(S^0/\nu) \ar[r]^-{q_*} & {\Ext}_A^{*,*}(S^4) \ar[r]^-{\cdot h_2} & \cdots}.$$
Also recall for notations, if an element $x$ in ${\Ext}_A^{*,*}(S^0/\nu)$ satisfies 
$$q_*(x) = a \neq 0 \in \Ext_A^{*,*}(S^4),$$ 
we denote $x$ by $a[4]$; otherwise, due to exactness, 
$$x = i_*(b) \ \text{for some} \ b \in \Ext_A^{*,*}(S^0),$$
and in this case, we denote $x$ by $b[0]$. 

We consider
\begin{align*}
x & = h_1 x_{121,7} && \in \Ext_A^{8,130},\\
\bar{x} & =h_1 x_{121,7} [4] + x_{126,8}[0] + x_{126,8,2}[0] && \in \Ext_A^{8,134}(S^0/\nu), \\
y & =h_0^2 x_{125,9,2} && \in \Ext_A^{11,136} ,\\
        \bar{y} & =h_0^2 x_{125,9,2}[0] && \in \Ext_A^{11,136}(S^0/\nu). 
  \end{align*}      

For conditions in Theorem~\ref{thm:158d451a}, we have:
\begin{enumerate}
      \item $d_0^{q, E_2}(h_1 x_{121,7} [4] + x_{126,8}[0] + x_{126,8,2}[0]) = h_1 x_{121,7}$.
      \item $d_3 (h_1 x_{121,7} [4] + x_{126,8}[0] + x_{126,8,2}[0]) = h_0^2 x_{125,9,2} [0]$. This is a classical Adams $d_3$-differential for $S^0/\nu$, and is obtained from our computations.
      \item 
      \begin{enumerate}
          \item The differential in (1) has no crossing, as it is a $d_0$-differential, and
          \item Upon inspection, the differential in (2) has no crossing. 
      \end{enumerate}
      \item 
      $d_0^{i, E_2} (h_0^2 x_{125,9,2}) = h_0^2 x_{125,9,2}[0]$, as $h_0^2 x_{125,9,2}$ is not divisible by $h_2$ in Ext.
      \end{enumerate}
    Since all conditions of Theorem~\ref{thm:158d451a} are  satisfied, we conclude that there is an $([h_2], E_4)$-extension:
        $$d_3^{[h_2], E_4} (h_1 x_{121,7}) = h_0^2 x_{125,9,2}.$$
In other words, we have the following relation:
      $$[h_1 x_{121,7}] \cdot [h_2] = \lambda^2 [h_0^2 x_{125,9,2}] \in \pi_{125,125+9} S^{0,0}/\lambda^3.$$  
Using the map $\rho:S^{0,0}/\lambda^5 \rightarrow S^{0,0}/\lambda^3$ from Notation~\ref{not:deltaandrho}, we lift the above relation and obtain:
$$[h_1 x_{121,7}] \cdot [h_2] = [\lambda^2 h_0^2 x_{125,9,2}] \in \pi_{125,125+9} S^{0,0}/\lambda^5.$$
In fact, since $h_1 x_{121,7} \cdot h_2 = 0$ in $\AF=9$ of Ext, we might obtain a relation of the form:
$$[h_1 x_{121,7}] \cdot [h_2] = [\lambda x] + [\lambda^2 h_0^2 x_{125,9,2}] \in \pi_{125,125+9} S^{0,0}/\lambda^5,$$
for some element $x$ in $\AF=10$. However, upon inspection, for all $x \in  \Ext_A^{10,125+10}$, the homotopy class $[\lambda x]$ either does not exist or can be chosen to be zero.

Using the map $\lambda^4: \Sigma^{0,-4}S^{0,0}/\lambda^5 \rightarrow S^{0,0}/\lambda^9$ from Notation~\ref{not:deltaandrho}, we further push the above relation and obtain the following relation:
$$[\lambda^4 h_1 x_{121,7}] \cdot [h_2] = \lambda [\lambda^5 h_0^2 x_{125,9,2}] \in \pi_{125,125+5} S^{0,0}/\lambda^9.$$
This completes the proof.
\end{proof}

Now we prove Proposition~\ref{prop:state5false}.

We draw attention to the following elements in Ext:
\begin{align*}
h_6 M d_0 & \in \Ext_A^{11,122+11}, \\
h_5x_{91,11} & \in \Ext_A^{12,122+12}.
\end{align*}
From Table~\ref{Table:S122} in the Appendix, we have
\begin{fact} \label{fact:stem122}
\begin{enumerate}
    \item[] 
    \item The element $h_6 M d_0$ survives to the Adams $E_\infty$-page.
    \item The element $h_5x_{91,11}$ survives to the Adams $E_\infty$-page.
\end{enumerate}
    
\end{fact}

\begin{proof}[Proof of Proposition~\ref{prop:state5false}]
    We assume that statement $(3)$ in Proposition~\ref{prop:possibleh62} is true. For the sake of a contradiction, we also assume statement $(5')$ is true.

    From Lemma~\ref{lem:nuext125} and Corollary~\ref{cor:2ext125}, there exists a homotopy class\\ $[\lambda^4 h_1 x_{121,7}] \in \pi_{122,122+4}S^{0,0}/\lambda^9$, such that
\begin{align*}
    [\lambda^4 h_1 x_{121,7}] \cdot [h_2] \cdot [h_0] & = \lambda [\lambda^5 h_0^2 x_{125,9,2}] \cdot [h_0] \\
    & = \lambda^8 [h_1h_4x_{109,12}] \neq 0 \in \pi_{125,125+6} S^{0,0}/\lambda^9,
\end{align*} 
    for some $[h_1h_4x_{109,12}]$. Note that statement $(3)$ in Proposition~\ref{prop:possibleh62}, Fact~\ref{fact:theta5sqAF}(2), and Proposition~\ref{prop:6de7d130} imply that the expression  $\lambda^8 [h_1h_4x_{109,12}]$ is nonzero in the homotopy of $S^{0,0}/\lambda^9$.

Since the element detecting $h_1h_4x_{109,12}$ in not an $h_2$-multiple in Ext, we must have the expression
$$[\lambda^4 h_1 x_{121,7}] \cdot [h_0]$$
in $\pi_{122,122+5} S^{0,0}/\lambda^9$ be nonzero, and have $\AF \le 12$. Upon inspection, the only possibilities are:
$$\lambda^6 [h_6 M d_0] \ \text{in} \ \AF=11,  \ \ \ \lambda^7 [h_5x_{91,11}] \ \text{in} \ \AF=12.$$
From Fact~\ref{fact:stem122} both $h_6 M d_0$ and $h_5x_{91,11}$ are nonzero permanent cycles in the classical Adams spectral sequence, so either possibility would lift to a relation in the homotopy groups of $S^{0,0}$.

For the case of $\lambda^6 [h_6 M d_0]$, consider the expression
$$\lambda^2 [h_6 M d_0] \cdot [h_2] \in \pi_{125, 125+10}S^{0,0}.$$
Since $h_6 M d_0 \cdot h_2$ in $\AF=12$ is killed by a classical $d_2$-differential, we know that 
$$\lambda [h_6 M d_0] \cdot [h_2] \in \pi_{125, 125+11}S^{0,0}$$ 
is in $\AF \ge 13$. Upon inspection, the permanent cycles in $\Ext_A^{13, 125+13}$ are all killed by $d_2$ or $d_4$-differentials, and thus, $$\lambda^2 [h_6 M d_0] \cdot [h_2] \in \pi_{125, 125+10}S^{0,0}$$ 
is in $\AF \ge 14$. Therefore, for the above relation in the homotopy of $S^{0,0}/\lambda^9$ to hold, we must have 
$$\lambda^2 [h_6 M d_0] \cdot [h_2] = \lambda^4 [h_1h_4x_{109,12}] \in \pi_{125, 125+10} S^{0,0}.$$

For the case of $\lambda^7 [h_5x_{91,11}]$, consider the expression
$$\lambda [h_5x_{91,11}] \cdot [h_2] \in \pi_{125, 125+12}S^{0,0}.$$
Since $h_5x_{91,11} \cdot h_2$ in $\AF=13$ is killed by a classical $d_2$-differential, we know that
$$\lambda [h_5x_{91,11}] \cdot [h_2] \in \pi_{125, 125+12}S^{0,0}$$
is in $\AF \ge 14$.
Therefore, for the above relation in the homotopy of $S^{0,0}/\lambda^9$ to hold, we must have
$$\lambda [h_5x_{91,11}] \cdot [h_2] = \lambda^2 [h_1h_4x_{109,12}] \in \pi_{125, 125+12}S^{0,0}.$$

In both cases, $\lambda^4 [h_1h_4x_{109,12}]$ is a $\lambda [h_2]$-multiple in the homotopy of $S^{0,0}$. Since the classical $\nu \in \pi_3$ has $\AF=1$, we have 
$$S^{0,0}/(\lambda [h_2]) \simeq \nu (S^0/\nu).$$
By the rigidity Theorem~\ref{thm:rigid} of the synthetic Adams spectral sequence for $S^{0,0}/(\lambda [h_2])$, we know that the element $\lambda^4 h_1h_4x_{109,12}[0]$ must be killed by a synthetic Adams differential, which corresponds to to a statement that in the classical Adams spectral sequence of $S^0/\nu$, the element $h_1h_4x_{109,12}[0]$ must be killed by a $d_r$-differential for $r \leq 5$.

%\begin{center}
%\scalebox{0.8}{
\begin{table}[h]
    \centering
\scalebox{0.85}{\begin{tabular}{|c|l|l|l|}\hline 
$s$ & Elements & $d_r$ & value\\\hline\hline
\multirow{4}{*}{14} & $h_0^{12}h_6^2[0]$ & $d_{2}^{-1}$ & $h_0^{11}h_7[0]$ \\\cline{2-4}
 & $x_{126,14}[0]$ & $d_{3}^{-1}$ & $(((x_{123,11,2})+(x_{123,11})+h_0 h_6 [B_4])[4])$ \\\cline{2-4}
 & $Q_2D_2(h_3[4])+D_2x_{68,8}[0]$ & $d_{3}^{-1}$ & $(((x_{123,11,2})+h_5 (x_{92,10}))[4])$ \\\cline{2-4}
 & $D_2x_{68,8}[0]$ & $d_{2}$ & $h_0Q_2x_{68,8}[0]$ \\\hline\hline
\multirow{6}{*}{13} & $h_0^{11}h_6^2[0]$ & $d_{2}^{-1}$ & $h_0^{10}h_7[0]$ \\\cline{2-4}
 & $h_1x_{120,11}(h_1[4])$ & $d_{12}$ & $?$ \\\cline{2-4}
 & $h_1x_{125,12,2}[0]$ & $d_{5}$ & $d_0^2x_{97,10}[0]$ \\\cline{2-4}
 & $h_6x_{56,10}(h_0 h_2[4])$ & $d_{3}$ & $h_1x_{124,15}[0]$ \\\cline{2-4}
 & $(((x_{122,13})+h_1^2 (x_{120,11})+h_0^2 h_6 (Md_0))[4])$ & $d_{2}$ & $x_{125,15}[0]+h_0^3x_{125,12}[0]$ \\\cline{2-4}
 & $h_0h_3x_{119,11}[0]$ & $d_{2}$ & $h_0^3x_{125,12}[0]$ \\\hline\hline
\multirow{4}{*}{12} & $d_1x_{94,8}[0]$ & $d_{2}^{-1}$ & $x_{127,10}[0]$ \\\cline{2-4}
 & $h_0^{10}h_6^2[0]$ & $d_{2}^{-1}$ & $h_0^9h_7[0]$ \\\cline{2-4}
 & $((h_5 (x_{91,11})+h_0 (x_{122,11}))[4])$ & $d_{3}$ & $Q_2x_{68,8}[0]$ \\\cline{2-4}
 & $h_3x_{119,11}[0]$ & $d_{2}$ & $h_0^2x_{125,12}[0]$ \\\hline\hline
\multirow{4}{*}{11} & $h_0^9h_6^2[0]$ & $d_{2}^{-1}$ & $h_0^8h_7[0]$ \\\cline{2-4}
 & $x_{126,11}[0]$ & $d_{3}^{-1}$ & $x_{127,8}[0]$ \\\cline{2-4}
 & $h_1x_{125,10,2}[0]+h_1x_{125,10}[0]$ & & Permanent \\\cline{2-4}
 & $h_1x_{125,10}[0]$ & $d_{14}$ & $?$ \\\hline\hline
\multirow{3}{*}{10} & $h_0^2x_{126,8}[0]$ & $d_{2}^{-1}$ & $h_0x_{127,7}[0]$ \\\cline{2-4}
 & $h_0^8h_6^2[0]$ & $d_{2}^{-1}$ & $h_0^7h_7[0]$ \\\cline{2-4}
 & $x_{126,10}[0]$ & $d_{3}$ & $nx_{94,8}[0]$ \\\hline\hline
\multirow{5}{*}{9} & $h_0x_{126,8}[0]$ & $d_{2}^{-1}$ & $x_{127,7}[0]$ \\\cline{2-4}
 & $h_0^7h_6^2[0]$ & $d_{2}^{-1}$ & $h_0^6h_7[0]$ \\\cline{2-4}
 & $h_1x_{125,8}[0]$ & $d_{16}$ & $?$ \\\cline{2-4}
 & $h_0x_{126,8,3}[0]$ & $d_{4}$ & $h_0x_{125,12}[0]$ \\\cline{2-4}
 & $x_{126,9}[0]$ & $d_{3}$ & $h_0^4x_{125,8}[0]$ \\\hline %\hline
%\multicolumn{4}{|c|}{\text{stem}=126} \\\hline
%\multicolumn{4}{|c|}{\TitleCellColor $E_2^{*,*}(S^0/\nu)$}\\\hline
\end{tabular}}
\caption{The classical Adams spectral sequence of $S^0/\nu$ for $9 \le s \le 14$ in stem 126}
\label{Table:Cnu126}
\end{table}

%\end{center}

However, from Table~\ref{Table:Cnu126} (obtained from \cite{LWXMachine} \cite{LWXZenodo}, and can be visualized from \cite{LinPlot}), in the classical Adams spectral sequence of $S^0/\nu$, the element  $h_1h_4x_{109,12}[0]$ is not killed by any $d_r$ for $r \leq 5$ (from the range of $9 \le s \le 14$ in stem 126). Therefore, we arrive at a contradiction.
\end{proof}

% \newpage

\section{Appendix: The classical Adams spectral sequence in the range $122 \le t-s \le 127, s \le 25$} \label{sec:App}

We provide a brief overview of Lin's computer program for computing Adams differentials and extensions. The program's functionality for propagating differentials and extensions relies on the following data:
\begin{itemize}
    \item The Adams $E_2$-pages of a collection of CW spectra.
    \item Maps between these $E_2$-pages.
    \item Adams $d_2$-differentials for certain CW spectra.
    \item There are three manually added differentials: $d_5 h_0^{24}h_6=h_0^2P^6d_0$ and\\
 $d_6 h_0^{55}h_7=h_0^2x_{126,60}$ in the Adams spectral sequence of $S^0$ (from the image of $J$), and $d_3 v_2^{16}=\beta^5g$ in the Adams spectral sequence of $\mathit{tmf}$, derived from power operations (Bruner--Rognes \cite{BR21}).
\end{itemize}

Detailed descriptions of these data are available in \cite{LWXZenodo}. Lin’s program extends these results by computing additional Adams differentials using tools such as the Leibniz rule, naturality, the Generalized Leibniz Rule, and the Generalized Mahowald Trick. All computed differentials and extensions are accessible via interactive plots \cite{LinPlot}.

Moreover, the proofs provided in \cite{LWXZenodo} offer more information than the interactive plots. These proofs include numerous disproofs of potential differentials, even for cases where the differentials remain unresolved. For example, consider 
$$x_{126,21}\in E_2^{21,126+21}(S^0).$$
The spectral sequence plot for $S^0$ shows that $x_{126,21}$ survives to the $E_4$-page, but the value of $d_4(x_{126,21})$ undetermined. By analyzing the proofs in \cite{LWXZenodo}, we observe that many potential values for $d_4(x_{126,21})$ have been ruled out. Consequently, we conclude:
$$d_4(x_{126,21})=x_{125,25,2}+x_{125,25}+g^4\Delta h_1g+\text{possibly }d_0^2e_0gB_4.$$

Tables~\ref{Table:S122}--\ref{Table:S127.9} present results from Lin’s program for the classical Adams spectral sequence of the sphere in the range $122 \le t-s \le 127, s \le 25$.

\begin{table}
    \centering
\scalebox{1}{\begin{tabular}{|c|l|l|l|}\hline
    $s$ & Elements & $d_r$ & value\\\hline\hline
    \multirow{2}{*}{25} & $e_0g^3\Delta h_1g$ & $d_{2}^{-1}$ & $g^2\Delta^2m$ \\\cline{2-4}
     & $d_0^3g[B_4]$ & $d_{4}$ & $d_0^4x_{65,13}$ \\\hline\hline
    \multirow{1}{*}{24} & $d_0g\Delta^2g^2$ & $d_{2}$ & $d_0e_0g^3\Delta h_2^2$ \\\hline\hline
    \multirow{2}{*}{23} & $Ph_1x_{113,18,2}$ & $d_{3}^{-1}$ & $x_{123,20}$ \\\cline{2-4}
     & $h_0^2d_0x_{108,17}$ & $d_{3}$ & $d_0^4Mg$ \\\hline\hline
    \multirow{3}{*}{22} & $e_0g^2Mg$ & $d_{3}^{-1}$ & $h_0^2h_3x_{116,16}$ \\\cline{2-4}
     & $x_{122,22}$ & $d_{3}$ & $h_0^2d_0x_{107,19}$ \\\cline{2-4}
     & $h_0d_0x_{108,17}$ & $d_{2}$ & $h_0^2d_0x_{107,18}$ \\\hline\hline
    \multirow{2}{*}{21} & $h_0^2d_0e_0x_{91,11}$ & & Permanent \\\cline{2-4}
     & $d_0x_{108,17}$ & $d_{2}$ & $d_0e_0\Delta h_2^2[B_4]+h_0d_0x_{107,18}$ \\\hline\hline
    \multirow{3}{*}{20} & $h_0^4x_{122,16}$ & $d_{2}^{-1}$ & $h_0h_3x_{116,16}$ \\\cline{2-4}
     & $g^3(C_0+h_0^6h_5^2)$ & $d_{3}$ & $g^3\Delta h_2^2n$ \\\cline{2-4}
     & $h_0d_0e_0x_{91,11}$ & $d_{2}$ & $h_0^6x_{121,16}$ \\\hline\hline
    \multirow{2}{*}{19} & $h_0^3x_{122,16}$ & $d_{2}^{-1}$ & $h_3x_{116,16}$ \\\cline{2-4}
     & $d_0e_0x_{91,11}$ & $d_{2}$ & $h_0d_0^2x_{93,12}$ \\\hline\hline
    \multirow{3}{*}{18} & $h_0^2x_{122,16}$ & $d_{3}^{-1}$ & $h_0^2x_{123,13,2}$ \\\cline{2-4}
     & $h_1x_{121,17}$ & $d_{7}^{-1}$ & $x_{123,11,2}+x_{123,11}+h_0h_6[B_4]$ \\\cline{2-4}
     & $d_0x_{108,14}$ & $d_{3}$ & $h_0d_0^2x_{93,12}+h_0^5x_{121,16}$ \\\hline\hline
    \multirow{2}{*}{17} & $h_0x_{122,16}$ & $d_{3}^{-1}$ & $h_0x_{123,13,2}$ \\\cline{2-4}
     & $g^3[H_1]$ & $d_{3}^{-1}$ & $\Delta h_2^2x_{93,8}$ \\\hline\hline
    \multirow{3}{*}{16} & $x_{122,16}+h_0x_{122,15,2}$ & $d_{3}^{-1}$ & $x_{123,13,2}$ \\\cline{2-4}
     & $\Delta h_2^2x_{92,10}$ & $d_{3}^{-1}$ & $x_{123,13}$ \\\cline{2-4}
     & $h_0x_{122,15,2}$ & & Permanent \\\hline\hline
    \multirow{2}{*}{15} & $x_{122,15}$ & $d_{3}$ & $g^3A$ \\\cline{2-4}
     & $x_{122,15,2}$ & $d_{2}$ & $h_0^2h_4x_{106,14}$ \\\hline\hline
    \multirow{1}{*}{14} & $h_0x_{122,13}$ & $d_{3}^{-1}$ & $x_{123,11,2}$ \\\hline\hline
    \multirow{3}{*}{13} & $h_0^2h_6Md_0$ & $d_{2}^{-1}$ & $x_{123,11}$ \\\cline{2-4}
     & $h_1^2x_{120,11}$ & & Permanent \\\cline{2-4}
     & $x_{122,13}$ & $d_{3}$ & $h_0h_4x_{106,14}$ \\\hline\hline
    \multirow{3}{*}{12} & $h_0h_6Md_0$ & $d_{2}^{-1}$ & $x_{123,10}$ \\\cline{2-4}
     & $h_0x_{122,11}$ & $d_{3}^{-1}$ & $h_0x_{123,8}$ \\\cline{2-4}
     & $h_5x_{91,11}$ & & Permanent \\\hline\hline
    \multirow{2}{*}{11} & $x_{122,11}+h_6Md_0$ & $d_{3}^{-1}$ & $x_{123,8}$ \\\cline{2-4}
     & $h_6Md_0$ & & Permanent \\\hline\hline
    \multirow{1}{*}{9-10} & \multicolumn{3}{c|}{}\\\hline\hline
    \multirow{1}{*}{8} & $h_1x_{121,7}$ & $d_{6}$ & $?$ \\\hline\hline
    \multirow{1}{*}{0-7} & \multicolumn{3}{c|}{}\\\hline
\end{tabular}}
\caption{The classical Adams spectral sequence of $S^0$ for $s \le 25$ in stem 122}\vspace{1cm}
\label{Table:S122}
\end{table}
%\newpage

\begin{table}
    \centering
\scalebox{0.95}{\begin{tabular}{|c|l|l|l|}\hline
$s$ & Elements & $d_r$ & value\\\hline\hline
\multirow{1}{*}{25} & \multicolumn{3}{c|}{}\\\hline\hline
\multirow{3}{*}{24} & $h_1Ph_1x_{113,18,2}$ & $d_{2}^{-1}$ & $x_{124,22}$ \\\cline{2-4}
 & $d_0^2\Delta h_2^2Mg$ & $d_{2}^{-1}$ & $d_0x_{110,18}$ \\\cline{2-4}
 & $d_0M\!Px_{56,10}$ & $d_{4}$ & $M\!Px_{69,18}$ \\\hline\hline
\multirow{1}{*}{23} & $g^2\Delta^2m$ & $d_{2}$ & $e_0g^3\Delta h_1g$ \\\hline\hline
\multirow{1}{*}{21-22} & \multicolumn{3}{c|}{}\\\hline\hline
\multirow{1}{*}{20} & $x_{123,20}$ & $d_{3}$ & $Ph_1x_{113,18,2}$ \\\hline\hline
\multirow{4}{*}{19} & $e_0g^2x_{66,7}+h_0^6x_{123,13,2}$ & $d_{2}^{-1}$ & $x_{124,17}$ \\\cline{2-4}
 & $h_0^6x_{123,13,2}$ & $d_{2}^{-1}$ & $h_0^3x_{124,14,2}$ \\\cline{2-4}
 & $h_0e_0x_{106,14}+h_0^2h_3x_{116,16}$ & $d_{3}^{-1}$ & $e_0x_{107,12}$ \\\cline{2-4}
 & $h_0^2h_3x_{116,16}$ & $d_{3}$ & $e_0g^2Mg$ \\\hline\hline
\multirow{3}{*}{18} & $h_0^5x_{123,13,2}$ & $d_{2}^{-1}$ & $h_0^2x_{124,14,2}$ \\\cline{2-4}
 & $e_0x_{106,14}+h_0h_3x_{116,16}$ & & Permanent \\\cline{2-4}
 & $h_0h_3x_{116,16}$ & $d_{2}$ & $h_0^4x_{122,16}$ \\\hline\hline
\multirow{4}{*}{17} & $h_0^2x_{123,15}+h_0^4x_{123,13,2}$ & $d_{2}^{-1}$ & $h_0x_{124,14}$ \\\cline{2-4}
 & $h_0^4x_{123,13,2}$ & $d_{2}^{-1}$ & $h_0x_{124,14,2}+h_0x_{124,14}$ \\\cline{2-4}
 & $d_0x_{109,13}$ & & Permanent \\\cline{2-4}
 & $h_3x_{116,16}$ & $d_{2}$ & $h_0^3x_{122,16}$ \\\hline\hline
\multirow{3}{*}{16} & $h_0x_{123,15}+h_0^3x_{123,13,2}$ & $d_{2}^{-1}$ & $x_{124,14}$ \\\cline{2-4}
 & $h_0^3x_{123,13,2}$ & $d_{2}^{-1}$ & $x_{124,14,2}+x_{124,14}$ \\\cline{2-4}
 & $h_1x_{122,15,2}$ & $d_{3}^{-1}$ & $h_4x_{109,12}$ \\\hline\hline
\multirow{3}{*}{15} & $x_{123,15}$ & $d_{4}^{-1}$ & $x_{124,11,2}+x_{124,11}$ \\\cline{2-4}
 & $h_4x_{108,14}$ & $d_{5}^{-1}$ & $h_0x_{124,9}$ \\\cline{2-4}
 & $h_0^2x_{123,13,2}$ & $d_{3}$ & $h_0^2x_{122,16}$ \\\hline\hline
\multirow{3}{*}{14} & $h_0^3x_{123,11}$ & $d_{2}^{-1}$ & $h_0x_{124,11}$ \\\cline{2-4}
 & $h_0x_{123,13,2}$ & $d_{3}$ & $h_0x_{122,16}$ \\\cline{2-4}
 & $\Delta h_2^2x_{93,8}$ & $d_{3}$ & $g^3[H_1]$ \\\hline\hline
\multirow{3}{*}{13} & $h_0^2x_{123,11}$ & $d_{2}^{-1}$ & $x_{124,11}$ \\\cline{2-4}
 & $x_{123,13,2}$ & $d_{3}$ & $x_{122,16}+h_0x_{122,15,2}$ \\\cline{2-4}
 & $x_{123,13}$ & $d_{3}$ & $\Delta h_2^2x_{92,10}$ \\\hline\hline
\multirow{3}{*}{12} & $h_0x_{123,11}+h_0^2h_6[B_4]$ & $d_{3}^{-1}$ & $x_{124,9,2}+h_0x_{124,8}$ \\\cline{2-4}
 & $x_{123,12}$ & $d_{3}^{-1}$ & $x_{124,9}+h_0x_{124,8}$ \\\cline{2-4}
 & $h_0^2h_6[B_4]$ & $d_{5}^{-1}$ & $h_6A$ \\\hline\hline
\multirow{5}{*}{11} & $h_0^2x_{123,9}$ & $d_{2}^{-1}$ & $x_{124,9}$ \\\cline{2-4}
 & $h_5x_{92,10}$ & $d_{5}^{-1}$ & $x_{124,6}$ \\\cline{2-4}
 & $x_{123,11,2}+x_{123,11}+h_0h_6[B_4]$ & $d_{7}$ & $h_1x_{121,17}$ \\\cline{2-4}
 & $x_{123,11}+h_0h_6[B_4]$ & $d_{3}$ & $h_0x_{122,13}$ \\\cline{2-4}
 & $h_0h_6[B_4]$ & $d_{2}$ & $h_0^2h_6Md_0$ \\\hline\hline
\multirow{3}{*}{10} & $h_0x_{123,9}$ & $d_{2}^{-1}$ & $x_{124,8}$ \\\cline{2-4}
 & $x_{123,10}+h_6[B_4]$ & $d_{3}^{-1}$ & $x_{124,7}$ \\\cline{2-4}
 & $h_6[B_4]$ & $d_{2}$ & $h_0h_6Md_0$ \\\hline\hline
\multirow{2}{*}{9} & $x_{123,9}+h_0x_{123,8}$ & $d_{12}$ & $?$ \\\cline{2-4}
 & $h_0x_{123,8}$ & $d_{3}$ & $h_0x_{122,11}+h_0h_6Md_0$ \\\hline\hline
\multirow{1}{*}{8} & $x_{123,8}$ & $d_{3}$ & $x_{122,11}+h_6Md_0$ \\\hline\hline
\multirow{1}{*}{0-7} & \multicolumn{3}{c|}{}\\\hline
\end{tabular}}
\caption{The classical Adams spectral sequence of $S^0$ for $s \le 25$ in stem 123}
\label{Table:S123}
\end{table}

%\newpage

\begin{table}
    \centering
\scalebox{0.95}{\begin{tabular}{|c|l|l|l|}\hline
$s$ & Elements & $d_r$ & value\\\hline\hline
\multirow{3}{*}{25} & $h_0^{11}x_{124,14,2}$ & $d_{2}^{-1}$ & $h_0^9x_{125,14}$ \\\cline{2-4}
 & $ix_{101,18}$ & $d_{2}$ & $d_0^2\Delta h_2^2x_{65,13}+h_0Pd_0x_{101,18}$ \\\cline{2-4}
 & $d_0e_0\Delta^3h_1g$ & $d_{2}$ & $d_0^2g^3m$ \\\hline\hline
\multirow{2}{*}{24} & $h_0^{10}x_{124,14,2}$ & $d_{2}^{-1}$ & $h_0^8x_{125,14}$ \\\cline{2-4}
 & $h_0^2d_0x_{110,18}$ & $d_{2}^{-1}$ & $h_0e_0x_{108,17}$ \\\hline\hline
\multirow{4}{*}{23} & $h_0^9x_{124,14,2}$ & $d_{2}^{-1}$ & $h_0^7x_{125,14}$ \\\cline{2-4}
 & $d_0g\Delta h_2^2[B_4]+h_0d_0x_{110,18}$ & $d_{2}^{-1}$ & $e_0x_{108,17}$ \\\cline{2-4}
 & $h_0d_0x_{110,18}$ & $d_{3}^{-1}$ & $x_{125,20}$ \\\cline{2-4}
 & $d_0x_{110,19}$ & & Permanent \\\hline\hline
\multirow{4}{*}{22} & $h_0^8x_{124,14,2}$ & $d_{2}^{-1}$ & $h_0^6x_{125,14}$ \\\cline{2-4}
 & $g^2\Delta^2t$ & $d_{3}^{-1}$ & $gx_{105,15}$ \\\cline{2-4}
 & $x_{124,22}$ & $d_{2}$ & $h_1Ph_1x_{113,18,2}$ \\\cline{2-4}
 & $d_0x_{110,18}$ & $d_{2}$ & $d_0^2\Delta h_2^2Mg$ \\\hline\hline
\multirow{2}{*}{21} & $h_0^7x_{124,14,2}$ & $d_{2}^{-1}$ & $h_0^5x_{125,14}$ \\\cline{2-4}
 & $h_0d_0^2[\Delta\Delta_1g]$ & $d_{2}^{-1}$ & $d_0gx_{91,11}$ \\\hline\hline
\multirow{2}{*}{20} & $h_0^6x_{124,14,2}$ & $d_{2}^{-1}$ & $h_0^4x_{125,14}$ \\\cline{2-4}
 & $d_0^2[\Delta\Delta_1g]$ & $d_{5}^{-1}$ & $h_1x_{124,14}$ \\\hline\hline
\multirow{2}{*}{19} & $h_0^5x_{124,14,2}$ & $d_{2}^{-1}$ & $h_0^3x_{125,14}$ \\\cline{2-4}
 & $d_0x_{110,15}$ & & Permanent \\\hline\hline
\multirow{2}{*}{18} & $h_0x_{124,17}+h_0^4x_{124,14,2}$ & $d_{2}^{-1}$ & $x_{125,16}$ \\\cline{2-4}
 & $h_0^4x_{124,14,2}$ & $d_{2}^{-1}$ & $h_0^2x_{125,14}$ \\\hline\hline
\multirow{4}{*}{17} & $h_0^3x_{124,14}$ & $d_{2}^{-1}$ & $x_{125,15}$ \\\cline{2-4}
 & $h_0^2x_{124,15}$ & $d_{2}^{-1}$ & $h_0x_{125,14}$ \\\cline{2-4}
 & $x_{124,17}$ & $d_{2}$ & $e_0g^2x_{66,7}+h_0^6x_{123,13,2}$ \\\cline{2-4}
 & $h_0^3x_{124,14,2}$ & $d_{2}$ & $h_0^6x_{123,13,2}$ \\\hline\hline
\multirow{5}{*}{16} & $h_0x_{124,15}$ & $d_{2}^{-1}$ & $x_{125,14}$ \\\cline{2-4}
 & $h_1x_{123,15}$ & $d_{3}^{-1}$ & $h_3x_{118,12}$ \\\cline{2-4}
 & $h_0^2x_{124,14}$ & & Permanent \\\cline{2-4}
 & $e_0x_{107,12}$ & $d_{3}$ & $e_0g^2x_{66,7}+h_0e_0x_{106,14}+h_0^2h_3x_{116,16}$ \\\cline{2-4}
 & $h_0^2x_{124,14,2}$ & $d_{2}$ & $h_0^5x_{123,13,2}$ \\\hline\hline
\multirow{4}{*}{15} & $x_{124,15}$ & $d_{4}^{-1}$ & $h_6x_{62,10}$ \\\cline{2-4}
 & $h_3^2x_{110,13}+h_0x_{124,14}$ & & Permanent \\\cline{2-4}
 & $h_0x_{124,14}$ & $d_{2}$ & $h_0^2x_{123,15}+h_0^4x_{123,13,2}$ \\\cline{2-4}
 & $h_0x_{124,14,2}$ & $d_{2}$ & $h_0^2x_{123,15}$ \\\hline\hline
\multirow{5}{*}{14} & $h_1x_{123,13}$ & $d_{2}^{-1}$ & $x_{125,12}$ \\\cline{2-4}
 & $h_1x_{123,13,2}$ & $d_{2}^{-1}$ & $x_{125,12,2}$ \\\cline{2-4}
 & $\Delta h_2^2x_{94,8}$ & & Permanent \\\cline{2-4}
 & $x_{124,14}$ & $d_{2}$ & $h_0x_{123,15}+h_0^3x_{123,13,2}$ \\\cline{2-4}
 & $x_{124,14,2}$ & $d_{2}$ & $h_0x_{123,15}$ \\\hline\hline
\multirow{4}{*}{13} & $h_0^5x_{124,8}$ & $d_{2}^{-1}$ & $h_0^3x_{125,8}$ \\\cline{2-4}
 & $[H_1](\Delta e_1+C_0+h_0^6h_5^2)$ & $d_{3}^{-1}$ & $x_{125,10,2}$ \\\cline{2-4}
 & $e_0\Delta h_6g$ & & Permanent \\\cline{2-4}
 & $h_4x_{109,12}$ & $d_{3}$ & $h_1x_{122,15,2}$ \\\hline
\end{tabular}}
\caption{The classical Adams spectral sequence of $S^0$ for $13 \le s \le 25$ in stem 124}
\label{Table:S124.13}
\end{table}

%\newpage

\begin{table}
    \centering
\scalebox{0.9}{\begin{tabular}{|c|l|l|l|}\hline %%%%%%%%%%%%%%%%%%%%%%
\multirow{5}{*}{12} & $h_0x_{124,11,2}+h_0x_{124,11}$ & $d_{2}^{-1}$ & $x_{125,10}$ \\\cline{2-4}
 & $h_0^2x_{124,10,2}+h_0^4x_{124,8}$ & $d_{2}^{-1}$ & $h_0x_{125,9,2}$ \\\cline{2-4}
 & $h_0^4x_{124,8}$ & $d_{2}^{-1}$ & $h_0^2x_{125,8}$ \\\cline{2-4}
 & $h_1x_{123,11,2}$ & $d_{3}^{-1}$ & $x_{125,9}$ \\\cline{2-4}
 & $h_0x_{124,11}$ & $d_{2}$ & $h_0^3x_{123,11}$ \\\hline\hline
\multirow{5}{*}{11} & $h_0x_{124,10,2}+h_0^3x_{124,8}$ & $d_{2}^{-1}$ & $x_{125,9,2}$ \\\cline{2-4}
 & $h_0^3x_{124,8}$ & $d_{2}^{-1}$ & $h_0x_{125,8}$ \\\cline{2-4}
 & $x_{124,11,3}$ & $d_{3}^{-1}$ & $x_{125,8,2}$ \\\cline{2-4}
 & $x_{124,11,2}+x_{124,11}$ & $d_{4}$ & $x_{123,15}$ \\\cline{2-4}
 & $x_{124,11}$ & $d_{2}$ & $h_0^2x_{123,11}$ \\\hline\hline
\multirow{5}{*}{10} & $h_1x_{123,9}+h_0^2x_{124,8}$ & $d_{2}^{-1}$ & $x_{125,8}$ \\\cline{2-4}
 & $x_{124,10,2}+h_0x_{124,9}$ & $d_{4}^{-1}$ & $h_6[H_1]$ \\\cline{2-4}
 & $x_{124,10}+h_0^2x_{124,8}$ & $d_{4}^{-1}$ & $h_6[H_1]+h_0x_{125,5}$ \\\cline{2-4}
 & $h_0^2x_{124,8}$ & & Permanent \\\cline{2-4}
 & $h_0x_{124,9}$ & $d_{5}$ & $h_4x_{108,14}$ \\\hline\hline
\multirow{3}{*}{9} & $x_{124,9,2}+h_0x_{124,8}$ & $d_{3}$ & $h_0x_{123,11}+h_0^2h_6[B_4]$ \\\cline{2-4}
 & $x_{124,9}+h_0x_{124,8}$ & $d_{3}$ & $x_{123,12}$ \\\cline{2-4}
 & $h_0x_{124,8}$ & $d_{2}$ & $h_0^2x_{123,9}$ \\\hline\hline
\multirow{1}{*}{8} & $x_{124,8}$ & $d_{2}$ & $h_0x_{123,9}$ \\\hline\hline
\multirow{3}{*}{7} & $h_0x_{124,6}$ & $d_{2}^{-1}$ & $x_{125,5}$ \\\cline{2-4}
 & $h_6A$ & $d_{5}$ & $h_0^2h_6[B_4]$ \\\cline{2-4}
 & $x_{124,7}$ & $d_{3}$ & $x_{123,10}+h_6[B_4]$ \\\hline\hline
\multirow{1}{*}{6} & $x_{124,6}$ & $d_{5}$ & $h_5x_{92,10}$ \\\hline\hline
\multirow{1}{*}{0-5} & \multicolumn{3}{c|}{}\\\hline
\end{tabular}}
\caption{The classical Adams spectral sequence of $S^0$ for $s \le 12$ in stem 124}
\label{Table:S124.12}
\end{table}

\begin{table}
    \centering
\scalebox{0.9}{\begin{tabular}{|c|l|l|l|}\hline
$s$ & Elements & $d_r$ & value\\\hline\hline
\multirow{4}{*}{25} & $d_0^2e_0g[B_4]$ & $d_{4}^{-1}$ & $h_1x_{125,20}$ \\\cline{2-4}
 & $x_{125,25,2}+x_{125,25}+g^4\Delta h_1g$ & $d_{4}^{-1}$ & $x_{126,21}+\text{possibly }h_1x_{125,20}$ \\\cline{2-4}
 & $g^4\Delta h_1g$ & & Permanent \\\cline{2-4}
 & $x_{125,25}$ & $d_{2}$ & $h_0x_{124,26}$ \\\hline\hline
\multirow{1}{*}{24} & $e_0g\Delta^2g^2$ & $d_{2}$ & $d_0g^4\Delta h_2^2$ \\\hline\hline
\multirow{2}{*}{23} & $h_0^2e_0x_{108,17}$ & $d_{3}$ & $d_0^3e_0Mg$ \\\cline{2-4}
 & $h_0^9x_{125,14}$ & $d_{2}$ & $h_0^{11}x_{124,14,2}$ \\\hline\hline
\multirow{4}{*}{22} & $g^3Mg$ & $d_{4}^{-1}$ & $x_{126,18}$ \\\cline{2-4}
 & $ix_{102,15}+h_0^8x_{125,14}$ & $d_{4}^{-1}$ & $gx_{106,14}+e_0x_{109,14,2}$ \\\cline{2-4}
 & $h_0^8x_{125,14}$ & $d_{2}$ & $h_0^{10}x_{124,14,2}$ \\\cline{2-4}
 & $h_0e_0x_{108,17}$ & $d_{2}$ & $h_0^2d_0x_{110,18}$ \\\hline\hline
\multirow{3}{*}{21} & $x_{125,21}$ & $d_{4}^{-1}$ & $d_0x_{112,13}$ \\\cline{2-4}
 & $h_0^7x_{125,14}$ & $d_{2}$ & $h_0^9x_{124,14,2}$ \\\cline{2-4}
 & $e_0x_{108,17}$ & $d_{2}$ & $d_0g\Delta h_2^2[B_4]+h_0d_0x_{110,18}$ \\\hline\hline
\multirow{3}{*}{20} & $h_0d_0gx_{91,11}$ & $d_{2}^{-1}$ & $x_{126,18,2}$ \\\cline{2-4}
 & $x_{125,20}$ & $d_{3}$ & $d_0g\Delta h_2^2[B_4]$ \\\cline{2-4}
 & $h_0^6x_{125,14}$ & $d_{2}$ & $h_0^8x_{124,14,2}$ \\\hline
 \end{tabular}}
\caption{The classical Adams spectral sequence of $S^0$ for $20 \le s \le 25$ in stem 125}
\label{Table:S125.20}
\end{table}

\begin{table}
    \centering
\scalebox{0.85}{\begin{tabular}{|c|l|l|l|}\hline %%%%%%%%%%%%%%%%%%%%%%
\multirow{3}{*}{19} & $gx_{105,15}$ & $d_{3}$ & $g^2\Delta^2t$ \\\cline{2-4}
 & $h_0^5x_{125,14}$ & $d_{2}$ & $h_0^7x_{124,14,2}$ \\\cline{2-4}
 & $d_0gx_{91,11}$ & $d_{2}$ & $h_0d_0^2[\Delta\Delta_1g]$ \\\hline\hline
\multirow{2}{*}{18} & $d_0^2x_{97,10}$ & $d_{5}^{-1}$ & $h_1x_{125,12,2}$ \\\cline{2-4}
 & $h_0^4x_{125,14}$ & $d_{2}$ & $h_0^6x_{124,14,2}$ \\\hline\hline
\multirow{2}{*}{17} & $h_0^2Q_2x_{68,8}$ & $d_{2}^{-1}$ & $h_0D_2x_{68,8}$ \\\cline{2-4}
 & $h_0^3x_{125,14}$ & $d_{2}$ & $h_0^5x_{124,14,2}$ \\\hline\hline
\multirow{4}{*}{16} & $h_0Q_2x_{68,8}$ & $d_{2}^{-1}$ & $D_2x_{68,8}$ \\\cline{2-4}
 & $h_1x_{124,15}$ & $d_{4}^{-1}$ & $h_0^2x_{126,10}$ \\\cline{2-4}
 & $x_{125,16}$ & $d_{2}$ & $h_0x_{124,17}+h_0^4x_{124,14,2}$ \\\cline{2-4}
 & $h_0^2x_{125,14}$ & $d_{2}$ & $h_0^4x_{124,14,2}$ \\\hline\hline
\multirow{5}{*}{15} & $h_0^3x_{125,12}$ & $d_{2}^{-1}$ & $h_0h_3x_{119,11}$ \\\cline{2-4}
 & $Q_2x_{68,8}$ & $d_{4}^{-1}$ & $h_1x_{125,10}$ \\\cline{2-4}
 & $h_1x_{124,14}$ & $d_{5}$ & $d_0^2[\Delta\Delta_1g]$ \\\cline{2-4}
 & $x_{125,15}$ & $d_{2}$ & $h_0^3x_{124,14}$ \\\cline{2-4}
 & $h_0x_{125,14}$ & $d_{2}$ & $h_0^2x_{124,15}$ \\\hline\hline
\multirow{3}{*}{14} & $h_0^2x_{125,12}$ & $d_{2}^{-1}$ & $h_3x_{119,11}$ \\\cline{2-4}
 & $h_1h_4x_{109,12}$ & & Permanent \\\cline{2-4}
 & $x_{125,14}$ & $d_{2}$ & $h_0x_{124,15}$ \\\hline\hline
\multirow{5}{*}{13} & $h_0^4x_{125,9,2}$ & $d_{2}^{-1}$ & $x_{126,11}$ \\\cline{2-4}
 & $h_0^5x_{125,8}$ & $d_{2}^{-1}$ & $h_0x_{126,10}$ \\\cline{2-4}
 & $nx_{94,8}$ & $d_{4}^{-1}$ & $h_1x_{125,8}$ \\\cline{2-4}
 & $h_0x_{125,12}$ & $d_{4}^{-1}$ & $h_1x_{125,8,2}$ \\\cline{2-4}
 & $h_3x_{118,12}$ & $d_{3}$ & $h_1x_{123,15}$ \\\hline\hline
\multirow{5}{*}{12} & $h_0h_6x_{62,10}$ & $d_{2}^{-1}$ & $x_{126,10}$ \\\cline{2-4}
 & $h_0^3x_{125,9,2}+h_0^4x_{125,8}$ & $d_{3}^{-1}$ & $x_{126,9}$ \\\cline{2-4}
 & $h_0^4x_{125,8}$ & $d_{3}^{-1}$ & $x_{126,9}+h_0x_{126,8,3}$ \\\cline{2-4}
 & $x_{125,12}$ & $d_{2}$ & $h_1x_{123,13}$ \\\cline{2-4}
 & $x_{125,12,2}$ & $d_{2}$ & $h_1x_{123,13,2}$ \\\hline\hline
\multirow{5}{*}{11} & $h_1x_{124,10,2}$ & $d_{3}^{-1}$ & $x_{126,8}$ \\\cline{2-4}
 & $h_1x_{124,10}$ & $d_{3}^{-1}$ & $x_{126,8,2}$ \\\cline{2-4}
 & $h_0^2x_{125,9,2}$ & $d_{5}$ & $?$ \\\cline{2-4}
 & $h_6x_{62,10}$ & $d_{4}$ & $x_{124,15}$ \\\cline{2-4}
 & $h_0^3x_{125,8}$ & $d_{2}$ & $h_0^5x_{124,8}$ \\\hline\hline
\multirow{5}{*}{10} & $h_0x_{125,9}$ & $d_{2}^{-1}$ & $x_{126,8,3}$ \\\cline{2-4}
 & $x_{125,10,2}$ & $d_{3}$ & $[H_1](\Delta e_1+C_0+h_0^6h_5^2)$ \\\cline{2-4}
 & $x_{125,10}$ & $d_{2}$ & $h_0x_{124,11,2}+h_0x_{124,11}$ \\\cline{2-4}
 & $h_0x_{125,9,2}$ & $d_{2}$ & $h_0^2x_{124,10,2}+h_0^4x_{124,8}$ \\\cline{2-4}
 & $h_0^2x_{125,8}$ & $d_{2}$ & $h_0^4x_{124,8}$ \\\hline\hline
\multirow{5}{*}{9} & $h_6(\Delta e_1+C_0+h_0^6h_5^2)$ & & Permanent \\\cline{2-4}
 & $h_5x_{94,8}$ & $d_{7}$ & $?$ \\\cline{2-4}
 & $x_{125,9}$ & $d_{3}$ & $h_1x_{123,11,2}$ \\\cline{2-4}
 & $x_{125,9,2}$ & $d_{2}$ & $h_0x_{124,10,2}+h_0^3x_{124,8}$ \\\cline{2-4}
 & $h_0x_{125,8}$ & $d_{2}$ & $h_0^3x_{124,8}$ \\\hline\hline
\multirow{2}{*}{8} & $x_{125,8,2}$ & $d_{3}$ & $x_{124,11,3}$ \\\cline{2-4}
 & $x_{125,8}$ & $d_{2}$ & $h_1x_{123,9}+h_0^2x_{124,8}$ \\\hline\hline
\multirow{1}{*}{7} & $h_0^2x_{125,5}$ & $d_{3}^{-1}$ & $x_{126,4}$ \\\hline\hline
\multirow{2}{*}{6} & $h_6[H_1]$ & $d_{4}$ & $x_{124,10,2}+h_0x_{124,9}$ \\\cline{2-4}
 & $h_0x_{125,5}$ & $d_{4}$ & $x_{124,10,2}+x_{124,10}+h_0x_{124,9}+h_0^2x_{124,8}$ \\\hline\hline
\multirow{1}{*}{5} & $x_{125,5}$ & $d_{2}$ & $h_0x_{124,6}$ \\\hline
\hline
\multirow{1}{*}{0-4} & \multicolumn{3}{c|}{}\\\hline
 \end{tabular}}
\caption{The classical Adams spectral sequence of $S^0$ for $s \le 19$ in stem 125}
\label{Table:S125.19}
\end{table}

\begin{table}
    \centering
\scalebox{0.85}{\begin{tabular}{|c|l|l|l|}\hline
$s$ & Elements & $d_r$ & value\\\hline\hline
\multirow{1}{*}{25} & $h_0^7x_{126,18}$ & $d_{3}^{-1}$ & $h_0^{21}h_7$ \\\hline\hline
\multirow{4}{*}{24} & $d_0e_0\Delta h_2^2Mg$ & $d_{2}^{-1}$ & $d_0x_{113,18}$ \\\cline{2-4}
 & $h_0^6x_{126,18}$ & $d_{3}^{-1}$ & $h_0^{20}h_7$ \\\cline{2-4}
 & $g^4\Delta h_2c_1$ & $d_{3}^{-1}$ & $g^3C^{\prime\prime}$ \\\cline{2-4}
 & $d_0Pd_0M^2$ & $d_{4}^{-1}$ & $d_0e_0[\Delta\Delta_1g]$ \\\hline\hline
\multirow{2}{*}{23} & $h_0^5x_{126,18}$ & $d_{3}^{-1}$ & $h_0^{19}h_7$ \\\cline{2-4}
 & $x_{126,23}$ & $d_{4}^{-1}$ & $e_0x_{110,15}$ \\\hline\hline
\multirow{2}{*}{22} & $h_0x_{126,21}+h_0^4x_{126,18}$ & $d_{3}^{-1}$ & $h_1x_{126,18,2}$ \\\cline{2-4}
 & $h_0^4x_{126,18}$ & $d_{3}^{-1}$ & $h_0^{18}h_7$ \\\hline\hline
\multirow{3}{*}{21} & $h_0^3x_{126,18}$ & $d_{3}^{-1}$ & $h_0^{17}h_7$ \\\cline{2-4}
 & $h_1x_{125,20}$ & $d_{4}$ & $d_0^2e_0g[B_4]$ \\\cline{2-4}
 & $x_{126,21}$ & $d_{4}$ & $x_{125,25,2}+x_{125,25}+g^4\Delta h_1g+\text{possibly }d_0^2e_0gB_4$ \\\hline\hline
\multirow{2}{*}{20} & $h_0^2x_{126,18}$ & $d_{3}^{-1}$ & $h_0^{16}h_7$ \\\cline{2-4}
 & $d_0x_{112,16}$ & $d_{5}^{-1}$ & $x_{127,15}$ \\\hline\hline
\multirow{2}{*}{19} & $h_0x_{126,18}$ & $d_{3}^{-1}$ & $h_0^{15}h_7$ \\\cline{2-4}
 & $g^3x_{66,7}$ & $d_{3}^{-1}$ & $gx_{107,12}$ \\\hline\hline
\multirow{4}{*}{18} & $x_{126,18}+e_0x_{109,14,2}$ & $d_{7}$ & $?$ \\\cline{2-4}
 & $e_0x_{109,14,2}$ & $d_{4}$ & $g^3Mg$ \\\cline{2-4}
 & $gx_{106,14}$ & $d_{4}$ & $ix_{102,15}+g^3Mg+h_0^8x_{125,14}$ \\\cline{2-4}
 & $x_{126,18,2}$ & $d_{2}$ & $h_0d_0gx_{91,11}$ \\\hline\hline
\multirow{4}{*}{17} & $h_0^{15}h_6^2$ & $d_{2}^{-1}$ & $h_0^{14}h_7$ \\\cline{2-4}
 & $h_1^2x_{124,15}$ & $d_{2}^{-1}$ & $h_6x_{64,14}$ \\\cline{2-4}
 & $x_{126,17}$ & $d_{8}$ & $?$ \\\cline{2-4}
 & $d_0x_{112,13}$ & $d_{4}$ & $x_{125,21}$ \\\hline\hline
\multirow{3}{*}{16} & $h_0^{14}h_6^2$ & $d_{2}^{-1}$ & $h_0^{13}h_7$ \\\cline{2-4}
 & $h_0^2D_2x_{68,8}$ & $d_{3}^{-1}$ & $x_{127,13}$ \\\cline{2-4}
 & $h_1^2x_{124,14}$ & $d_{6}^{-1}$ & $h_2x_{124,9}+h_0^2x_{127,8}$ \\\hline\hline
\multirow{2}{*}{15} & $h_0^{13}h_6^2$ & $d_{2}^{-1}$ & $h_0^{12}h_7$ \\\cline{2-4}
 & $h_0D_2x_{68,8}$ & $d_{2}$ & $h_0^2Q_2x_{68,8}$ \\\hline\hline
\multirow{4}{*}{14} & $h_0^{12}h_6^2$ & $d_{2}^{-1}$ & $h_0^{11}h_7$ \\\cline{2-4}
 & $x_{126,14}$ & $d_{4}^{-1}$ & $h_0^2x_{127,8}$ \\\cline{2-4}
 & $h_1h_3x_{118,12}$ & $d_{5}^{-1}$ & $h_1x_{126,8,2}$ \\\cline{2-4}
 & $D_2x_{68,8}$ & $d_{2}$ & $h_0Q_2x_{68,8}$ \\\hline\hline
\multirow{3}{*}{13} & $h_0^{11}h_6^2$ & $d_{2}^{-1}$ & $h_0^{10}h_7$ \\\cline{2-4}
 & $h_1x_{125,12,2}$ & $d_{5}$ & $d_0^2x_{97,10}$ \\\cline{2-4}
 & $h_0h_3x_{119,11}$ & $d_{2}$ & $h_0^3x_{125,12}$ \\\hline\hline
\multirow{5}{*}{12} & $d_1x_{94,8}$ & $d_{2}^{-1}$ & $x_{127,10}$ \\\cline{2-4}
 & $h_0x_{126,11}$ & $d_{2}^{-1}$ & $h_3x_{120,9}$ \\\cline{2-4}
 & $h_0^{10}h_6^2$ & $d_{2}^{-1}$ & $h_0^9h_7$ \\\cline{2-4}
 & $h_0^2x_{126,10}$ & $d_{4}$ & $h_1x_{124,15}$ \\\cline{2-4}
 & $h_3x_{119,11}$ & $d_{2}$ & $h_0^2x_{125,12}$ \\\hline\hline
\multirow{6}{*}{11} & $h_0^2x_{126,9}$ & $d_{2}^{-1}$ & $h_0x_{127,8}$ \\\cline{2-4}
 & $h_0^9h_6^2$ & $d_{2}^{-1}$ & $h_0^8h_7$ \\\cline{2-4}
 & $h_1x_{125,10,2}+h_1x_{125,10}$ & & Permanent \\\cline{2-4}
 & $h_1x_{125,10}$ & $d_{4}$ & $Q_2x_{68,8}$ \\\cline{2-4}
 & $x_{126,11}$ & $d_{2}$ & $h_0^4x_{125,9,2}$ \\\cline{2-4}
 & $h_0x_{126,10}$ & $d_{2}$ & $h_0^5x_{125,8}$ \\\hline
 \end{tabular}}
\caption{The classical Adams spectral sequence of $S^0$ for $11 \le s \le 25$ in stem 126}
\label{Table:S126.11}
\end{table}

\begin{table}
    \centering
\scalebox{0.85}{\begin{tabular}{|c|l|l|l|}\hline %%%%%%%%%%%%%%%%%%%%%%
\multirow{5}{*}{10} & $h_0x_{126,9}$ & $d_{2}^{-1}$ & $x_{127,8}$ \\\cline{2-4}
 & $h_0^2x_{126,8}$ & $d_{2}^{-1}$ & $h_0x_{127,7}$ \\\cline{2-4}
 & $h_0^8h_6^2$ & $d_{2}^{-1}$ & $h_0^7h_7$ \\\cline{2-4}
 & $h_0^2x_{126,8,3}$ & & Permanent \\\cline{2-4}
 & $x_{126,10}$ & $d_{2}$ & $h_0h_6x_{62,10}$ \\\hline\hline
\multirow{6}{*}{9} & $h_0x_{126,8}$ & $d_{2}^{-1}$ & $x_{127,7}$ \\\cline{2-4}
 & $h_0^7h_6^2$ & $d_{2}^{-1}$ & $h_0^6h_7$ \\\cline{2-4}
 & $h_1x_{125,8}$ & $d_{4}$ & $nx_{94,8}$ \\\cline{2-4}
 & $h_1x_{125,8,2}$ & $d_{4}$ & $h_0x_{125,12}$ \\\cline{2-4}
 & $x_{126,9}$ & $d_{3}$ & $h_0^3x_{125,9,2}+h_0^4x_{125,8}$ \\\cline{2-4}
 & $h_0x_{126,8,3}$ & $d_{3}$ & $h_0^3x_{125,9,2}$ \\\hline\hline
\multirow{6}{*}{8} & $h_0^6h_6^2$ & $d_{2}^{-1}$ & $h_0^5h_7$ \\\cline{2-4}
 & $h_6(C^{\prime}+X_2)$ & $d_{17}$ & $?$ \\\cline{2-4}
 & $x_{126,8,4}+x_{126,8}$ & $d_{6}$ & $?$ \\\cline{2-4}
 & $x_{126,8}$ & $d_{3}$ & $h_1x_{124,10,2}$ \\\cline{2-4}
 & $x_{126,8,2}$ & $d_{3}$ & $h_1x_{124,10}$ \\\cline{2-4}
 & $x_{126,8,3}$ & $d_{2}$ & $h_0x_{125,9}$ \\\hline\hline
\multirow{2}{*}{7} & $h_0^5h_6^2$ & $d_{2}^{-1}$ & $h_0^4h_7$ \\\cline{2-4}
 & $h_1h_6[H_1]$ & $d_{18}$ & $?$ \\\hline\hline
\multirow{2}{*}{6} & $h_0^4h_6^2$ & $d_{2}^{-1}$ & $h_0^3h_7$ \\\cline{2-4}
 & $x_{126,6}$ & $d_{3}$ & $h_5x_{94,8} +\text{possibly }h_6(\Delta e_1 + C_0 + h_0^6h_5^2)$ \\\hline\hline
\multirow{1}{*}{5} & $h_0^3h_6^2$ & $d_{2}^{-1}$ & $h_0^2h_7$ \\\hline\hline
\multirow{2}{*}{4} & $h_0^2h_6^2$ & $d_{2}^{-1}$ & $h_0h_7$ \\\cline{2-4}
 & $x_{126,4}$ & $d_{3}$ & $h_0^2x_{125,5}$ \\\hline\hline
\multirow{1}{*}{3} & $h_0h_6^2$ & $d_{2}^{-1}$ & $h_7$ \\\hline\hline
\multirow{1}{*}{2} & $h_6^2$ & $d_{7}$ & $?$ \\\hline\hline
\multirow{1}{*}{0-1} & \multicolumn{3}{c|}{}\\\hline
\end{tabular}}
\caption{The classical Adams spectral sequence of $S^0$ for $s \le 10$ in stem 126}
\label{Table:S126.10}
\end{table}

\begin{table}
    \centering
\scalebox{0.85}{\begin{tabular}{|c|l|l|l|}\hline
$s$ & Elements & $d_r$ & value\\\hline\hline
\multirow{3}{*}{25} & $h_0^{24}h_7$ & $d_{3}$ & $h_0^{10}x_{126,18}$ \\\cline{2-4}
 & $ix_{104,18}$ & $d_{2}$ & $d_0^3x_{84,15,2}+h_0d_0x_{112,22}$ \\\cline{2-4}
 & $d_0g\Delta^3h_1g$ & $d_{2}$ & $d_0e_0g^3m$ \\\hline\hline
\multirow{3}{*}{24} & $h_0^2d_0x_{113,18}$ & $d_{2}^{-1}$ & $h_0gx_{108,17}$ \\\cline{2-4}
 & $h_1x_{126,23}$ & $d_{3}^{-1}$ & $x_{128,21}$ \\\cline{2-4}
 & $h_0^{23}h_7$ & $d_{3}$ & $h_0^9x_{126,18}$ \\\hline\hline
\multirow{4}{*}{23} & $e_0g\Delta h_2^2[B_4]+h_0d_0x_{113,18}$ & $d_{2}^{-1}$ & $gx_{108,17}$ \\\cline{2-4}
 & $h_0d_0x_{113,18}$ & $d_{4}$ & $d_0^3x_{84,15,2}$ \\\cline{2-4}
 & $h_0^{22}h_7$ & $d_{3}$ & $h_0^8x_{126,18}$ \\\cline{2-4}
 & $d_0Pd_0x_{91,11}$ & $d_{3}$ & $h_1x_{125,25}$ \\\hline\hline
\multirow{3}{*}{22} & $d_0x_{113,18,2}$ & $d_{4}^{-1}$ & $d_0e_0x_{97,10}$ \\\cline{2-4}
 & $h_0^{21}h_7$ & $d_{3}$ & $h_0^7x_{126,18}$ \\\cline{2-4}
 & $d_0x_{113,18}$ & $d_{2}$ & $d_0e_0\Delta h_2^2Mg$ \\\hline\hline
\multirow{4}{*}{21} & $h_0^6x_{127,15}$ & $d_{2}^{-1}$ & $h_0^5x_{128,14}$ \\\cline{2-4}
 & $x_{127,21}+g^3C^{\prime\prime}$ & & Permanent \\\cline{2-4}
 & $h_0^{20}h_7$ & $d_{3}$ & $h_0^6x_{126,18}$ \\\cline{2-4}
 & $g^3C^{\prime\prime}$ & $d_{3}$ & $g^4\Delta h_2c_1$ \\\hline
 \end{tabular}}
\caption{The classical Adams spectral sequence of $S^0$ for $21 \le s \le 25$ in stem 127}
\label{Table:S127.21}
\end{table}

\begin{table}
    \centering
\scalebox{0.9}{\begin{tabular}{|c|l|l|l|}\hline %%%%%%%%%%%%%%%%%%%%%%
\multirow{3}{*}{20} & $h_0^5x_{127,15}$ & $d_{2}^{-1}$ & $h_0^4x_{128,14}$ \\\cline{2-4}
 & $d_0e_0[\Delta\Delta_1g]$ & $d_{4}$ & $d_0Pd_0M^2$ \\\cline{2-4}
 & $h_0^{19}h_7$ & $d_{3}$ & $h_0^5x_{126,18}$ \\\hline\hline
\multirow{5}{*}{19} & $h_0^4x_{127,15}$ & $d_{2}^{-1}$ & $h_0^3x_{128,14}$ \\\cline{2-4}
 & $h_1x_{126,18}$ & & Permanent \\\cline{2-4}
 & $e_0x_{110,15}$ & $d_{4}$ & $x_{126,23}$ \\\cline{2-4}
 & $h_1x_{126,18,2}$ & $d_{3}$ & $h_0x_{126,21}+h_0^4x_{126,18}$ \\\cline{2-4}
 & $h_0^{18}h_7$ & $d_{3}$ & $h_0^4x_{126,18}$ \\\hline\hline
\multirow{5}{*}{18} & $h_0^3x_{127,15}$ & $d_{2}^{-1}$ & $h_0^2x_{128,14}$ \\\cline{2-4}
 & $h_0^3h_6x_{64,14}$ & $d_{2}^{-1}$ & $h_0^2h_6x_{65,13}$ \\\cline{2-4}
 & $g^2\Delta h_1H_1$ & $d_{3}^{-1}$ & $gx_{108,11}$ \\\cline{2-4}
 & $h_1x_{126,17}$ & & Permanent \\\cline{2-4}
 & $h_0^{17}h_7$ & $d_{3}$ & $h_0^3x_{126,18}$ \\\hline\hline
\multirow{5}{*}{17} & $h_0^2h_2x_{124,14}$ & $d_{2}^{-1}$ & $x_{128,15}$ \\\cline{2-4}
 & $h_0^2x_{127,15}$ & $d_{2}^{-1}$ & $x_{128,15}+h_0x_{128,14}$ \\\cline{2-4}
 & $h_0^2h_6x_{64,14}$ & $d_{2}^{-1}$ & $h_0h_6x_{65,13}$ \\\cline{2-4}
 & $gx_{107,13}$ & $d_{4}^{-1}$ & $h_0h_3x_{121,11}$ \\\cline{2-4}
 & $h_0^{16}h_7$ & $d_{3}$ & $h_0^2x_{126,18}$ \\\hline\hline
\multirow{6}{*}{16} & $h_0x_{127,15}+h_0h_2x_{124,14}$ & $d_{2}^{-1}$ & $x_{128,14}$ \\\cline{2-4}
 & $h_0h_6x_{64,14}$ & $d_{2}^{-1}$ & $h_6x_{65,13}$ \\\cline{2-4}
 & $h_0h_2x_{124,14}$ & & Permanent \\\cline{2-4}
 & $x_{127,16}$ & & Permanent \\\cline{2-4}
 & $h_0^{15}h_7$ & $d_{3}$ & $h_0x_{126,18}$ \\\cline{2-4}
 & $gx_{107,12}$ & $d_{3}$ & $g^3x_{66,7}$ \\\hline\hline
\multirow{5}{*}{15} & $h_1x_{126,14}$ & $d_{2}^{-1}$ & $x_{128,13,2}$ \\\cline{2-4}
 & $h_2x_{124,14}$ & & Permanent \\\cline{2-4}
 & $x_{127,15}$ & $d_{5}$ & $d_0x_{112,16}$ \\\cline{2-4}
 & $h_0^{14}h_7$ & $d_{2}$ & $h_0^{15}h_6^2$ \\\cline{2-4}
 & $h_6x_{64,14}$ & $d_{2}$ & $h_1^2x_{124,15}$ \\\hline\hline
\multirow{3}{*}{14} & $h_0g\Delta h_6g$ & $d_{2}^{-1}$ & $x_{128,12,2}$ \\\cline{2-4}
 & $h_0h_3x_{120,12}$ & $d_{2}^{-1}$ & $h_3x_{121,11}$ \\\cline{2-4}
 & $h_0^{13}h_7$ & $d_{2}$ & $h_0^{14}h_6^2$ \\\hline\hline
\multirow{5}{*}{13} & $h_0^3x_{127,10}$ & $d_{2}^{-1}$ & $h_0x_{128,10}$ \\\cline{2-4}
 & $g\Delta h_6g$ & $d_{3}^{-1}$ & $x_{128,10,2}$ \\\cline{2-4}
 & $h_3x_{120,12}$ & $d_{4}^{-1}$ & $h_2x_{125,8,2}$ \\\cline{2-4}
 & $x_{127,13}$ & $d_{3}$ & $h_0^2D_2x_{68,8}$ \\\cline{2-4}
 & $h_0^{12}h_7$ & $d_{2}$ & $h_0^{13}h_6^2$ \\\hline\hline
\multirow{3}{*}{12} & $h_0^2x_{127,10}$ & $d_{2}^{-1}$ & $x_{128,10}$ \\\cline{2-4}
 & $h_1x_{126,11}$ & $d_{3}^{-1}$ & $h_1x_{127,8}$ \\\cline{2-4}
 & $h_0^{11}h_7$ & $d_{2}$ & $h_0^{12}h_6^2$ \\\hline\hline
\multirow{4}{*}{11} & $h_0h_3x_{120,9}$ & $d_{3}^{-1}$ & $h_3D_2h_6$ \\\cline{2-4}
 & $h_0h_2x_{124,9}$ & & Permanent \\\cline{2-4}
 & $h_0x_{127,10}$ & & Permanent \\\cline{2-4}
 & $h_0^{10}h_7$ & $d_{2}$ & $h_0^{11}h_6^2$ \\\hline\hline
\multirow{6}{*}{10} & $h_1^2x_{125,8}$ & & Permanent \\\cline{2-4}
 & $h_2x_{124,9}+h_0^2x_{127,8}$ & $d_{6}$ & $h_1^2x_{124,14}$ \\\cline{2-4}
 & $h_0^2x_{127,8}$ & $d_{4}$ & $x_{126,14}$ \\\cline{2-4}
 & $x_{127,10}$ & $d_{2}$ & $d_1x_{94,8}$ \\\cline{2-4}
 & $h_3x_{120,9}$ & $d_{2}$ & $h_0x_{126,11}$ \\\cline{2-4}
 & $h_0^9h_7$ & $d_{2}$ & $h_0^{10}h_6^2$ \\\hline
  \end{tabular}}
\caption{The classical Adams spectral sequence of $S^0$ for $10 \le s \le 20$ in stem 127}
\label{Table:S127.20}
\end{table}

\begin{table}
    \centering
\scalebox{0.95}{\begin{tabular}{|c|l|l|l|}\hline %%%%%%%%%%%%%%%%%%%%%%
\multirow{5}{*}{9} & $h_0^2x_{127,7}$ & $d_{2}^{-1}$ & $h_0x_{128,6}$ \\\cline{2-4}
 & $h_1x_{126,8}$ & & Permanent \\\cline{2-4}
 & $h_1x_{126,8,2}$ & $d_{5}$ & $h_1h_3x_{118,12}$ \\\cline{2-4}
 & $h_0x_{127,8}$ & $d_{2}$ & $h_0^2x_{126,9}$ \\\cline{2-4}
 & $h_0^8h_7$ & $d_{2}$ & $h_0^9h_6^2$ \\\hline\hline
\multirow{7}{*}{8} & $h_0x_{127,7,2}+h_0x_{127,7}+h_0^2x_{127,6}$ & $d_{2}^{-1}$ & $x_{128,6}$ \\\cline{2-4}
 & $h_0^2x_{127,6}$ & $d_{2}^{-1}$ & $h_0x_{128,5}$ \\\cline{2-4}
 & $h_2h_6A$ & & Permanent \\\cline{2-4}
 & $h_2x_{124,7}$ & $d_{9}$ & $?$ \\\cline{2-4}
 & $x_{127,8}$ & $d_{2}$ & $h_0x_{126,9}$ \\\cline{2-4}
 & $h_0x_{127,7}$ & $d_{2}$ & $h_0^2x_{126,8}$ \\\cline{2-4}
 & $h_0^7h_7$ & $d_{2}$ & $h_0^8h_6^2$ \\\hline\hline
\multirow{5}{*}{7} & $h_0x_{127,6}$ & $d_{2}^{-1}$ & $x_{128,5}$ \\\cline{2-4}
 & $h_1x_{126,6}$ & $d_{10}$ & $?$ \\\cline{2-4}
 & $x_{127,7,2}+x_{127,7}$ & $d_{3}$ & $?$ \\\cline{2-4}
 & $x_{127,7}$ & $d_{2}$ & $h_0x_{126,8}$ \\\cline{2-4}
 & $h_0^6h_7$ & $d_{2}$ & $h_0^7h_6^2$ \\\hline\hline
\multirow{2}{*}{6} & $x_{127,6}$ & $d_{4}$ & $?$ \\\cline{2-4}
 & $h_0^5h_7$ & $d_{2}$ & $h_0^6h_6^2$ \\\hline\hline
\multirow{1}{*}{5} & $h_0^4h_7$ & $d_{2}$ & $h_0^5h_6^2$ \\\hline\hline
\multirow{1}{*}{4} & $h_0^3h_7$ & $d_{2}$ & $h_0^4h_6^2$ \\\hline\hline
\multirow{2}{*}{3} & $h_1h_6^2$ & $d_{14}$ & $?$ \\\cline{2-4}
 & $h_0^2h_7$ & $d_{2}$ & $h_0^3h_6^2$ \\\hline\hline
\multirow{1}{*}{2} & $h_0h_7$ & $d_{2}$ & $h_0^2h_6^2$ \\\hline\hline
\multirow{1}{*}{1} & $h_7$ & $d_{2}$ & $h_0h_6^2$ \\\hline\hline
\multirow{1}{*}{0} & \multicolumn{3}{c|}{}\\\hline
 \end{tabular}}
\caption{The classical Adams spectral sequence of $S^0$ for $s \le 9$ in stem 127}
\label{Table:S127.9}
\end{table}

\makebibliography

\hspace{1cm}\vspace{5pt}

% {\scshape Weinan Lin, School of Mathematical Sciences, Peking University, 5 Yi He Yuan Road, Haidian District, Beijing, 100871, China.}

% \emph{Email address}: \url{lwnpku@math.pku.edu.cn}

\end{document}